\documentclass[english,reqno,3p]{elsarticle}
\usepackage{amsmath}
\usepackage{amsthm,amsfonts,amssymb,bbm,mathtools}
\usepackage{tikz}
\usepackage{algpseudocode}
\usepackage[numbers]{natbib}
\usepackage{hyperref}
\usepackage[english]{babel}
\usepackage{framed}

\usepackage[norefs,nocites]{refcheck}

\makeatletter
\def\imod#1{\allowbreak\mkern10mu{\operator@font mod}\,\left(#1\right)}
\newcommand{\defeq}{\mathrel{:\mkern-0.25mu=}}
\newcommand{\eqdef}{\mathrel{=\mkern-0.25mu:}}
\makeatother

\newtheorem{thm}{Theorem}[section]	
\newtheorem{lemma}[thm]{Lemma}
\newtheorem{claim}[thm]{Claim}

\newtheorem{cor}[thm]{Corollary}
\newtheorem{conjecture}[thm]{Conjecture}

\theoremstyle{definition}
\newtheorem{rem}[thm]{Remark}
\newtheorem{definition}[thm]{Definition}
\newtheorem{fact}[thm]{Fact}
\newtheorem{remark}[thm]{Remark}

\newtheoremstyle{named}{}{}{\itshape}{}{\bfseries}{.}{.5em}{#1\thmnote{#3}}
\theoremstyle{named}
\newtheorem*{namedtheorem}{}

\numberwithin{equation}{section}

\def\N{\mathbb{N}}
\def\Z{\mathbb{Z}}
\def\Q{\mathbb{Q}}
\def\R{\mathbb{R}}
\def\P{\mathbb{P}}
\def\E{\mathbb{E}}
\def\F{\mathbb{F}}
\def\1{\mathbbm{1}}

\def\BB{\mathbb}

\def\X{\mathcal{X}}

\def\O{\mathbf{O}}
\def\o{\mathbf{o}}
\def\L{\mathcal{L}}
\def\S{\mathcal{S}}
\def\I{\mathbb{I}}
\def\U{\textsc{Uniform}}
\def\Li{\operatorname{Li}}

\def\ord{\operatorname{ord}}
\def\refp#1{\ref{#1}~(p.~\pageref{#1})}

\algnewcommand{\IIf}[1]{\State\algorithmicif\ #1\ \algorithmicthen}
\algnewcommand{\EndIIf}{\unskip\ \algorithmicend\ \algorithmicif}	

\let\originalleft\left
\let\originalright\right
\renewcommand{\left}{\mathopen{}\mathclose\bgroup\originalleft}
\renewcommand{\right}{\aftergroup\egroup\originalright}

\begin{document}
\title{Rigorous Analysis of a Randomised Number Field Sieve}

\author[Lee]{Jonathan D. Lee}
\address[Lee]{Mathematical Institute, University of Oxford, UK \& Microsoft Research Redmond}
\ead{jonathan.lee@merton.ox.ac.uk, jonatlee@microsoft.com}
\author[Venkatesan]{Ramarathnam Venkatesan}
\address[Venkatesan]{Microsoft Research India \& Redmond}
\ead{venkie@microsoft.com}

\renewcommand*{\today}{June 27, 2017}



\date{\today}

\begin{abstract}
Factorisation of integers $n$ is of number theoretic and cryptographic significance. The Number Field Sieve (NFS) introduced circa 1990, is still the state of the art algorithm, but no rigorous proof that it halts or generates relationships is known. We propose and analyse an explicitly randomised variant.
For each $n$, we show that these randomised variants of the NFS and Coppersmith's multiple polynomial sieve find congruences of squares in expected times matching the best-known heuristic estimates.
\end{abstract}

\maketitle

\setcounter{tocdepth}{1}
\tableofcontents

\section{Introduction}
For real numbers $a, b, x$, we write
\[
L_x\left(a, b\right) = \exp\left(b\left(\log x\right)^a\left(\log \log x\right)^{1-a}\right).
\]
To factor $n$, modern factoring algorithms first find a congruence of squares $x^2 = y^2 \imod{n}$, which is hopefully not trivial in the sense $x \neq \pm y \imod{n}$, and next compute $\gcd(x \pm y, n)$ to obtain factors of $n$. Hence the runtime analysis is devoted to the first part and studied actively~\cite{BuhlerLenstraPomerance, PomeranceICM, Tetali, BBM-squares, QS-Pomerance, dixon, Vallee, SSL,PomeranceSieve}, while the second part has been elusive and heuristic with the exception of variants of Dixon’s algorithm and the class group algorithm. In the subsequent, we introduce a randomised variant of the Number Field Sieve and provide an unconditional analysis on the first part, and provide evidence that the factors so obtained are non-trivial. In particular:
\begin{namedtheorem}[Theorems \refp{NFS-L13-analysis} and \refp{NFS-fruitfull-0-1/2}]
There is a randomised variant of the Number Field Sieve which for each $n$ finds congruences of squares $x^2 = y^2 \imod{n}$ in expected time:
\[
L_n\left(\frac{1}{3}, \sqrt[3]{\frac{64}{9}} + \o(1)\right) \simeq L_n\left(\frac{1}{3}, 1.92299\ldots + \o(1)\right).
\]
These congruences of squares are not trivially of the form $x = \pm y$: conditional on a mild character assumption (Conjecture~\refp{char-decor}), for $n$ the product of two primes congruent to $3$ mod $4$, the factors of $n$ may be recovered in the same asymptotic run time.
\end{namedtheorem}

We use a probabilistic technique, which we term \emph{stochastic deepening}, to avoid the need to show second moment bounds on the distribution of smooth numbers.
These results can be shown to extend to Coppersmith's multiple polynomial sieve of \cite{coppersmithMPS}, a randomised variant of which finds congruences of squares modulo $n$ in expected time:
\[
L_n\left(\frac{1}{3}, \sqrt[3]{\frac{92 + 26\sqrt{13}}{27}} + \o(1)\right) \simeq L_n\left(\frac{1}{3}, 1.90188\ldots + \o(1)\right).
\]

Part of the randomisation is similar to the polynomial selection algorithm of Kleinjung~\cite{Kleinjung}, which is popular in empirical studies, in that we add an $(X-m)R(X)$ to the field polynomial where $m$ is the root of that polynomial in $\Z/n\Z$. Kleinjung chooses $m$ and $R$ to minimise certain norms and improve smoothness, whilst our $R$ is random.

Integer factorisation is of fundamental importance both in algorithmic number theory and in cryptography. In the latter setting, it is especially important to have effective bounds on the run time of existing algorithms, as many existing systems depend on being able to produce integers whose factorisations will remain unknown for decades, even allowing for the rapid increases in the cost-effectiveness of computational hardware. For example, an understanding of the factoring of numbers $n$ with $\log_2 n \approx 4096$ is important in practice, while the public record for a factorisation of a general number stands at $\log_2 n \approx 768$. A uniform and effective bound will be useful in understanding the run time as $\log_2 n$ increases. While our methods apply to general composites, in applications there is particular interest in factoring \emph{semiprimes}, integers with two prime factors of nearly equal size, which are considered to be the most challenging type of integer to factor.

The Number Field Sieve (NFS) has been the state of the art algorithm for factorisation since its introduction nearly three decades ago~\cite{BuhlerLenstraPomerance}. Unfortunately, its analysis has been thus far entirely heuristic~\cite{PomeranceSieve}, with the claimed run time on an input $n$ of $L_n\left(\frac{1}{3}, \sqrt[3]{\frac{64}{9}} + \o(1)\right)$. This became of practical importance in the mid 1990s when it bettered the (also heuristic) $L_n\left(\frac{1}{2}, 1+\o(1)\right)$ run time of the previous champion Quadratic Sieve.

It is a priori unclear how to argue that the NFS even halts~\cite{LenstraPC}. Even assuming standard conjectures (e.g.; GRH), there is no analysis that any substantial part of the NFS will halt. In particular, the NFS and other algorithms critically depend on the existence of sufficient numbers of smooth elements among rational or algebraic integers on certain linear forms, which cannot be guaranteed in current algorithms. Similarly, in implementations the NFS cannot assure the reduction from smooth relations to a congruence of squares, because ideal factorisation is avoided in favour of Adleman's approach based on characters. Our explicit randomisation allows us to get around these problems by analysing the average case as opposed to the worst case, influenced by the recent works on distribution of smooths on arithmetic progressions~\cite{Soundararajan,G1,G2,Harper} and the philosophy that sums of arithmetic functions are essentially determined by the part over smooths~\cite{Gran-Sound-large,Zhang}. In short, we make essential use and strengthening of these tools as well as probabilistic combinatorics, and it may explain why no analysis was available earlier.

The fastest algorithms with known rigorous analysis are unfortunately much slower, with the best result being $L_n\left(\frac{1}{2}, 1+\o(1)\right)$~\cite{SSL}, where the basic operations are performed in the class group on quadratic forms; they also show that hazarding new conjectures that seem necessary to formally analyse run times can be risky, as they may formally contradict earlier natural conjectures. In this paper, we will present and analyse an explicitly randomised version of the NFS. We will show bounds on the expectation of the time taken to produce \emph{congruences of squares}
\[
(x, y) : x^2 \equiv y^2 \imod{n}.
\]
These bounds will be of form
\[
L_n\left(\frac{1}{3}, \boldsymbol{\Theta}(1)\right),
\]
and are the first time that bounds of this type have been obtained for any factorisation algorithm. To obtain sharper estimates for the $\Theta(1)$ term, we use randomness to remove dependence on second moment bounds for which proofs known to us use the Riemann Hypothesis. This is analogous to the situation between the Miller and Miller-Rabin primality tests. 

Historically, there has been a close link in the sieving aspects of integer factorisation and the discrete logarithm problems. The NFS, along with many other factorisation algorithms, has an identically named analogue for computing discrete logarithms. For the discrete logarithm in small characteristic, recent breakthrough results~\cite{Joux,barbulescu2014heuristic} have suggested that much faster algorithms exist. We will not touch on an analysis of this algorithm for the discrete logarithm in this paper.

We provide a conditional analysis of whether the congruences of squares will be \emph{fruitful}, that is whether they yield a non-trivial factorisation of $n$. In the specific case that $n = pq$ is semiprime with $p \equiv q \equiv 3 \imod{4}$, and modulo a character decorrelation conjecture, we are able to show that the factors are non-trivial with probability $1/2$. As the conjecture may indicate, the analysis of this fruitfulness seems involved and likely to require methods that are substantially different from the initial analysis of relationship formation. For example, the analysis of Pollard's Rho algorithm for the discrete logarithm, the run time for forming relationships was shown to be $\sqrt{p}\log^3 p$ \cite{MillerVenkatesan} using characters and quadratic forms; this was later improved to be optimal up to constant factors by Kim, Montenegro, Peres and Tetali \cite{KimMontenegraoPeresTetali} using combinatorial methods. However, the known proof that the relations are fruitful \cite{MillerVenkatesan2} still uses analytic methods with a substantially more complex analysis. For the Number Field Sieve we expect that the analysis will be even more arduous.

\subsection{Combinations of Congruences}\leavevmode

All modern factoring algorithms have core similarities, and are referred to as \emph{combinations of congruences} algorithms to draw attention to this fact. To present the main ideas involved, we will discuss Dixon's random squares algorithm. The first observation, due to Fermat, is that
\[
x^2 \equiv y^2 \imod{n} \Leftrightarrow n \mid x^2 - y^2 = (x+y)(x-y)
\]
and if we are lucky we may find that $n$ does not divide $x\pm y$; in this case $\gcd(n, x+y)$ is a non-trivial factor of $n$. We can generate pairs
\[
(x_i, z_i) : x_i^2 \equiv z_i \imod{n},
\]
by choosing $x_i$ at random to be an integer in $[n]$ and setting $z_i = x_i^2 \imod{n}$. Then to produce a pair $(x, y)$ it suffices to find a subset $S$ of the $z_i$ whose product is a square in $\Z$. We note that even the problem of finding how large a random subset of $[n]$ must be to contain a subset $S$ whose product is a square is of substantial independent interest~\cite{PomeranceICM,Tetali}.
The main step is to search for $B$-\emph{smooth} $z_i$:
\[
z_i : p \text{ prime}, p | z_i \Rightarrow p < B \qquad\Leftrightarrow\qquad
z_i = \prod_{p_j < B}p_j^{e_{i,j}}, \; p_j\text{ prime}.
\]
If all the $z_i$ are $B$-smooth then a product $z_i^{s_i}$ is square if and only if
\[
\forall j,\; 2 | \sum_i s_i e_{i,j} \quad\Leftrightarrow\quad s\in \ker_{\F_2}(E)
\]
where $E=(e_{i,j})$ and $s$ is a column vector of $s_i$, which can be found by standard means whenever it exists.
This calculation with indices $e_{i,j}$ is what gives this class of algorithm its name.
Once a relationship $
x^2 \equiv y^2 \imod{n}
$
has been found, we compute $\gcd(x\pm y, n)$ and hope that at least one is a non-trivial factor of $n$; in this case we say that the congruence is \emph{fruitful}.

Hence to have a functional algorithm it suffices to have methods for finding $B$-smooth values of the $z_i$. Analysis of the run time additionally requires some estimate of the probability that $z_i$ is $B$-smooth. At a high-level, we can see that as $B$ is increased, the density of $B$-smooth integers increases, whilst the number of $B$-smooth $z_i$ we will need to find to guarantee that a vector $s_i$ exists will also increase. These two effects are balanced when $B = L_n\left(\frac{1}{2}, \boldsymbol{\Theta}(1)\right)$. For Dixon's algorithm, this choice results in a run time of $L_n\left(\frac{1}{2}, 2\sqrt{2} + \o(1)\right)$ (see Corollary~\refp{CEP-cor} with $b=1$ and $a= \frac{1}{2}$).

Various modifications of this core algorithm exist. One line of modifications is to keep track of $z_i$ which are \emph{almost} $B$-smooth, in the sense of having few factors which are too large, hoping to combine them later to find $B$-smooth numbers lying under a square in $\Z$~\cite{LargePrime}. Another approach is to attempt to make the numbers $z_i$ smaller, since heuristically the density of $B$-smooth numbers is decreasing in $|z_i|$. This is the core idea in Vall\'ee's algorithm \cite{Vallee}, which can be rigorously shown to have a run time of $L_n\left(\frac{1}{2}, \sqrt{4/3} + \o(1)\right)$. The Quadratic Sieve reduces the size of the $z_i$ to be $n^{\frac{1}{2} + \o(1)}$ by choosing $x_i \simeq \sqrt{n}$, and achieves a heuristic run time of $L_n\left(\frac{1}{2}, 1+\o(1)\right)$, though its analysis seems out of reach.

Observe that in all of these algorithms, we use combinations of congruences to find a product of the $z_i$ which is a square $y^2$, but ensure that the associated product of $x_i^2$ is a square by ensuring that each relation $x_i^2 \equiv z_i \imod{n}$ has a square on the left-hand side. Further gains are made by relaxing this condition, so that we find both $x$ and $y$ as a result of combining congruences. For example, the Schorr-Seysen-Lenstra algorithm~\cite{SSL} shifts its attention from square integers to quadratic forms with one coefficient smooth and of discriminant $-dn$ for small values of $d$, and is able to achieve an expected run time of $L_n\left(\frac{1}{2}, 1+\o(1)\right)$.

\subsection{The Number Field Sieve}\leavevmode

In the NFS, we instead observe that there are rings other than $\Z$ lying over $\Z/n\Z$. In particular, if we are given a monic polynomial $f$ with a root modulo $n$ at some integer $m$, we can form the following commuting diagram:
\begin{figure}[!htb]
\centering
\begin{tikzpicture}[scale=1.5]
\node (A) at (0.7,1.4) {$\Z[X]$};
\node [right] (B) at (1.2,0.7) {$\Z[X] / \left(f\right)$};
\node [left] (C) at (0.2,0.7) {$\Z$};
\node (D) at (0.7,0) {$\Z/n\Z$};
\node (F) at (0.7,0.7) {$\circ$};
\path[->,font=\scriptsize]
(A) edge node[right]{$\imod{f}$} (B.north west)
(A) edge node[left]{$x \rightarrow m$} (C.north east)
(B.south west) edge node[right]{$X \rightarrow m, \imod{n}$} (D)
(C.south east) edge node[left]{$\imod{n}$} (D);
\end{tikzpicture}
\end{figure}

If $f$ is reducible, we may directly extract factors of $n$, and so we may assume $f$ is irreducible. We can identify values of $\Z/n\Z$ which are squares mod $n$ either by virtue of each of them lying under a square in $\Z$ or a square in $\Z[X]/(f)$. The second ring is then a subset of the ring of integers of the number field $\Q[X]/(f)$; on the ring of integers we have a notion of divisibility into \emph{prime ideals}, a notion of size via the \emph{field norm}, and thus we can define a natural analogue of $B$-smoothness.

In the NFS, we choose linear polynomials in $\Z[X]$ with coefficients of size at most $L_n\left(\frac{1}{3}, \boldsymbol{O}(1)\right)$, and project them into both $\Z$ and $\Z[X]/(f)$. We then hope to find many polynomials such that both projections are $B$-smooth. Then we use linear algebra to find a subset whose product is square in $\Z$ and also square in $\Z[X]/(f)$. Then we take square roots in both rings, and project the roots down to $\Z/n\Z$ to produce a congruence of squares.

In practice, the NFS is rather more complex, as factorisation in the ring of integers of $\Q[X]/(f)$ is complicated to work with. Substantial extra bookkeeping needs to be done with characters of large conductor on the number field to guarantee that the square we find has a root in $\Z[X]/(f)$ rather than in the larger ring of integers. However, the gains are substantial. With optimal choice of parameters, both $m$ and the values that we need to be smooth are of size $L_n\left(\frac{2}{3}, \boldsymbol{O}(1)\right)$. Assuming that all the numbers behave as independent uniformly random integers of this size and optimising $B$ yields a run time of
$
L_n\left(\frac{1}{3}, \boldsymbol{\Theta}(1)\right),
$
which is much smaller asymptotically than the other algorithms provide. In practice, the NFS is the fastest known algorithm for factoring numbers in excess of 100 digits.

In the NFS as usually implemented, there is a fixed choice of the polynomial $f$ for each $m$. Additionally, the additional bookkeeping needed on the number field side is standardised. Both of these choices make the NFS very rigid, and a proper analysis would seem to require precisely understanding the distribution of smooth numbers on curves of high degree. Our modification, the \emph{Randomised NFS}, carefully randomises the coefficients of $f$, and chooses the extra bookkeeping characters stochastically. This allows us to reduce the required analysis to an understanding of the average distribution of smooth numbers along arithmetic progressions.

\section{Our Results}\leavevmode

We introduce and analyse a variant we call the ``Randomised NFS'', which provides more
easily controllable behaviour on average.
We heavily use a combination of methods of probabilistic combinatorics and
analytic aspects of number theory, touching on a range of topics.
%
\begin{thm} \label{NFS-L13-analysis}
For \emph{any} $n$, the Randomised NFS runs in \emph{expected} time:
\[
L_n\left(\frac{1}{3},\sqrt[3]{\frac{64}{9}}+\o\left(1\right)\right),
\]
and produces a pair $x,y$ with $x^2=y^2 \mod n$.
\end{thm}



\begin{rem}
We note the importance of the algorithm under discussion being a variant of the NFS. Whilst it is trivial to generate pairs $(x,y)$ such that $x^2 = y^2 \imod{n}$ by taking $x = \pm y \imod{n}$, it is non-trivial to find sub-exponential algorithms that could in principle produce a congruence of squares where $x \neq \pm y \imod{n}$. As in the standard NFS, the entirety of the run time is devoted to finding congruences of squares, as the recovery of a (potentially trivial) factor of $n$ amounts to a trivial gcd calculation. By convention, these algorithms are run repeatedly until a non-trivial factor is found, using independent internal coinflips on each run. The general belief is that NFS type algorithms will not always output trivial factors of $n$ (see Remark~\ref{NFS-conj-rem}), and hereafter we refer to the dominant computation as finding the \emph{congruence} without further comment.
\end{rem}

We also present a partial result on the fruitfulness of the congruences.
\begin{thm}\label{NFS-fruitfull-0-1/2}
For a fixed $n$ semiprime with both prime factors congruent to $3\imod{4}$, conditional on Conjecture~\refp{char-decor} the Randomised NFS finds a pair $x,y$ such that $x^2=y^2 \imod{n}$ and $x \neq \pm y \imod{n}$ in expected time $L_n\left(\frac{1}{3}, \sqrt[3]{\frac{64}{9}} + \o(1)\right)$.
\end{thm}
\begin{rem}\label{NFS-conj-rem}
In this case the Randomised NFS is a probabilistic algorithm for factorisation in the style of the Miller-Rabin or Solovay-Strassen primality tests. If Conjecture~\refp{char-decor} fails to hold for a given $n$ and $f$, then \emph{any} congruences of squares found by inspection of $\Z[\alpha]$ and $\Z$ would be trivial. We note that since the NFS has been successfully run to found factors of numbers of this form, the conjecture is not false in general.
\end{rem}

Our analysis splits along the same lines as the internal structure of NFS-type algorithms. We will first study how many smooth relationships exist and prove the following theorem:

\begin{thm}\label{NFS-smoothness-base}
Take $\delta, \kappa, \sigma, \beta, \beta'$ subject to the conditions of Equation~\refp{const-bounds} and~\refp{d-bound}.
For any $n$, the Randomised NFS can almost surely find an irreducible polynomial $f$ of degree $d = \delta\sqrt[3]{\log n / \log \log n}$ and height at most $L_n\left(\frac{2}{3}, \kappa\right)$, with $\alpha$ a root of $f$, $n | f(m)$, and
\[
L_n\left(\frac{1}{3}, \max(\beta, \beta') + \o\left(1\right)\right)
\]
distinct pairs $a < |b| \leq L_n\left(1/3, \sigma\right)$ such that $\left(a-bm\right)$ is $L_n\left(1/3, \beta\right)$-smooth and $\left(a - b\alpha\right)$ is $L_n\left(1/3, \beta'\right)$-smooth, in expected time at most $L_n\left(1/3, \lambda\right)$ for any
\[
\lambda > \max\left(\beta, \beta'\right) + \frac{\delta^{-1}\left(1+\o\left(1\right)\right)}{3\beta} + \frac{\left(\sigma\delta + \kappa\right)\left(1 + \o\left(1\right)\right)}{3\beta'}.
\]
In particular, the probability that the Randomised NFS fails to produce such a set is bounded above by $L_n\left(\frac{2}{3}, \kappa - \delta^{-1}\right)^{-1+\o(1)}$.
\end{thm}

We also show that we can reduce a collection of smooth relationships to a congruence of squares.

\begin{thm}\label{NFS-characters-uncond}
Let $B = L_n\left(\frac{1}{3}, \beta\right)$ and $B' = L_n\left(\frac{1}{3}, \beta'\right)$. Let $f$ be irreducible of degree $d = \delta\sqrt[3]{\log n / \log \log n}$ and height at most $L_n\left(\frac{2}{3}, \kappa\right)$, and let $\alpha$ be a root of $f$. Then for all but a $L_n\left(\frac{2}{3}, \frac{\kappa-\delta^{-1}}{2}\left(1+\o\left(1\right)\right)\right)^{-1}$ fraction of the set of $f$, if we are given
\[
L_n\left(\frac{1}{3}, \max\left(\beta, \beta'\right)\right)\boldsymbol{\Omega}\left(\log \log n\right)
\]
pairs $a< b \leq L_n\left(\frac{1}{3}\right)$ such that $a-mb$ is $B$-smooth and $a-b\alpha$ is $B'$-smooth, we can find a congruence of squares modulo $n$ in expected time at most
\[
L_n\left(\frac{1}{3}, 2\max\left(\frac{2\delta}{3}, \beta, \beta'\right)\right)^{1+\o\left(1\right)}.
\]
\end{thm}

\begin{rem}\label{MPNFS-thm-mod}
In the case of Coppersmith's \textbf{multiple polynomial Number Field Sieve}~\cite{coppersmithMPS}, we instead have to find a single $m$ and $L_n\left(\frac{1}{3}, \eta\right)$ irreducible polynomials $f^{(i)}$ such that $f^{(i)}(m) = n$, and a collection of $L_n\left(\frac{1}{3}, \max(\beta, \beta' + \eta)\right)$ pairs $(a, b)$ such that $a-mb$ is $B$-smooth and some $f^{(i)}(a, b)$ is $B'$-smooth. In this case the second constraint of equation~\refp{const-bounds} is replaced by $
2\sigma + \eta > \max\left(\beta, \beta' + \eta\right) + \frac{\delta^{-1}\left(1+\o\left(1\right)\right)}{3\beta} + \frac{\left(\sigma\delta + \kappa\right)\left(1 + \o\left(1\right)\right)}{3\beta'}
$
and $\lambda > \max\left(\beta, \beta' + \eta\right) + \frac{\delta^{-1}\left(1+\o\left(1\right)\right)}{3\beta} + \frac{\left(\sigma\delta + \kappa\right)\left(1 + \o\left(1\right)\right)}{3\beta'}$.
The reduction to a congruence of squares similarly has $\beta'$ replaced by $\beta' + \eta$ throughout.
\end{rem}

\section{Preliminaries}

\subsection{Notation and Definitions}\leavevmode
\begin{definition}
For any finite set $S$, we denote the uniform measure over $S$ by $\U\left(S\right)$.
\end{definition}
\begin{definition}
For any two measures $\mu, \nu$ over an additive group $G$, we define their \emph{convolution} to be:
\[
\left(\mu \star \nu\right)\left(x\right) \defeq \sum_{y \in G} \mu\left(y\right)\nu\left(x-y\right).
\]
\end{definition}
\begin{definition}\label{centredinterval}
We define the \emph{centred interval} of length $L$ in $\Z$ to be
\[
\I\left(L\right) \defeq\left[-\frac{1}{2}L, \frac{1}{2}L\right) \cap \Z.
\]
\end{definition}

We now turn to a collection of classical number theoretic results:

\begin{definition}The \emph{prime counting functions} are given by
\begin{align*}
\pi\left(x\right) &\defeq |\{y < x : y \in \N, y\textrm{ prime}\}| \\
\pi_{r.s}\left(x\right) &\defeq |\{y < x : y \in \N, y\textrm{ prime}, y = s \imod{r}\}|,
\end{align*}
\end{definition}
\begin{definition}The \emph{logarithmic integral} $\Li\left(x\right)$ is given by
\[
\Li\left(x\right) \defeq \int_2^x \frac{dt}{\log t} = \frac{x}{\log x}\left(1 + \O\left(\frac{1}{\log x} \right)\right).
\]
\end{definition}
\begin{fact}[The Prime Number Theorem]
There is a constant $a > 0$ such that:
\begin{align*}
\pi\left(x\right) &= \Li\left(x\right) + \O\left(\frac{x}{\log x}\exp\left(-a\sqrt{\log x}\right)\right)
= \frac{x}{\log x} \left(1 + \o\left(1\right)\right)
\end{align*}
\end{fact}
\begin{definition} We say $n \in \N$ is a \emph{semiprime} if
$n = pq$, with $p, q$ distinct primes.
\end{definition}

\begin{definition}
We define a family of functions $L_n\left(a,c\right): \N \rightarrow \R^{+}$ by:
\[
L_n\left(a, c\right) = \exp\left(c \left(\log n\right)^{a}\left(\log \log n\right)^{1-a}\right).
\]
We note that $a, c$ may be functions of $n$. In our applications, $a\left(n\right)$ will always tend to a constant and $c\left(n\right) = \left(\log \log n\right)^{\boldsymbol{o}\left(1\right)}$, and we will say:
\[
f\left(n\right) = L_n\left(a\right) \Leftrightarrow \exists c\left(n\right) = \left(\log \log n\right)^{\boldsymbol{o}\left(1\right)}\textrm{ s.t. }f\left(n\right) = L_n\left(a, c\right),
\]

We will often perform \emph{arithmetic} directly with these functions. We note in particular that:
\[
L_n\left(a,c\right) L_n\left(a,c'\right) = L_n\left(a, c + c'\right)
\]
and for $d = \delta \left(\frac{\log n}{\log \log n}\right)^\epsilon$, with $\delta = \left(\log \log n\right)^{\boldsymbol{O}\left(1\right)}$, $\epsilon = \boldsymbol{O}(1)$ as functions of $n$:
\[
L_n\left(a,c\right)^d = L_n\left(a+\epsilon, c\delta\right)
\]
\end{definition}

\begin{rem} We note that our definition coincides with the standard definition of $L_n\left(a,c\right)$ when $a$ is taken to be a \emph{constant} function of $n$ and $c$ tends to some finite limit. Throughout, we will mention $\o\left(1\right)$ terms for the exponent $c$ explicitly in our notation.
\end{rem}

\begin{definition}
For $y \in \N$, we say $x \in \N$ is $y$-smooth if
$
p\textrm{ prime} \wedge p \mid x \Rightarrow p < y.
$

For any $x,y,r,a \in \N$ and $\chi$ a multiplicative character, we define:
\begin{align*}
\Psi\left(x, y\right) &\defeq \left|\{z \in \N : z < x, z \textrm{ is }y\textrm{-smooth}\}\right|\\
\Psi_r\left(x, y\right) &\defeq \left|\{z \in \N : z < x, z \textrm{ is }y\textrm{-smooth}, \left(z,r\right) = 1\}\right|\\
\Psi\left(x, y; r, a\right) &\defeq \left|\{z \in \N : z < x, z \textrm{ is }y\textrm{-smooth}, z \equiv a \imod{r}\}\right| \\
\Psi\left(x, y; \chi\right) &\defeq \sum_{z < x}\1_{\{z' : z'\textrm{ is }y\textrm{-smooth}\}}\left(z\right) \chi\left(z\right)\\
\varrho\left(x, y\right) &\defeq \Psi\left(x,y\right)x^{-1}
\end{align*}
\end{definition}
\begin{fact}[Canfield, Erd\H{o}s and {Pomerance~\cite[Corollary pp.15]{CanfieldErdosPomerance}}]
Let $\epsilon > 0$ be arbitrary and let $3 \leq u \leq (1-\epsilon) \frac{\log x}{\log \log x}$. Then:
\begin{equation*}
\Psi\left(x, x^{1/u}\right) = x \exp\left(-u\left(\log u + \log \log u  - 1 	+ \frac{\log \log u - 1}{\log u} \right.\right.
\left.\left.+ \O_{\epsilon}\left(\frac{\log
\log^2 u}{\log ^2 u}\right)\right)\right)
\end{equation*}
\end{fact}
\begin{cor}\label{CEP-cor}
Fix $0 < a < b \leq 1$. Then uniformly in $c,d > 0$:
\[
\varrho\left(L_x(b, d), L_x\left(a, c\right)\right) = L_x\left(b-a, \frac{d\left(b-a\right)}{c}\right)^{-1+\o\left(1\right)}.
\]
\end{cor}
\begin{proof}
Define
\[
u = \frac{\log L_x(b, d)}{\log L_x\left(a, c\right)} = \frac{d}{c} \left(\frac{\log x}{\log \log x}\right)^{b-a}
\]
Then $u \rightarrow \infty$ and $u = \o\left(\frac{\log x}{\log \log x}\right)$. Hence:
\begin{align*}
\varrho\left(L_x(b, d), L_x\left(a, c\right)\right) &= \exp(-(1+\o(1))u \log u) \\
&= \exp\left(-(1+\o(1))\frac{d(b-a)}{c}\log^{b-a} x (\log \log x)^{1-(b-a)}\right)\\
& = L_x\left(b-a,\frac{d(b-a)}{c}\right)^{-1+\o(1)} \qedhere
\end{align*}

\end{proof}
In the sequel, we will mainly take $b = \frac{2}{3}$ and $a = \frac{1}{3}$ in the above corollary, so that the probability of an $L_n(\frac{2}{3})$ sized number being $L_n(\frac{1}{3})$-smooth is $L_n(\frac{1}{3})^{-1}$.

Being substantially more careful allows short intervals of integers to be effectively sieved for smooth numbers, 
yielding for example:
\begin{fact}[{Hildebrand~\cite[Theorems 3 and 1]{HildebrandInterval}}]\label{H-interval}
Fix any $\epsilon > 0$. For any $x\geq 3, \log x \geq \log y \geq (\log \log x)^{5/3 + \epsilon}, 1 \leq z \leq y^{5/12}$, the following estimate which holds uniformly:
\[
\Psi\left(x\left(1+z^{-1}\right), y\right) - \Psi(x,y) = \frac{\Psi(x, y)}{z}\left(1 + \O\left(\frac{\log (u+1)}{\log y}\right)\right).
\]
\end{fact}
\begin{rem}
We note that Theorem 3 of~\cite{HildebrandInterval} provides a short interval estimate in terms of the Dickman function. Theorem 1 of~\cite{HildebrandInterval} allows us to replace this with $\Psi(x, y)$ over the same range and with multiplicative errors of the same order.
\end{rem}
\begin{fact}[Hildebrand and {Tenenbaum~\cite[Theorem 3]{HildebrandTenebaumSaddle}}]\label{HT-smooths}
For any $x,y$, we set $u \defeq \frac{\log x}{\log y}$. There exists an $\alpha = \alpha\left(x,y\right)$, the so-called \emph{saddlepoint}, such that for any $1 \leq c \leq y$:
\begin{align*}
\Psi\left(cx, y\right) &= \Psi\left(x,y\right)c^{\alpha\left(x,y\right)}\left(1+\O\left(\frac{1}{u}+\frac{\log y}{y}\right)\right),\text{ with} \\
\alpha\left(x,y\right) &= \frac{\log\left(\frac{y}{\log x} + 1\right)}{\log y} \left(1+ \O\left(\frac{\log \log \left(y+1\right)}{\log y}\right)\right) \end{align*}
\end{fact}

\begin{fact}[{Tenenbaum~\cite[Main Theorem]{Tenenbaum}}] Take $c > 0$ an arbitrary constant. Denote the number of prime factors (without multiplicity) of $q$ by $\omega(q)$. Let $q$ be $y$-smooth, $2 \leq y \leq x$ and with $\omega(q) \leq y^{c \left(\log (1+u)\right)^{-1}}$. Then:
\[
\Psi_q(x,y) = \frac{\phi(q)}{q}\Psi(x,y)\left(1 + \O_c\left(\frac{\log(1+u)\log(1+\omega(q))}{\log y}\right)\right)
\]
\end{fact}
We record the following easy corollary as observed by Tenenbaum:
\begin{cor}\label{psi-q-psi}
Take $c > 0$ an arbitrary constant, and retain $\omega$ as above. Let $2 \leq y \leq x$ and with $\omega(q) \leq y^{c \left(\log (1+u)\right)^{-1}}$. Then:
\[
\Psi_q(x,y) = \frac{\phi(q)}{q}\Psi(x,y)\left(1 + \O_c\left(\frac{\log(1+u)\log(1+\omega(q))}{\log y}\right)\right)\left(1 + \O\left(\frac{\omega(q)}{y}\right)\right)
\]
\end{cor}
\begin{proof}
Let $q = sr$ for $s$ a $y$-smooth integer and $r$ with no prime factor less than $y$. Then $\Psi_s(x, y) = \Psi_r(x, y)$, $\phi(q)= \phi(r)\phi(s)$ and $\phi(r)r^{-1} = \prod_{\text{prime }p | r}(1 - p^{-1}) = 1 + \O\left(\omega(q)y^{-1}\right)$ which implies the given bound.
\end{proof}

As mentioned earlier, a key ingredient in combination of congruence algorithms is the detection and factorisation of $y$-smooth numbers. The main difficulty here is that the algorithm must be polynomial time in the logarithm of the integer it is to factor, although it is permitted to be merely sub-exponential in the logarithm of the smoothness bound. That such an algorithm exists is by no means guaranteed.

Typically, the algorithm used here will be the Elliptic Curve Method, due to Lenstra~\cite{Lenstra-ECM}.
For technical reasons, we instead use the somewhat more complex Hyperelliptic Curve Method, which works on the Jacobian of a hyperelliptic curve in place of an elliptic curve.
\begin{fact}[Lenstra, Pila and {Pomerance~\cite[Theorem 1.1]{LenstraPilaPomerance}}]
There exists a constant $c$ such that the hyperelliptic curve method finds a non-trivial factor of any $x$ which has a prime factor less than $y$ in expected time bounded by
$
L_y\left(\frac{2}{3}, c\right)\left(\log x\right)^2
$
\end{fact}
\begin{cor}\label{HECM-speed}
Suppose $y = \log^{\boldsymbol{\omega}\left(1\right)} x$. Then the hyperelliptic curve method can factor any $y$-smooth number below $x$ in expected time at most
$
L_y\left(\frac{2}{3}, c\right)\left(\log x\right)^3 = y^{\o\left(1\right)}
$
\end{cor}
\begin{rem}
Both the ECM and HECM are successful if the order of the Jacobian of the randomly chosen curve is smooth.
In the HECM case, the Hasse-Weil interval is of the form $[x-4x^{3/4}, x+4x^{3/4}]$, and the density of smooth numbers in such intervals is unconditionally understood.
\end{rem}

\subsection{Overview of NFS algorithms}\leavevmode

We now provide a detailed look at the function of the NFS and the Randomised NFS. From a \emph{number theoretic} perspective, we fix some $\alpha \in \BB{C}$ with minimal polynomial $f$ over $\Z$, with leading coefficient $f_d$, such that $f\left(m\right) \equiv 0 \imod{n}$. Hence in particular $f_d\alpha$ is an algebraic integer. We will summarise the following discussion with the following diagram:
\begin{figure}[!htb]
\centering
\begin{tikzpicture}[scale=1.5]
\node (A) at (0.7,1.4) {$\Z[X]$};
\node [right] (B) at (1.2,0.7) {$\Z[X] / \left(f\right) \simeq \Z[\alpha]$};
\node [left] (C) at (0.2,0.7) {$\Z$};
\node (D) at (0.7,0) {$\Z/n\Z$};
\node (F) at (0.7,0.7) {$\circ$};
\node [right] (G) at (4.0,0.7) {$\left\{\frac{r}{f_d^s} :r, s \in \N\right\}\subset \Q$};
\node [right] (H) at (4.0,0) {$\F_2^{\textrm{poly}(\log n)}$};
\path[->,font=\scriptsize]
(A) edge node[right]{$\imod{f}$} (B.north west)
(A) edge node[left]{$x \rightarrow m$} (C.north east)
(B.south west) edge node[right]{$X \rightarrow m, \imod{n}$} (D)
(C.south east) edge node[left]{$\imod{n}$} (D)
(B.east) edge node[above]{$\mathbf{N}$} (G)
(B.south east) edge node[below]{$\chi_{\mathfrak{p}_i}$}(H.west);
\end{tikzpicture}
\caption{Algebra underlying the Number Field Sieve.}
\end{figure}
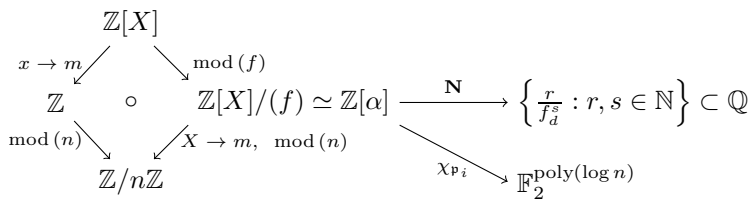

The key to finding a congruence of squares in $\Z/n\Z$ is to suppose that some $P \in \Z[X]$ projects to two squares, say $u^2 \in \Z[X] / \left(f\right)$ and $v^2 \in \Z$. Since the diagram commutes $u\left(m\right)^2 \equiv v^2 \imod{n}$, and so we have found a congruence of squares. We will find the squares in $\Z$ and $\Z[\alpha]$ by combining congruences, and so first we present a notion of smoothness for both rings.

We observe that both $\Z$ and $\Z[\alpha]$ have norms, given by the absolute value and the field norm $\mathbf{N}$ respectively. Recall that the field norm $\mathbf{N}$ is given by the product of all projections of the number field into $\mathbb{C}$. In particular, $f_d\mathbf{N}(a-b\alpha) = b^d f(a/b)$ for $d$ the degree of $f$. In general on the ring of integers $\mathcal{O}_{\Q(\alpha)}$ the norm is integral. Hence we say that an element of $\mathcal{O}_{\Q(\alpha)}$ is \emph{smooth} if its norm is smooth in $\Z$, and say that the linear polynomial is smooth in $\Z[X]/(f)$ if $b^d f(a/b)$ is smooth. Now, if an element of $\Z$ is smooth it can be factored into primes of small norm.

In $\Z[\alpha]$ this is not so straightforward. First, we may have $\Z[\alpha] \not\subseteq\mathcal{O}_{\Q(\alpha)}$, as $\alpha$ need not be an algebraic integer. Hence the norm is a rational whose denominator is a power of $f_d$, or more formally the direct limit $\underrightarrow{\lim}\{f_d^{-i}\Z : i \in \N\}$. In the case where $f$ is monic this is simply the integers. Note that $f_d(a - b\alpha) \in \Z[f_d\alpha]$, so that $N(a - b\alpha) \in \frac{1}{f_d^{d-1}}\Z$. More generally $\Z[f_d\alpha] \subseteq\mathcal{O}_{\Q(\alpha)}$. Second, prime ideals in $\Q(\alpha)$ are not necessarily contained in $\Z[\alpha]$, but instead in $\mathcal{O}_{\Q\left(\alpha\right)}$. Hence we cannot deduce that an element is a square of an element of $\Z[f_d\alpha]$ from the multiplicity of each prime dividing it being even. Finally, we cannot guarantee that the ring of integers $\mathcal{O}_{\Q\left(\alpha\right)}$ is a unique factorisation domain, and so irreducible factors need not be prime.

To address these difficulties in $\Z[\alpha]$, we only partially control the factorisation into ideals, and introduce a collection of additional multiplicative characters $\chi_{\mathfrak{p}_i}$ on $\Z[\alpha]$. We will be able to guarantee that if these characters all evaluate to $1$ on a subset product, then it is a square with reasonable probability; in particular the quotient group formed by taking these pseudo-squares and quotienting by the squares is of logarithmic size, and so we can guarantee that with only a small number of relationships we can find a pair whose product in $\mathcal{O}_{\Q\left(\alpha\right)}$ is in fact square. To ensure that the root is in fact in $\Z[\alpha]$, we then multiply throughout by an additional, carefully chosen square polynomial. In $\Z$, the additional factors of $f_d$ that have been introduced are controlled by insisting that a product of an even number of relationships is taken.

Hence we will search for $P$ by finding linear factors which induce smooth elements of $\Z$ and $\Z[\alpha]$, and then multiply these factors to obtain a suitable $P$. Since square roots in $\Z[f_d\alpha]$ and in $\Z$ can be taken in time polynomial in the degree and the logarithm of the coefficients, this allows us to recover the polynomial $u$ and the integer $v$, and thus a congruence of squares.

The above discussion holds for both the NFS (as observed in detail in \cite{BuhlerLenstraPomerance}) and the Randomised NFS, but thus far we have not shown how we choose the parameters of the algorithm. As previously noted, the difference between the two algorithms lies entirely in the process by which $f$ and the characters $\chi_{\mathfrak{p}_i}$ are chosen.

Computationally, the algorithm proceeds as follows. We choose a degree
\[
d \sim \sqrt[3]{\frac{\log n}{\log \log n}}.
\]
In the Randomised NFS, we will additionally insist that $d$ is odd, whilst the NFS does not make any such insistence. We then choose an $m$ such that:
\[
m^d \leq n < 2m^d.
\]
As a corollary, $m = L_n\left(\frac{2}{3}\right)$. We then choose an irreducible polynomial $f$ such that $n | f(m)$. We define a polynomial $\hat{f}_{n,m}$ by expressing $n$ in base $m$, taking the coefficients of the resulting expression and using them as the coefficients of $\hat{f}_{n,m}$. Note that by construction, $\hat{f}_{n,m}$ is monic of degree $d$.

In the NFS, we take $f = \hat{f}_{n,m}$.
In the Randomised NFS, we will generally homogenise $f$ for notational convenience, writing $f(x, y) = y^df\left(\frac{x}{y}\right)$ and set:
\begin{equation}\label{first-f-def}
f\left(x,y\right) \defeq \hat{f}\left(x,y\right) + R\left(x, y\right),\qquad R\left(x,y\right)=\sum_{i=0}^{d-1}c_{i}\left(x-ym\right)y^{i}x^{d-i-1}
\end{equation}
where $c_{i}$ are uniform and independently chosen with $|c_{i}| \leq L_{n}\left(\frac{2}{3}\right)$. The key purpose of this randomisation is to cause the norm of the image of $a-bX$ to become a random variable in $\Z$. This allows us to show that in the Randomised NFS, for any fixed linear polynomial we consider, the smoothness of the images in $\Z$ and $\Z[\alpha]$ are independent.

In both the NFS and the Randomised NFS, we will search through linear terms $a-bX$ with $|a|, |b| \leq L_n\left(\frac{1}{3}\right)$. We observe that $\mathbf{N}(a-b\alpha) = f(a,b)f_d^{-1}$. Since $f$ is of degree $d$ and has coefficients which are at most of size $L_n\left(\frac{2}{3}\right)$, both $f_d\mathbf{N}(a-b\alpha)$ and $a-bm$ are integers of size $L_n\left(\frac{2}{3}\right)$. Hence in both the NFS and the Randomised NFS, we take the smoothness bound $B$ to be $L_n\left(\frac{1}{3}\right)$ so that heuristically the likelihood that both numbers are $B$-smooth is $L_n\left(\frac{1}{3}\right)^{-1}$.

The remaining ambiguity is in the selection of characters $\chi_{\mathfrak{p}_i}$. In the NFS these are canonically taken to be quadratic characters induced
by finding a map from $\Z[\alpha]$ into $\F_r$ for primes $r$ which are just above $B$, and lifting the Legendre symbol modulo $r$. In the Randomised NFS, we follow a similar pattern, but choose maps from $\Z[\alpha]$ into $\F_{r^k}$ stochastically and close to uniformly across all $k \log r < L_n\left(\frac{1}{3}\right)$. This exponential increase in the size of the fields used to induce characters is needed to guarantee that we get unconditional equi-distribution of the characters. However, even on the GRH we require taking $k \log r \sim \log^{4/3} n (\log \log n)^{-1/3}$, which is substantially larger than the characters used in the NFS.

To recognise and factor these smooth numbers in the Randomised NFS we use the
\emph{hyperelliptic curve method} of Lenstra, Pila and Pomerance~\cite{LenstraPilaPomerance} which provides
a completely rigorous and efficient means to recognise and factor smooth numbers.

Once a sufficiently large set of linear factors have been found with smooth images in both $\Z$ and $\Z[\alpha]$, we combine congruences to find a subset with even multiplicity of each factor and with image $1$ under the quadratic characters $\chi_{\mathfrak{p}_i}$. Whilst we could use general matrix inversion methods to find a non-trivial element of the kernel, we can exploit the structure of the matrix of exponents to find such an element more quickly. In particular, since both $f_d\mathbf{N}(a-b\alpha)$ and $|a-bm|$ are bounded by $L_n\left(\frac{2}{3}\right)$, they have at most a logarithmic number of factors and so the matrix of exponents is \emph{sparse}. Hence we can use the faster algorithms of Wiedemann~\cite{ThomeWiedemann} or Montgomery~\cite{MontgomeryLanczos}, which are specialised to finding non-trivial elements of the kernel of sparse matrices over $\F_2$.

We also note that the selection of suitable $m, f$ is challenging, as there is no guarantee that all pairs give similar densities of linear factors. We demonstrate a stochastic search method that allows us to to find suitable $m, f$ and extract a congruence of squares with at most a logarithmic slowdown compared to the run time on the heuristic that linear factors with smooth images in $\Z$ and $\Z[\alpha]$ have the same density for all $m, f$. In turn, this means that we do not need to show bounds on the second moments of the number of linear factors available as $m, f$ vary, which allows us to obtain results without use of assumptions such as the Generalised Riemann Hypothesis. The situation as noted earlier may be compared to the analogous case of primality testing, where prior to the AKS results, the deterministic Miller primality test was known to work only under GRH, whilst the randomised Miller-Rabin test worked unconditionally but probabilistically.

\subsection{Concrete Specification of the Algorithm}\leavevmode

We define the Randomised NFS following Buhler, Lenstra and Pomerance \cite{BuhlerLenstraPomerance}. In the subsequent analyses, we use \textsc{IsSmooth} and \textsc{KernelVector}, implicitly implemented via the Hyperelliptic Curve Method and the Wiedemann algorithm respectively. Furthermore, we assume that once \textsc{IsSmooth} has been called, the order of divisibility of each prime below the smoothness bound is known. We abuse notation slightly to write $\log_{-1}$ as the map from the multiplicative group $\{\pm 1\}$ to the additive group of $\F_2$.
\begin{algorithmic}
\Function{RandomNFS}{$n, \beta, \beta', \delta, \sigma, \kappa$}
\State $d \gets \delta\sqrt[3]{\log n\left(\log \log n\right)^{-1}}$ 
\While{\textbf{true}}
  \For{$0 \leq i \leq (2\sigma - \tau) \log^{1/3} n \log \log n)^{2/3}$}
    \For{$0 \leq j \leq 2^{i}$}
      \State{$k \gets 0$}, $m \gets \U\left(\left(\frac{n}{2}\right)^{\frac{1}{d}}, n^{\frac{1}{d}}\right)$
      \State $\hat{f}\left(x,y\right) \gets \sum_{l=0}^d \bar{c}_l x^ly^{d-l} : \bar{c}_l \in \N, \bar{c}_d = 1, \bar{c}_l < m, n = \sum_{i=0}^d \bar{c}_l m^l$
      \State $c_l \gets \U\left(-L_n\left(\frac{2}{3}, \kappa - \delta^{-1}\right), L_n\left(\frac{2}{3}, \kappa-\delta^{-1}\right)\right)$
      \State $f\left(x,y\right) \gets \hat{f}\left(x,y\right) + \sum_{i = 0}^{d-1} c_i\left(x - my\right)x^{d-i-1}y^i$
      \IIf{$f$ is reducible} \Return FAIL \EndIIf
      \State $\S \gets \left\{p < L_n\left(\frac{1}{3}, \beta\right) : p\textrm{ prime}\right\},\;
      \S' \gets \left\{p < L_n\left(\frac{1}{3}, \beta'\right) : p\textrm{ prime}\right\}$
      \State $\mathcal{R} \gets \left\{4d(\delta\kappa\log_2\left(n\right) + \frac{\delta^2\kappa}{2\log2}\frac{\log^{4/3}n}{\log\log n^{1/3}})\textrm{ random pairs }\left(q, s\right) \textrm{ s.t. }q \in [\exp\left(d^4\right), 2\exp\left(d^4\right)], \right.$ \newline
$ \left.\qquad\qquad\qquad q \textrm{ prime, }q \mid f\left(s, 1\right), q \nmid \left(\frac{\partial f}{\partial x}\left(x,y\right)\right)\left(s, 1\right) \right\}$
      \State $\mathcal{L} \gets$ Empty list of pairs $((a, b), \{0,1\}^*)$
      \For{$0 \leq l \leq  2^{-i} . 4(B + B')L_n\left(\frac{1}{3}, \frac{\delta^{-1}}{3\beta} + \frac{\kappa+\sigma\delta}{3\beta'}\right) . \frac{4}{3}\log \log n$}
        \State $a, b \gets \U\left(\left\{\left(a,b\right) : a < |b| \in \left[\frac{1}{2}L_n\left(\frac{1}{3}, \sigma\right), L_n\left(\frac{1}{3}, \sigma\right)\right]\right\}\right)$
        \If{$\Call{IsSmooth}{a-mb, \S} \wedge \Call{IsSmooth}{f\left(a,b\right), \S'}$}
          \State $E_1\gets \langle \ord_p\left(a - bm\right) \in \F_2 : p \in \S \rangle \rangle$
          \State $E_2\gets \left\langle \1_{p \mid a+br}\ord_p\left(N\left(a - b\alpha\right)\right) \in \F_2 : \begin{array}{l}p\in\S', r\in[p], p \mid f(1, r)\end{array}\right\rangle $
          \State $E_3\gets \langle \log_{-1}\left(\frac{a+bs}{q}\right) : \left(q, s\right) \in \mathcal{Q}\rangle$
          \State $\mathcal{L} \gets \mathcal{L} \cup \left\{\left(\left(a,b\right), \langle E_1, E_2, E_3\rangle\right)\right\}$
        \EndIf
        \If{$|\mathcal{L}| > 1+ |\S| + d|\S'| + |\mathcal{R}|$}
          \State $M \gets$ The matrix $M_{i.} = \langle E_1\left(a,b\right), E_2\left(a,b\right), E_3\left(a,b\right), 1\rangle$ for $i \in \mathcal{L}_k$
          \State $V \gets \Call{KernelVector}{M}$
          \State $u_k \gets \left(\prod f_d(a - bm) : \left(a, b\right) = \mathcal{L}[i]\text{ and }V_i = 1\right)$
          \State $v_k \gets \left(\prod f_d(a - bX) : \left(a, b\right) = \mathcal{L}[i]\text{ and }V_i = 1\right)\imod f$
          \State $k \gets k + 1$
          \State $\mathcal{L} \gets$ Empty list of pairs $((a, b), \{0,1\}^*)$
        \EndIf
        \If{$ k = \frac{4}{3} \log \log n$}
          \For{$S \subset [\frac{4}{3}\log \log n], 0 < |S| \leq 2$}
            \If{$f'^2 \prod_{s \in S} v_s$ is square in $\Z[X] \imod f$}
              \State $u \gets \sqrt{f'\left(m\right)^2 \prod_{s \in S} u_s},
              \quad v \gets \left(\sqrt{f'^2 \prod_{s \in S} v_s}\right)(m)$
              \State \Return $\Call{gcd}{u+v, n}$
            \EndIf
          \EndFor
        \EndIf
      \EndFor
    \EndFor
  \EndFor
\EndWhile
\EndFunction
\end{algorithmic}

\subsection{Heuristic Difficulties}\leavevmode

As one would expect, a significant portion of the analysis revolves around precise control over smooth
numbers. Heuristically, one would expect that $f\left(a,b\right)$ behaves as a uniformly random
number below some bound. However, this turns out not to be the case. We are required
to ensure that $f\left(m, 1\right) = n$; this is enforced by ensuring that the random polynomial $R$
is a multiple of $x - my$. As a corollary (see Equation~\refp{first-f-def}),
for $a, b$ fixed our randomisation will constrain
$f\left(a,b\right)$ to lie on an arithmetic progression of common difference $a-mb$. We postpone the
numerical details of the coefficients of $f$ and $\hat{f}$ to Equations~\refp{f-definition} and~\refp{ap-definition}.

The heuristic analysis of the NFS assumes that $\hat{f}$ is
a ``random'' polynomial in some suitable sense. However, $\hat{f}$ is in
fact determined entirely by the fixed $n$ and our chosen $m$. In applications,
$m$ is often chosen carefully to attempt to optimise $\hat{f}$ so that when it is reduced modulo small
primes, it has an unusually large number of roots~\cite{MurphyThesis}. This makes the NFS as used substantially more complicated to analyse, as no variables other than $a$ and $b$ can be considered to be random in a natural way.

We note that even if the polynomial $f$ was completely random, almost all of our analysis
would still be required. In particular, we would still need to show that since a single $f$ is
fixed, the smoothness of the values of $f\left(a,b\right)$ for many pairs $\left(a,b\right)$ are not too correlated.
For example, if a small collection of $f$ were responsible for the majority of smooth values of $f\left(a,b\right)$,
then we would have to examine a large number of different polynomials $f$ before we found one for which
we could generate many pairs $\left(a,b\right)$ as required.

In fact, our polynomial $f = \hat{f} + R$ is not
entirely random, which introduces a degree of additional complexity. However, we are able to show
that its value distribution for small values of $x, y$ is such that the two numbers
$
f\left(a,b\right)$ and $\left(a-bm\right)
$
are $L_{n}\left(\frac{1}{3}\right)$ smooth as often as needed when we choose $\left(a,b\right)$ at
random with their values bounded by $L_{n}\left(\frac{1}{3}\right).$ We also provide a
rigorous analysis of the process of lifting a congruence of squares involving norms
to a congruence involving elements in the number field. As is standard, this involves
an analysis of the primes in $\Q\left(\alpha\right)$, and a small collection
of quadratic characters.

We record the following summary of the computations involved in both the NFS and the Randomised NFS:

\begin{framed}
\begin{enumerate}
\item Fix $m$ an integer, $f$ a homogeneous bivariate polynomial such that $n | f\left(m, 1\right)$, $f$ is irreducible of degree $\left(\frac{\log n}{\log \log n}\right)^{\frac{1}{3} + \o\left(1\right)}$ and with coefficients which are not too large.
\item Generate polynomials $\left(a-bX\right)$ at random for $a, b$ which are not too large
\item \label{smooth-generate}Keep only those polynomials such that $a-mb \in \Z$ and $a-b\alpha \in \Z[\alpha]$ both being smooth.
 (Recall that $a-b\alpha$ is smooth iff $f\left(a,b\right)$ is smooth)
\item \label{square-formation}Find a subsets $S_i$ of pairs $\left(a,b\right)$ such that
\[
\prod_{S_i} f_d\left(a-mb\right)\text{ and } \prod_{S_i}f_d\left(a-b\alpha\right)\text{ are square, and }\forall \chi_{\mathfrak{p}_j},\,
\prod_{S_i} \chi_{\mathfrak{p}_j}\left(a-mb\right) = 1.
\]
\item \label{find-squares}Produce a polynomial whose projection into $\Z$ and $\Z[\alpha]$ are both squares.
\item Produce a congruence of squares in $\Z/n\Z$.
\end{enumerate}
\end{framed}

Note that for Step~\ref{square-formation} to be sure of success, we must find at least as many polynomials in Step~\ref{smooth-generate} as the sum of the number of primes of small norm in $\Z$ and $\Z[\alpha]$. The success of Step~\ref{square-formation} or Step~\ref{find-squares} is not established in the NFS; in the Randomised NFS it is precisely controlled.

Theorem~\refp{NFS-characters-uncond} will give us broad conditions under which Step~\ref{square-formation} and Step~\ref{find-squares} can be completed in the Randomised NFS sufficiently quickly asymptotically almost surely. Primarily, this will correspond to ensuring that we can find a square root in $\Q\left(\alpha\right)$, and will require working with a relatively small random collection of large quadratic characters.

Our other theorems primarily concern themselves with characterising situations in which Step~\ref{smooth-generate} can be achieved with sufficiently high probability. In particular, we will use the flexibility in the choice of $f$ and $m$ to make the events ``$a-b\alpha$ is smooth'' and ``$a-bm$ is smooth'' almost independent and characterise the probability with which they occur. By bounding various correlations we are able to show that for a reasonably large fraction of the $f$ we might choose, the probability with which a polynomial $a-bX$ passes the conditions of Step~\ref{smooth-generate} is reasonably large.

\section{The Randomised Number Field Sieve}

Recall that we add a large random multiple of $\left(X-m\right)$ to our polynomial $f$. This will not substantially increase the coefficients, whilst ensuring that $f\left(m, 1\right) = n$ and ensuring that values of the polynomial at small values of $x$ are randomised usefully. Additionally, for technical reasons we will insist that the degree of our polynomial is \emph{odd} (see the proof of Lemma~\refp{prob-zero-as-used}).

We fix smoothness bounds $B = L_n\left(\frac{1}{3}, \beta\right)$, $B' = L_n\left(\frac{1}{3}, \beta'\right)$, and parameters $\delta, \kappa$, and $\sigma$ to control the degree, coefficients and points of evaluation of our polynomial. We insist that:
\begin{align}\label{const-bounds}
\kappa > \delta^{-1}, \qquad
2\sigma > \max\left(\beta, \beta'\right) + \frac{\delta^{-1}\left(1+\o\left(1\right)\right)}{3\beta} + \frac{\left(\sigma\delta + \kappa\right)\left(1 + \o\left(1\right)\right)}{3\beta'},
\end{align}
\begin{equation}\label{d-bound}
\delta^{-1} < \frac{\kappa + \sigma\delta}{2}
\end{equation}
See Remark~\refp{const-reasons} for a discussion of these bounds.

\begin{definition} Let $\X$ be the set of tuples $\left(f, n, m, a, b\right)$ such that the following four conditions hold:
\begin{enumerate}
\item $m$ is an integer, $m \in \left[2^{-\frac{1}{d}}L_n\left(\frac{2}{3}, \delta^{-1}\right), L_n\left(\frac{2}{3}, \delta^{-1}\right)\right]$,
\item $f$ is a homogeneous polynomial of degree $d = \delta\sqrt[3]{\frac{\log n}{\log \log n}}$, $d$ odd, in two variables with integer coefficients bounded by $L\left(\frac{2}{3}, \kappa\right) \left(1+\o\left(1\right)\right)$, with $f\left(m,1\right) = n$. In particular, we count such $f$ such that that:
\begin{align}
c_i &\in \I\left(2L_n\left(\frac{2}{3}, \kappa-\delta^{-1}\right)\right)\text{ (Recall Definition~\refp{centredinterval}), and set}\nonumber\\
\qquad\quad\label{f-definition} f\left(x,y\right) &\defeq \hat{f}_{n,m}\left(x,y\right) + \sum_{i=0}^{d-1} c_i\left(x - my\right)x^{d-i-1}y^i
\end{align}
and $\hat{f}_{n,m}\left(x,y\right) \defeq \sum C_i x^{d-i}y^i$, with $C_i$ given by expressing $n$ as a polynomial in $m$ with coefficients in $\left[0, m\right)$ (that is, by expressing $n$ as an $m$-ary number). We recall that this is the major alteration in the Randomised NFS, as the NFS can be seen to take $c_i \equiv 0$, whereas in the Randomised NFS the $c_i$ are chosen independently and uniformly randomly.
\item $a,b$ are integers, $0 \leq a < |b| \in \left[\frac{1}{2}L_n\left(\frac{1}{3}, \sigma\right), L_n\left(\frac{1}{3}, \sigma\right)\right]$, with $a - bm$ being $B$-smooth and $f\left(a,b\right)$ being $B'$-smooth.
\end{enumerate}
We also define $\X_{n,m,f}$ be the set of pairs $(a, b)$ such that $(f, n, m, a, b) \in \X$.
\end{definition}

Recalling our earlier discussion of combination of congruence algorithms, it can be seen that the condition $(a, b) \in \X_{n,m,f}$ are \emph{almost} those required for the factor $\left(a - Xb\right)$ to be used in the combination of congruences. Hence showing that the number of such tuples is \emph{large} will correspond to showing that we can find a large number of tuples \emph{quickly}. The sole missing condition is that we do not require $f$ to be irreducible; indeed, we will freely interchange between $f$ considered as a homogeneous bivariate polynomial and the single variable non-homogeneous $f$ usually discussed in the NFS.

\begin{rem}\label{const-reasons}
The constraints given in Equation~\refp{const-bounds}. The first condition ensures that $c_i \gg m$. We will use this to show, speaking loosely, that for any fixed pair $a < |b| < L_n\left(\frac{1}{3}, \sigma\right)$, the event of $f\left(a,b\right)$ being smooth is driven by the values of $c_i$ rather than by the inflexible interaction of $n$ and $m$. The second constraint from Equation~\refp{const-bounds} will ensure that there are enough suitable pairs $a,b$ that almost surely there will be a congruence of squares. Equation~\refp{d-bound} will ensure that the distribution of smooth numbers $f(a,b)$ modulo $a-mb$ can be controlled by character methods. We further note that the value of $f\left(a,b\right)$ lies on the arithmetic progression:
\begin{align}\label{ap-definition}
\left\{\hat{f}_{n,m}\left(a,b\right) + \left(a-mb\right) z : |z| \leq dL_n\left(\frac{2}{3}, \kappa-\delta^{-1}\right)b^d\right\}
\end{align}
Crucially, we will later show that as $c$ is randomised, $f\left(a,b\right)$ is $B'$-smooth as often as a uniformly random element of this progression is.
\end{rem}

\begin{rem} Since $c_i \gg m$, the coefficients in $f$ are somewhat larger than in $\hat{f}_{n,m}$. Thus the bounds on the discriminant $\Delta\left(f\right)$ are weakened in the Randomised NFS by comparison to the standard NFS. This will have an impact in the proof of Theorem~\refp{NFS-characters-uncond}, although we will see there that the bounds are still sufficiently tight. In particular, the squares of smooth-normed elements of $\Z[\alpha]$ are still a comparatively large subset of the elements of $\Z[\alpha]$ with smooth and square \emph{norms}, and so a small collection of quadratic characters can be used to identify the squares amongst elements with smooth and square norms.
\end{rem}

We first reduce Theorem~\refp{NFS-L13-analysis} to Theorems~\refp{NFS-smoothness-base} and~\refp{NFS-characters-uncond}.

\begin{proof}[Proof of Theorem~\refp{NFS-L13-analysis}]
Fix $n, \beta, \beta'$, $\sigma$, $\delta$, $\kappa$ satisfying the conditions of Equation~\refp{const-bounds} and Equation~\refp{d-bound}. Then by Theorem~\refp{NFS-characters-uncond} we can extract a congruence of squares mod $n$ from $L_n\left(\frac{1}{3}, \max\left(\beta, \beta'\right) + \o(1)\right)$ pairs $(a, b) \in \X_{n, m, f}$ for a fixed $(m, f)$ in expected time
\[
L_n\left(\frac{1}{3}, 2 \max\left(\frac{2\delta}{3}, \beta, \beta'\right)(1+\o(1))\right)
\]

Theorem~\refp{NFS-smoothness-base} tells us that a fixed $(m, f)$ and this many pairs $(a, b) \in \X_{n, m, f}$ will be found in expected time
\[
L_n\left(\frac{1}{3}, \max\left(\beta, \beta'\right) + \frac{\delta^{-1}\left(1+\o\left(1\right)\right)}{3\beta} + \frac{\left(\sigma\delta + \kappa\right)\left(1 + \o\left(1\right)\right)}{3\beta'}\right).
\]
Hence we can run the Randomised NFS to obtain a congruence of squares mod $n$ with the expected time bounded by
\begin{gather*}
L_n\left(\frac{1}{3}, \lambda\left(1+\o\left(1\right)\right)\right), \quad
\lambda \defeq \max\left(2\max\left(\frac{2\delta}{3}, \beta, \beta'\right), \max\left(\beta, \beta'\right) + \left(\frac{\delta^{-1}}{3\beta} + \frac{\kappa+\sigma\delta}{3\beta'}\right)\right)
\end{gather*}
Note that to obtain a concrete bound we must choose $\beta, \beta', \delta, \sigma, \kappa$ subject to Equation~\refp{const-bounds} and~\refp{d-bound}, which we collect here for convenience:
\begin{gather*}
\min\left(\frac{\kappa + \sigma\delta}{2}, \kappa\right) > \delta^{-1},
2\sigma > \max\left(\beta, \beta'\right) + \frac{\delta^{-1}\left(1+\o\left(1\right)\right)}{3\beta} + \frac{\left(\sigma\delta + \kappa\right)\left(1+ \o\left(1\right)\right)}{3\beta'}.
\end{gather*}

We optimise the constants. Note that increasing the lesser of $\beta, \beta'$ cannot increase $\lambda$ or cause the conditions on the constants to be violated, so we can assume $\beta = \beta'$. We can compute the following bound on $\sigma$:
\begin{align*}
2\sigma \geq \lambda &\geq \min_{\beta, \delta}\left(\beta + \frac{2\delta^{-1} + \sigma\delta + \o\left(1\right)}{3\beta}\right)
\geq  \min_{\beta}\left(\beta + \frac{\sqrt{8\sigma}+ \o\left(1\right)}{3\beta}\right)
\geq 2\sqrt[4]{\frac{8\sigma}{9}} + \o\left(1\right)
\end{align*}

Fix any $\epsilon > 0$, $\epsilon = \o\left(1\right)$. If we take $\beta = \beta' = \sigma = \frac{2\delta}{3} = \sqrt[3]{\frac{8}{9}} + \epsilon$, $\kappa = \sqrt[3]{3^{-1}}+ \epsilon$, the above are all equalities (up to $\O\left(\epsilon\right)$ terms). Furthermore, all the conditions of Equation~\refp{const-bounds} and~\refp{d-bound} are satisfied, and $\lambda = 2\sqrt[3]{\frac{8}{9}} + \o\left(1\right)$ matching the heuristic optima as claimed.
\end{proof}

\begin{rem}
The above argument, with the statement of Theorems~\refp{NFS-smoothness-base} and~\refp{NFS-characters-uncond} modified as in Remark~\refp{MPNFS-thm-mod}, plainly establishes that a Randomised Coppersmith multiple polynomial NFS finds a congruence of squares in the given time. Optimising the constants yields $\beta = \frac{3\beta'}{\sqrt{13}-1} = \sigma = \frac{3\delta}{4\sqrt{13}-10} = \frac{3\eta}{4 - \sqrt{13}} = (\frac{46 + 13\sqrt{13}}{108})^{1/3} + \o(1)$, $\kappa = \delta^{-1} + \o(1)$, achieving $\lambda = \sqrt[3]{\frac{92 + 26\sqrt{13}}{27}} + \o(1)$.
\end{rem}

\section{Finding Many Relationships and the Proof of Theorem~\ref{NFS-smoothness-base}}
Given $\left(f,n,m,a,b\right) \in \X$, let $\alpha \in \BB{C}$ with $f\left(\alpha,1\right) = 0$. Then the map
\[
\Z[\alpha] \simeq \Z[X]/\left(f\left(x, 1\right)\right) \rightarrow \Z / n\Z\textrm{ defined by }1 \rightarrow 1, \alpha \rightarrow m
\]
and extended multiplicatively is a homomorphism of rings, since $f\left(m, 1\right) \rightarrow n \equiv 0 \imod{n}$. We also have a multiplicative map from the ring of integers $\Q\left(\alpha\right) \rightarrow \Q$, the so-called \emph{field norm} $\mathbf{N} = \mathbf{N_{\Q\left(\alpha\right)/\Q}}$. This norm can be defined by sending any element $z$ of the number field to the product of all images of $z$ under embeddings of the field into $\BB{C}$.

Note that on $\Z +\alpha \Z$, such a product can be expressed as a sum of integer multiples of products of the symmetric polynomials evaluated at the roots of $f$. Since the elementary symmetric polynomials in the roots of $f$ are the \emph{coefficients} of $f / f_d \in \frac{1}{f_d}\Z$, the field norm is guaranteed to be in $\frac{1}{f_d}\Z$ on $\Z + \alpha\Z$.

Hence if $f$ is irreducible, we are in the setting discussed earlier and so the established NFS strategy can be used to find a congruence of squares modulo $n$.
\begin{lemma}\label{f-reducible-probability}
$\P(f\text{ is reducible}) \leq L_{N}\left(\frac{2}{3}, \frac{\kappa-\delta^{-1} + \o\left(1\right)}{3}\right)^{-1}$.
\end{lemma}
\begin{proof} Fix $n,m$, and let $H = 2L_n\left(\frac{2}{3}, \kappa-\delta^{-1}\right)$ be the range of each coefficient of the random part of our polynomial $f$. Note that if a polynomial over $\Z$ is reducible it is reducible modulo every prime. Hence if we bound the number of reducible polynomials modulo $\F_z$ for each prime $z$, and bound how often a polynomial is reducible modulo several primes $z$, we can get good bounds on the number of irreducible polynomials $f$.

We count the reducible polynomials $f$ with the Tur\'an Sieve, as in \cite[Section 4.3]{CojocauruMurty}. Let:
$
\mathcal{A} \defeq \{(c_{d-1}, \ldots, c_0) \in \Z^d, |c_i| < H\}
$
which we equate with the set of $f$ as before as
$
f(x,y) = \hat{f}_{n,m}(x,y) + (x - my)R(x,y)
$
with $f, \hat{f}_{n,m}$ both homogeneous of degree $d$ and with $(c_i)$ the coefficients of R. For any prime $r$, let $\mathcal{A}_r$ correspond to the subset of $\mathcal{A}$ corresponding to irreducible polynomials mod $r$. Note that for any $f$ to correspond $g\imod{r}$ we must have
$
(x - my) | \hat{f}_{n,m} - g \in \F_r[X,Y]$ or equivalently $g(m,1) \equiv n \imod{r}
$

We do not insist that $g$ is monic, although any irreducible $g$ must be a scalar multiple of a monic irreducible. To estimate the number of irreducibles, we follow the argument of~\cite[Chapter 2]{rosen}:
\begin{claim}
For any $0 < i < r$, the number of monic irreducibles $g$ of degree $D$ such that $g(m) \equiv i \imod r$ is $\frac{r^{D-1}}{D(D-1)} + \O(r^{D/2})$.
\end{claim}
\begin{proof}
We work in $\F_r$, and let $|g| \defeq r^{\operatorname{deg}(g)}$. We observe that for $\chi$ a non-trivial multiplicative character:
\[
\zeta_{\F_r[X]}(s, \chi) = \sum_{\substack{g \in \F_r[X],\\ g\text{ monic}}}\frac{\chi(g(m))}{r^{s\deg(g)}} = 1,
\]
as for every degree $d \geq 1$ the number of monic polynomials whose evaluation at $m$ is any chosen $i$ is exactly $r^{d-1}$. Let $v = r^{-s}$, and let $a_{d,i}$ be the number of irreducibles $g$ of degree $d$ with $g(m) = i$. As is standard, we express the sum as an Euler product over the monic irreducibles and take the logarithmic derivative:
\[
1 = \prod_{d, i} \left(1 - \chi(i)v^d\right)^{-a_{d,i}},\text{ so }
0 = \sum_{d,i} \frac{da_{d,i}\chi(i)v^{d-1}}{1-\chi(i)v^d}
\]
Expanding and comparing terms, we obtain that for every $D$:
\[
0 = \sum_{d | D}d \sum_{i} a_{d,i} \chi(i)^{D/d},\text{ so }
\Rightarrow \sum_{i} a_{D,i} \chi(i) = - D^{-1}\sum_{d | D, d < D}d \sum_{i} a_{d,i} \chi(i)^{D/d}
\]
Note that $\sum_{d | D}d \sum_{i\neq 0} a_{d,i} = r^D - 1$, and $\sum_{d | D, d < D}d \sum_{i} a_{d,i} = \O(r^{D/2})$. Hence by writing the indicator $\mathbbm{1}_{x \equiv i \imod{r}}$ as a sum of characters, we obtain:
\[
a_{D,i} = \frac{r^D - 1}{D(D-1)} + \O\left(r^{D/2}\right). \qedhere
\]
\end{proof}
To continue the proof of Lemma~\refp{f-reducible-probability}, we note that for any $g$ over $\F_r$ such that $g(m) \equiv n$ with $r \ll \sqrt{H}$, there are:
\[
\left(\frac{H}{r} + \O(1)\right)^{d} = \left(\frac{H}{r}\right)^{d} + \O\left(\left(\frac{H}{r}\right)^{d-1}\right)
\]
polynomials lying over $g$ in $\mathcal{A}$, and none if $g(m) \not\equiv n \imod{r}$. Hence by a union bound over the irreducibles mod $r$:
\begin{align*}
|\mathcal{A}_r| &\leq \frac{H^{d}}{d(d-1)} + \O\left(\frac{H^{d}}{r^{d/2-1}}\right) + \O\left(H^{d-1}r\right) \\
|\mathcal{A}_r \cap \mathcal{A}_{r'}| &\leq \frac{H^d}{d^2(d-1)^2} + \O\left(\frac{H^d}{r^{d/2-1}}\right) + \O\left(\frac{H^{d}}{{r'}^{d/2-1}}\right) + \O\left(H^{d-1}rr'\right)
\end{align*}
From the Tur\'an Sieve \cite[Theorem 4.3.1]{CojocauruMurty}, considering all primes $r \leq z$ for any $z \ll \sqrt{H}$, the number of reducible polynomials $f$ is 
$
\ll H^{d}z^{-1}\log z + H^{d-1}z^2,
$
which for $z \sim H^{1/3}\log^{1/3}H$ is $H^{d-\frac{1}{3}}\log^{\frac{2}{3}} H$. The number of potential $f$ for this fixed $n,m$ is $H^{d}$, and so the probability that $f$ is reducible is at most:
\[
H^{-\frac{1}{3}}\log^\frac{2}{3}H = L_n\left(\frac{2}{3}, \frac{\kappa-\delta^{-1} + \o\left(1\right)}{3}\right)^{-1}. \qedhere
\]
\end{proof}

\begin{rem}
We will assume a fortiori that if $f$ is reducible then the algorithm fails. We will sample at most $L_n\left(\frac{1}{3}\right)$ polynomials, and so the probability that any of them are reducible is $\o(1)$.
\end{rem}

For $f$ irreducible, $\mathbf{N}\left(a - b\alpha\right) = f_d^{-1}f\left(a,b\right)$. We prove the following Theorem later; assuming it we can complete the proof of Theorem~\refp{NFS-smoothness-base}.

\begin{thm}\label{NFS-X-large} With $\beta = \beta', \delta, \sigma, \kappa$ chosen subject to Equation~\refp{const-bounds} and~\ref{d-bound}:
\[
\E_{m,f}\left(|\X_{n,m,f}|\right) \geq L_n\left(\frac{1}{3},\tau\right)\textrm{, with }\tau = 2\sigma - \frac{\delta^{-1}}{3\beta'}\left(1+\o\left(1\right)\right) - \frac{\sigma\delta + \kappa}{3\beta}\left(1 + \o\left(1\right)\right)
\]
\end{thm}

\begin{rem}
The constant $\tau$ defined above is natural, and we will see the terms comprising it regularly in this work. Observe that
$
m \sim L_n\left(\frac{2}{3}, \delta^{-1}\right)
$
and that since $a, b \sim L_n(\frac{1}{3}, \sigma)$ and $d = \delta\log^{\frac{1}{3}} n (\log \log n)^{-\frac{1}{3}}$, for all $0 \leq i \leq d$,
$
a^ib^{d-i} \sim L_n\left(\frac{2}{3}, \sigma\delta\right).
$
We also note that the coefficients of $f$ are of size $L_n(\frac{2}{3}, \kappa)$. Hence when terms of the form $\kappa + \sigma\delta$ appear in the exponents of $L_n$, this should be thought of heuristically as taking a typical evaluation $f(a,b)$, whilst terms of the form $\delta^{-1}$ denote a typical evaluation $a-mb$.

The replacement of $\frac{2}{3}$ by $\frac{1}{3}$ as the first argument and division of the exponent by $3\beta$ or $3\beta'$ correspond exactly to considering the probability that an $L_n(\frac{2}{3})$ number is in fact $B$-smooth or $B'$-smooth.
\end{rem}
\begin{proof}[Proof of Theorem~\refp{NFS-smoothness-base}]
Define $\tau = 2\sigma - \frac{\delta^{-1}}{3\beta'} - \frac{\sigma\delta + \kappa}{3\beta}$, and note that:
\[
\lambda \geq \max\left(\beta, \beta'\right) + \frac{\delta^{-1}\left(1+\o\left(1\right)\right)}{3\beta} + \frac{\left(\sigma\delta + \kappa\right)\left(1 + \o\left(1\right)\right)}{3\beta'} = \max\left(\beta, \beta'\right) + 2 \sigma - \tau + \o(1).
\]

For any fixed pair $(m, f)$, Corollary~\refp{HECM-speed} that we can use the hyperelliptic curve method to examine any pair $\left(a,b\right)$ for suitable smoothness of $a-mb$ and $f\left(a,b\right)$ in $\max\left(B, B'\right)^{\o\left(1\right)}$ time. Hence we can determine whether a pair $(a, b)$ is in $\X_{n, m, f}$ in time $L_n(\frac{1}{3}, \o(1))$.

Lemma~\refp{f-reducible-probability} implies that the probability that $f$ is reducible is $L_n(\frac{2}{3})$, and we have an unconditional uniform bound $|\X_{n, m, f}| \leq L_n(\frac{1}{3}, 2 \sigma)$. Hence from Theorem~\refp{NFS-X-large} we deduce
\[
\E_{m, f}\left(\left|\X_{n,m,f}\right| \middle| f\textrm{ irreducible}\right) \geq L_n\left(\frac{1}{3}, \tau + \o(1)\right)
\]

We now introduce a method of searching large parameter spaces we term \emph{stochastic deepening} to complete the proof. In particular, once $m, f$ have been chosen the depth of the search for pairs $(a,b)$ for is random, with deeper searches being rarer. Suppose there is a reasonable probability that a normal depth search fails for random $m,f$. Then it must be that most $|X_{n,m,f}|$ are small. Since the expectation is controlled, this means that in the remaining cases $|X_{n,m,f}|$ must be large. In this case, a much shallower search will find enough pairs if $|X_{n,m,f}|$ is large, so we can test many $m,f$ less intensely. To make this intuition rigorous, we first note:
\begin{lemma}\label{expectation-bounds}
If a random variable $X$ has $\E(X) = \mu$ and there is a $K \geq 1$ such that $0 \leq X \leq K \mu$ uniformly, then $\exists i \in \{0, \ldots, \lceil\log_2 K\rceil\}$ such that:
\[
\P\left(X \geq \frac{2^{i} \mu}{1 + \lceil\log_2 K\rceil}\right) \geq \frac{1}{2^{i+1}}
\]
\end{lemma}
\begin{proof}
Suppose not. Then:
\[
\E(X) < \sum_{i = 0}^{\lceil\log_2 K\rceil}\left(\frac{1}{2^{i}} - \frac{1}{2^{i+1}}\right) \frac{2^{i} \mu}{1 + \lceil\log_2 K\rceil} + \frac{K\mu}{2^{1 + \lceil\log_2 K\rceil}} \leq \mu\qedhere
\]
\end{proof}

\begin{rem}
Conceptually, this lemma states that for non negative variables which do not vary too much, there must be a reasonably large set where the value is large, whose contribution to the mean is large. This is the core observation that permits stochastic deepening to provide a search algorithm whose run times are shown to be near optimal without establishing accurate variance bounds.
\end{rem}

In our application, we consider $|\X_{n,m,f}|$ to be a random variable of $(m, f)$, with $K \leq L_n(\frac{1}{3}, 2\sigma - \tau)$. Hence for some $i^* \leq 1 + \lceil\log_2 K\rceil = \O(\log^{1/3} n (\log \log n)^{2/3})$, we have (absorbing logarithmic terms):
\[
\P_{m, f}\left(|X_{n, m, f}| \geq 2^{i^*} L_n\left(\frac{1}{3}, \tau + \o(1)\right)\right) > 2^{-i^*}
\]

To find a collection of pairs the algorithm iterates through each $i \in \{0, \ldots, 1 + \lceil \log_2 K\rceil\}$, and for each $i$ generates $2^i$ pairs $(m, f)$, and for each pair $(m, f)$ generates $2^{-i} L_n(\frac{1}{3}, \max(\beta, \beta') + 2\sigma - \tau + \o(1))$ pairs $(a,b)$ and tests for smoothness of $a - mb$ and $f(a, b)$.

Then if $|\X_{n, m, f}| > 2^i L_n\left(\frac{1}{3}, \tau + \o\left(1\right)\right)$, with constant probability the algorithm finds $L_n(\frac{1}{3}, \max(\beta, \beta') + \o(1))$ pairs as required. Furthermore, if $\P_{m, f}\left(|X_{n, m, f}| \geq 2^i L_n\left(\frac{1}{3}, \tau + \o(1)\right)\right) > 2^{-i}$ then with constant probability at least one of the pairs $(m, f)$ satisfies this condition.

Note that the total time taken to test a single $i$ is $L_n\left(\frac{1}{3}, \max\left(\beta, \beta')\right) + 2\sigma - \tau + \o(1)\right)$, and so we can absorb the logarithmic number of iterations into the $\o(1)$ term. Since this algorithm succeeds with constant probability, iterating it at most a logarithmic number of times reduces the probability of failure to $L_n\left(\frac{2}{3}, \kappa - \delta^{-1}\right)$. 

Hence the expected time taken to complete the algorithm is:
\[
L_n\left(\frac{1}{3}, \max\left(\beta, \beta'\right) + 2 \sigma - \tau + \o(1)\right)
\]
as required.
\end{proof}
\begin{rem}
We can save the logarithmic factors lost by the stochastic deepening by noting that if for a particular $m, f$ the search for pairs $(a,b)$ is to succeed, it must find $L_n\left(\frac{1}{3}, \max(\beta, \beta')\right)$ of them. As a corollary, at some early stage of a planned search (say a $\ll \left(\log n\right)^{-1}$ fraction of the way through), one has reasonable estimates of the density of pairs $(a, b)$ for this $m,f$. Aborting searches early can be shown to reduce the cost of searches that would fail to generate at least $1-\o(1)$ of the needed pairs by a factor $\gg \log n$, whilst discarding almost no searches that would find enough pairs. Hence continuing any search that is not aborted to $1+\o(1)$ of its planned depth will find enough relationships.


\end{rem}

Our goal is the proof of Theorem~\refp{NFS-X-large}, which appears on page \pageref{NFS-X-large-pf}, and we proceed with preparatory lemmas. For each $n$ and $b$, we determine how likely the pair $(a,b)$ is to be in $\X_{n,m,f}$ as $a,m,f$ vary. In particular, the distribution of $f$ is well understood, whilst the resulting distribution of $f\left(a,b\right)$ is not. We seek to show that this randomness of $f$ causes $f\left(a,b\right)$ to be as likely to be smooth as a typical number of the same size. An obstruction is that $a-mb$ must be $B$-smooth, which is a rare event and hence heuristically derived ``typical'' behaviour does not have to hold at the points where we evaluate $f$. We will bound how far $f\left(a,b\right)$ deviates from being uniformly random along any arithmetic progression of common difference $\left(a-mb\right)$. Then we can show that $f\left(a,b\right)$ is as likely to be smooth as a random integer.

Recall from Equation~\refp{f-definition} that for any $n, m$, we take $f$ to be uniformly random by choosing:
\begin{align}\label{c-definition}
\left(c_i\right) \sim \mu \defeq \U\left(\I\left(2L_n\left(\frac{2}{3}, \kappa-\delta^{-1}\right)\right)^{d}\right)
\end{align}
and defining $f$ according to definition of Equation~\refp{f-definition}. Note that $\hat{f}_{n,m}$ is completely determined by $n,m$, but the random sum \[
f\left(x,y\right)-\hat{f}_{n,m}\left(x,y\right) = R(x,y) = \sum_{i=0}^{d-1}c_i\left(x-my\right)x^{d-i-1}y^i
\]
dominates $\hat{f}$ as $\kappa > \delta^{-1}$. For any $a,b$, $f\left(a,b\right) \equiv \hat{f}_{n,m}\left(a,b\right) \imod{a-mb}$. Clearly, $\gcd\left(a,b\right)^d$ divides $f\left(a,b\right)$ and $\hat{f}\left(a,b\right)$. Hence $R\left(a,b\right)$ has $\left(a-mb\right)\gcd\left(a,b\right)^{d-1}$ as a factor. We take $b$ and $a$ to be uniformly random in their ranges.

\begin{lemma}\label{a-mb-is-smooth}
Fix $b$ in its interval. If $a, m$ are uniformly random, then:
\[
\P_{a,m}\left(a-bm\textrm{ is } B\textrm{-smooth}\right) = L_n\left(\frac{1}{3}, \frac{\delta^{-1}}{3\beta}\left(1+\o\left(1\right)\right)\right)^{-1}.
\]
\end{lemma}
\begin{proof}
We fix $b$. Note that $a$ is uniformly random on an interval of length $b$, and $m$ uniformly random over an interval of length comparable to its largest value. In particular:
\begin{align*}
a-bm &\sim \U\left[-bL_n\left(\frac{2}{3}, \delta^{-1}\right), -b\left(2^{-\frac{1}{2d}}L_n\left(\frac{2}{3}, \delta^{-1}\right)+1\right)\right) 
= \U\left[-x(1+z^{-1}), -x)\right)
\end{align*}
for $x = L_n\left(\frac{2}{3}, \delta^{-1}(1+\o(1))\right)$ and $z \approx 2^{\frac{1}{2d}} - 1 = \O(d^{-1})$. Note that $d = \o(B^{5/12})$, and that $\log \log B = \O(\log \log x)$. Hence from Fact~\refp{H-interval} the number of smooth values of $a-mb$ is:
\[
\frac{\Psi(x, B)}{z}\left(1 + \O\left(\frac{\log (u + 1)}{\log B}\right)\right)
\]
Since the range of values is of length $x/z$,
\[
\P_{a,m}\left(a-bm\textrm{ is } B\textrm{-smooth}\right) = \varrho(x, B)\left(1 + \O\left(\frac{\log (u + 1)}{\log B}\right)\right).
\]
Recall that $B = L_n\left(\frac{1}{3}, \beta\right)$ and $x = L_n\left(\frac{2}{3}, \delta^{-1}\right)^{1+\o(1)}$. Furthermore, note that $\log u < \log
\log n = \o(\log B)$. Hence recalling Corollary~\refp{CEP-cor}:
\[
\varrho\left(L_n\left(\frac{2}{3}, \delta^{-1}\right)^{1+\o(1)}, B\right) = L_n\left(\frac{1}{3}, \frac{\delta^{-1}}{3\beta}\right)^{-1+\o(1)}.
\]
We can absorb the multiplicative $1+\o(1)$ error into the $\o(1)$ error in the exponent to obtain:
\[
\P_{a,m}\left(a-bm\textrm{ is } B\textrm{-smooth}\right) = L_n\left(\frac{1}{3}, \frac{\delta^{-1}}{3\beta}(1+\o(1))\right)^{-1}.
\]
\end{proof}
\begin{rem}
To prove the analogous statement for Coppersmith's multiple polynomial NFS, as modified by Remark~\refp{MPNFS-thm-mod}, we use Lemma~\ref{expectation-bounds} twice, first to select an $m$ and for each $m$ to attempt to find many polynomials $f^{(i)}$ with many smooth pairs $a, b$. If we guess correctly the values of $i$ correctly at both steps then with probability $\O(1)$ our sample of values of $m$ contains a value, such that with probability $\O(1)$ the sample of $f^{(j)}$ chosen for this $m$ has a large enough $\sum|\X_{n, m, f^j}|$ that with probability $\O(1)$ we find enough pairs $(a,b)$ that are smooth for some $f^{(j)}$.
\end{rem}
\begin{rem} To prove Theorem~\refp{NFS-X-large}, we will estimate the probability that $f\left(a,b\right)$ is $B$-smooth. Note that for a pair $\left(a,b\right)$ to be in $\X_{n,m,f}$, it is required that $a-mb$ to be $B$-smooth. As a corollary, we know that for all of these pairs the greatest common divisor of the pair $\left(a,b\right)$ is $B$-smooth. Hence we can divide $a$ and $b$ by $\gcd\left(a,b\right)$ without changing the $B$-smoothness of $f\left(a,b\right)$. In Lemma~\refp{S-for-phi} to Lemma~\refp{f-is-smooth} we only seek to establish the $B$-smoothness of $f\left(a,b\right)$. Hence for convenience we will take $\gcd\left(a,b\right) = 1$ without loss of generality.
\end{rem}

We wish to show that $f\left(a,b\right)$ is as likely to be $B$-smooth as a random number of the same size. To do this, we will show that
\begin{enumerate}
\item $f\left(a,b\right)$ is close to uniformly distributed along long progressions of common difference $\left(a-mb\right)$ (proved in Lemma~\refp{S-for-phi} to Lemma~\refp{intervalmult}).
\item For most $B$-smooth moduli $a-mb$, the $B'$-smooth numbers are approximately uniformly distributed modulo $a-mb$.
\end{enumerate}
To show the first property, we fix the residue $f\left(a,b\right) \imod{a-mb}$, and consider the effect of our random choice of vector $c$ in Equation~\refp{c-definition}. First, we show that there exist small changes to $c$ that will alter $f\left(a,b\right)$ by any small multiple of $\left(a-mb\right)$ in Equation~\refp{f-c-modification}. To

To show the second, we introduce a notion of goodness for moduli which is strong enough to allow us to control the NFS should the modulus $a-mb$ turn out to be $B'$-good.
\begin{definition}Fix $B = L_n\left(\frac{1}{3}\right)$, $F = L_n\left(\frac{2}{3}\right)$ and some $\epsilon\left(F,B,r,a\right) = \o_n\left(1\right)$, $\omega = L_n\left(\frac{1}{3}\right)$. We say a modulus $r$ is $B$-good for $F$ if uniformly over all $\left(a,r\right) = 1$:
\begin{align*}
\Psi\left(F,B; r,a\right) &= \left(\frac{\Psi_r\left(F,B\right)}{\phi\left(r\right)}\right)^{1+\epsilon}
\end{align*}
and $B$-bad for $F$ otherwise. We will routinely suppress $\epsilon$, as we only need that the error exponent is taken to be $\o\left(1\right)$.

If $\mathcal{F} = L_n\left(\frac{2}{3}\right)$ and for all $F \in [\mathcal{F}\omega^{-1}, \mathcal{F}]$, $r$ is $B$-good for $F$ then we say $r$ is $B$-good near $\mathcal{F}$. Often, we will suppress $\mathcal{F}$ and say $r$ is $B$-good. Our results on the number of $B$-good moduli $r$ will not be sensitive to the precise form of $\omega$, and so we suppress it.
\end{definition}
Heuristically, a modulus is $B$-good when $B$-smooth numbers up to $F \in \mathcal{F}$ modulo $r$ are close to uniformly distributed.

\begin{lemma}\label{S-for-phi}
Given $a < b$, with $\gcd\left(a,b\right) = 1$, define $\varphi = \varphi_{a,b} : \Z^d \rightarrow \Z$,
\[
\varphi\left(\left(v_0, \ldots, v_{d-1}\right)\right) \defeq \sum_{i=0}^{d-1} v_i a^{d-1-i} b^i.
\]
There exists a set $S \subseteq \I\left(4 L_n\left(\frac{1}{3}, \sigma\right)\right)^{d}$ such that $\varphi$ bijects $S$ and $\I\left(b^{d-1}\right)$.
\end{lemma}
\begin{proof}
For each $i \geq 0$, we claim that for any $|t| \leq b^i + a^{i+1}$ there exists a representation:
\[
t = a^i x_0 + a^{i-1}bx_1 + \ldots + b^i x_i
\]
with $|x_0|, \ldots, |x_i| \leq a + b$. We proceed inductively. Note that the number of terms in the sum is $i+1$. The case $i = 0$ is trivial. If $i > 0$, we may choose $y$ with $|y| < a$ such that
\[
|t - ya^i| \leq b^i.
\]
We then fix $z \in [b]$ such that $za^i \equiv t - ya^i \imod{b}$. Note that $|y| < a$ and $|z| < b$. We set $x_0 = y+z$, so $|x_0| \leq a+b$. Note that $b \mid t - x_0a^i$ and that:
\[
\left|\frac{t - x_0a^i}{b}\right| = \left|\frac{(t-ya^i) - za^i}{b}\right| \leq b^{i-1} + a^i
\]
We need $(t - x_0a^i)b^{-1} = a^{i-1}x_1 + a^{i-2}bx_2 + \ldots + b^{i-1}x_i$, which we can guarantee inductively with $|x_1|, \ldots, |x_i| \leq a+b$.

We now directly show the existence of $S$. For any $t \in \I\left(b^{d-1}\right)$, $|t| \leq b^{d-1}$ and so the conditions of the above hold with $i = d-1$. So there exist a sequence $x_0, \ldots x_{d-1}$ such that $\sum_{i=0}^{d-1} x_ia^{d-1-i}b^i = t$ with $|x_0|,\ldots,|x_{d-1}| < a+b$. Hence we have a vector $v_t$ given by $(v_t)_i = x_i$, with $v_t \in \I\left(2\left(a+b\right)\right)^d \subset \I\left(4L_n\left(\frac{1}{3}, \sigma\right)\right)^d$ and $\varphi\left(v_t\right) = t$.

Hence take $S = \left\{v_t : t \in \I\left(b^{d-1}\right) \right\}$. For each $t$ we have constructed a single $v_t$, so function $\phi$  is injective and surjective on $S$ as required.
\end{proof}

By definition of $\varphi = \varphi_{a,b}$, and making the dependence of $f$ on $c$ explicit as $f_c$ (with $m, n$ held constant):
\begin{equation}\label{rho-and-f}
f_c\left(a,b\right) = f_{c'}\left(a,b\right) + \left(a-mb\right)\varphi_{a,b}\left(c-c'\right).
\end{equation}
This motivates the following definition, which will give us an additive kernel whose support is bounded to a small cube and which makes a uniformly random small change to $f_c(a,b)$ when it is applied to $c$.
\begin{definition}
We take $S$ to be the set given from Lemma~\refp{S-for-phi}. For any $l \leq b^{d-1}$, we define a set $S_l$ and a measure $\nu_l$ as follows:
\begin{equation}\label{nu-l-def}
S_l \defeq \left\{v \in S : \varphi\left(v\right) \in \I\left(l\right) \right\}, \qquad
\nu_l \defeq \U\left(S_l\right).
\end{equation}
In particular, $\nu_l$ gives a uniformly random element of $S$ whose image under $\varphi$ is in $\I\left(l\right)$.
\end{definition}
From the definition of Equation~\refp{rho-and-f}, if $v \sim \nu_l$,
\begin{equation}\label{f-c-modification}
f_{v}\left(a,b\right) \sim f_{\underline{0}}\left(a,b\right) + \left(a-mb\right)\U\left(\I\left(l\right)\right),
\end{equation}
i.e. that measures $\nu_l$, with support $S_l$, give us additive alterations that can be made to the vector $c$ of coefficients which will alter $f\left(a,b\right)$ additively by $a-mb$ times a uniformly random value on $\I\left(l\right)$.
\begin{rem}The key observation is that $S$ (and thus the sets $S_\ell$), projected onto any axis, is much smaller than the range of any of the entries $c_i$ as $c$ varies. As a corollary, we hope to show that the randomness implicit in $c$ will in fact cause $f_c\left(a,b\right)$ to be almost uniformly random over short intervals, as~\refp{rho-and-f} allows us to replace randomness of $c$ over cosets of $S_l$ with randomness of $f_c\left(a,b\right)$ over short arithmetic progressions.
\end{rem}
\begin{definition}
For $\bar{\mu} : X \rightarrow \R^{+}$ a measure and $F : X \rightarrow Y$ a function, we define a measure $F^{\bar{\mu}} : Y \rightarrow \R^{+}$ by:
$
F^{\bar{\mu}}\left(y\right) \defeq \bar{\mu}\left(\{F^{-1}\left(y\right)\}\right) = \sum_{x : F\left(x\right) = y}\bar{\mu}\left(x\right)
\Longrightarrow F^{\bar{\mu}}$ is the output distribution of $F$ when the input distribution is $\bar{\mu}$.
\end{definition}
\begin{definition}\label{def-F-max}
In any context where $a, b, d$ are fixed, we say
\[
F_{\max} \defeq L_n\left(\frac{2}{3}, \kappa-\delta^{-1}\right)\left(a-mb\right)\sum_{i = 0}^{d-1}a^ib^{d-1-i}.
\]
\end{definition}

We sketch the aims, methods and use of Lemma~\refp{intervalmult} and Lemma~\refp{f-is-smooth}. Recall that $\mu$ is a uniform distribution on a cube of side $2L_n\left(\frac{2}{3}, \kappa-\delta^{-1}\right)$. Furthermore, $\nu_l$ is uniform and has support bounded to a cube of side length $4L_n\left(\frac{1}{3}, \sigma\right)$. We will show that for almost every $v \sim \mu$, $\mu |_{v + S_l}$ is uniform and equal to $\mu\left(v\right)$. Heuristically, this holds as $v$ is at least $4L_n\left(\frac{1}{3}, \sigma\right)$ from the boundary of the support of $\mu$. We will then deduce that $\left(f-\hat{f}\right)\left(a,b\right)$ is close to uniform on short ranges of multiples of $\left(a-mb\right)$.

From this we will show that $\mu$ is not substantially altered (in the $\ell_1$ metric) by convolving it with the distributions $\nu_\ell$. Furthermore, the linearity of $\varphi$ implies that when it is applied to any ``reasonably smooth'' convolution involving $\nu_\ell$ the result is ``reasonably close'' to uniform on short intervals. Then in particular $\mu$ is close to $\mu \star \nu_\ell$, the latter being close to uniform on short progressions of common difference $\left(a-mb\right)$.

We use this convolution to formally show the heuristically obvious claim that the large random sum contributing to $f\left(a,b\right)$ does in fact make it close to uniformly random on short progressions. In fact, we will convolve with several distributions $\nu_{\ell_i}$, with each convolution allowing us (heuristically) to treat each coefficient in $f$ as if it were independent and uniformly random.

We begin by showing that $\varphi^{\mu}$ is close to a convolution of uniform distributions on intervals. The proof of this claim is an exercise in checking that the required convolution can be constructed by an additive kernel whose support is bounded to a cube of size much smaller than the randomness in our choice of $c$, and is not core to the intuitions of the proof of Lemma~\refp{f-is-smooth}. We place the proof here to collect the required results about $\varphi^{\mu}$ to a single place.
\begin{lemma}\label{intervalmult}
Fix $a,b$.
There is a distribution $\vartheta$ such that $\vartheta$ is the convolution of uniform distributions on intervals of lengths $L_n\left(\frac{2}{3}, \kappa-\delta^{-1}\right)a^ib^{d-1-i}$ for $i = 0$ to $d-1$ with:
\begin{align*}
||\varphi^{\mu} - \vartheta||_1 = \O\left(L_n\left(\frac{2}{3}, (\kappa-\delta^{-1})\left(1+\o\left(1\right)\right)\right)^{-1}\right),
\quad|\E(\vartheta)| \leq \sum_{i=0}^{d-1}a^ib^{d-1-i}.
\end{align*}
\end{lemma}
\begin{rem}
In the Randomised NFS, we will consider at most $L_n\left(\frac{1}{3}\right)$ polynomials $f$, and hence at most $L_n\left(\frac{1}{3}\right)$ samples of $\varphi^{\mu}$ for any fixed $a,b$. Note that here we bound the total variation by $L_n\left(\frac{2}{3}\right)^{-1}$. As a corollary the total variation between our sample from $\varphi^{\mu}$ and a sample from $\vartheta$ of the same length is $L_n\left(\frac{2}{3}\right)^{-1}$. Our desired probabilities for smoothness are $L_n\left(\frac{1}{3}\right)^{-1}$, so establishing events occur with this probability for $\vartheta$ guarantees that they occur with this probability for $\varphi^{\mu}$.
\end{rem}
\begin{proof} We denote the convolution of distributions by $\star$, and define:
\[
\nu \defeq \mu\star\left[\bigstar_{i= 0}^{d-1} \nu_{a^i b^{d-1-i}}\right].
\]

From Lemma~\refp{S-for-phi}, the support of $\nu_{a^ib^{d-i}}$ is contained in a cube of side $4L_n\left(\frac{1}{3}, \sigma\right)$. Hence the support $P$ of $\bigstar_{i= 0}^{d-1} \nu_{a^i b^{d-1-i}}$ is contained in a cube of side $4dL_n\left(\frac{1}{3}, \sigma\right)$.
When $||x||_\infty < L_n\left(\frac{2}{3}, \kappa-\delta^{-1}\right) - 4dL_n\left(\frac{1}{3}, \sigma\right)$ and $p \in P$:
\[
\mu(x - p) = \mu(x) = |\operatorname{supp}(\mu)|^{-1},
\]
so $\nu(x)$ is a convex combination of values in $\{\mu(x-p) : p \in P\} = \{\mu(x)\}$. Hence $\nu(x) = \mu(x)$ on the $l_\infty$ ball of radius $L_n\left(\frac{2}{3}, \kappa-\delta^{-1}\right) - 4dL_n\left(\frac{1}{3}, \sigma\right)$.
Then since $L_n\left(\frac{1}{3}, \sigma\right)$ is $L_n\left(\frac{2}{3}, \o\left(\kappa-\delta^{-1}\right)\right)$:
\begin{align*}
\P_{x \sim \mu}\left(\nu\left(x\right) = \mu\left(x\right)\right) &\geq \left(1-4dL_n\left(\frac{2}{3}, \kappa-\delta^{-1}\right)^{-1+\o\left(1\right)}\right)^{d}\\
&\geq 1 - 4d^2L_n\left(\frac{2}{3}, \kappa-\delta^{-1}\right)^{-1+\o\left(1\right)}
 = 1 - L_n\left(\frac{2}{3}, \kappa-\delta^{-1}\right)^{-1+\o\left(1\right)}, \\
\end{align*}
In particular, we have a bound on the $\ell_1$ distance between $\mu$ and $\nu$:
\begin{align*}
||\mu - \nu||_1 &= \sum_{x\in \Z^{d}}|\nu\left(x\right) - \mu\left(x\right)|
\leq \P_{x \sim \mu}\left(\nu\left(x\right) \neq \mu\left(x\right)\right) . \left(||\mu||_{\infty} + ||\nu||_{\infty}\right)\\
&= \O\left(L_n\left(\frac{2}{3}, (\kappa-\delta^{-1})\left(1+\o\left(1\right)\right)\right)^{-1}\right).
\end{align*}
Now for fixed $a, b$ we apply the map $\varphi$ to $\mu$ and $\nu$ to obtain:
\[
||\phi^{\mu} - \phi^{\nu}||_1 \leq ||\mu - \nu||_1 \leq
\O\left(L_n\left(\frac{2}{3}, (\kappa-\delta^{-1})\left(1+\o\left(1\right)\right)\right)^{-1}\right).
\]
and so the $\ell_1$ difference of the distributions $\phi\left(\mu\right)$ and $\phi\left(\nu\right)$ on $\Z$ is small.

Recall from~\refp{nu-l-def} that $\varphi^{\nu_l} = \U\left(\I\left(l\right)\right)$. Since applying our map $\varphi$ to a measure commutes with convolution of measures:
\[
\varphi^{\nu} = \varphi^{\mu} \star \left[\bigstar_{i=0}^{d-1}\U\left(\I\left(a^ib^{d-1-i}\right)\right)\right].
\]
Since $c_i \sim \U\left(\I\left(L_n\left(\frac{2}{3}, \kappa-\delta^{-1}\right)\right)\right)$ are independent random variables:
\begin{equation}\label{c+unif}
c_ia^i b^{d-1-i} +\U\left(\I\left(a^ib^{d-1-i}\right)\right)
\end{equation}
is uniformly distributed along an interval of length $L_n\left(\frac{2}{3}, \kappa-\delta^{-1}\right)a^ib^{d-1-i}$. Note that $\underline{c} \sim \mu$. Hence there is a constant $C$ such that for all $x \in \Z$:
\begin{equation*}
\varphi^{\nu} \left(x\right) \sim \bigstar_{i = 0}^{d-1}\left[\U\left(\I\left(L_n\left(\frac{2}{3}, \kappa-\delta^{-1}\right) a^i b^{d-1-i}\right)\right)\right] \left(x -C\right)
\end{equation*}
The shift $C$ accounts for the difference in expectation caused by the fact that the intervals associated with~\refp{c+unif} are not centred (recall Definition~\refp{centredinterval}). However, the centre of each of these intervals has modulus at most $\frac{1}{2}a^ib^{d-1-i} + \frac{1}{2}$, and so $|C| \leq \sum_{i=0}^{d-1}a^ib^{d-1-i}$. We take $\vartheta = \varphi^{\nu}$ to complete the proof of Lemma~\refp{intervalmult}.
\end{proof}
The convolution $\vartheta$ allows us to replace $R\left(a,b\right) = (a-mb)\varphi_{a,b}\left(c\right)$ by $a-mb$ times a convolution of uniform measures on intervals. In Lemma~\refp{f-is-smooth}, this will give us control over the distribution of $R(a,b)$ on progressions of common difference $a-mb$.

It remains to control $f\left(a,b\right) \imod{a-mb}$. 
In Section~\refp{NFS-uniform-chap}, we will characterise the moduli for which the smooth numbers are uniformly distributed across their residue classes, at which point the specific residue class of $f\left(a,b\right) \imod{a-mb}$ will not significantly affect its probability of being smooth as $c$ varies.

We now combine the previous claims to show that $f\left(a,b\right)$ is $B$-smooth as often as random integers of the same size. Note that if $\gcd\left(a,b\right)$ had been greater than one, then throughout we could have divided it out, and the probability of smoothness would be increased.

\begin{lemma}\label{f-is-smooth}
 Fix $a,b,m,n$ in their intervals and let $f$ be uniformly random as before. Then:
\begin{equation*}
\P_{f}\left(f\left(a,b\right)\textrm{ is }B'\textrm{-smooth} \mid \left(a-mb\right)\textrm{ is }B'\textrm{-good}\right) = L_n\left(\frac{1}{3}, \frac{\kappa +\sigma\delta}{3\beta'} \left(1+\o\left(1\right)\right)\right)^{-1}.
\end{equation*}

\end{lemma}
In the subsequent, $\delta^{-1} - \kappa$ controls the exponent of the error terms in several uniformity claims: for this reason we imposed the condition $\kappa > \delta^{-1}$ in Equation~\refp{const-bounds}.

\begin{proof}
Let $a-mb = r$. Recalling Lemma~\refp{intervalmult}:
\begin{multline*}
\P_{n,f}\left(f_{n,m}\left(a,b\right)\textrm{ is }B'\textrm{-smooth}\right)
= \P_{n,c}\left(\hat{f}_{n,m}\left(a,b\right) + r\phi_{a,b}\left(c\right) \textrm{ is }B'\textrm{-smooth}\right) \\
= \P_{n,\vartheta}\left(\hat{f}_{n,m}\left(a,b\right) + r\vartheta \textrm{ is }B'\textrm{-smooth}\right)
+ \O\left(L_n\left(\frac{2}{3}, (\kappa-\delta^{-1})\left(1+\o\left(1\right)\right)\right)^{-1}\right)
\end{multline*}
Recall that $\vartheta$ has
$
|\E(\vartheta)| \leq \sum_{i=0}^{d-1}a^ib^{d-1-i}
$
and is sampled according to the convolution of uniform measures on intervals of length $L_n\left(\frac{2}{3}, \kappa-\delta^{-1}\right)a^ib^{d-1-i}$ for $i = 0, \ldots, d-1$. Hence $\vartheta$ is unimodal with mode at some $M$ satisfying
\[
|M| \leq \sum_{i=0}^{d-1}a^ib^{d-1-i} < db^{d-1} = L_n\left(\frac{2}{3}, \sigma\delta(1+\o(1))\right),
\]
and the support of $\vartheta$ is contained in $[M - F_{\max}|r|^{-1}, M + F_{\max}|r|^{-1}]$. 
We choose an $\omega = L_n\left(\frac{2}{3}, \o(1)\right)$, such that $\omega \rightarrow \infty$, and set
\[
Y \defeq L_n\left(\frac{2}{3}, \kappa-\delta^{-1}\right)b^{d-1} \omega^{-1} = L_n\left(\frac{2}{3}, \kappa-\delta^{-1} + \sigma\delta-\o(1)\right).
\]
Now, we define a measure $\vartheta'$ to be
\[
\vartheta'(x) \defeq \left\{
\begin{aligned}
&\vartheta\left(\max\left(x, Y \right)\right)&\quad x \geq 0\\
&\vartheta\left(\min\left(x, -Y \right)\right)&\quad x < 0\\
\end{aligned}
\right.
\]
Later, we will see that using this measure allows us to control the density of smooth numbers only on progressions of length at least $Y$. Then:
\begin{align*}
||\vartheta'-\vartheta||_1 &\leq \P_{z \sim \vartheta}(|z| < Y) \leq 2Y\left(L_n\left(\frac{2}{3}, \kappa-\delta^{-1}\right)b^{d-1}\right)^{-1} \\
 &= 2\omega^{-1},
\end{align*}
from the definition of $Y$. Note that $Y$ is much larger than $M$ and so $\vartheta'$ is monotone decreasing away from $0$; hence there are non-negative weights $W_y$ for $y \in \Z$, with $W_y = 0$ for $|y| > F_{\max}|r|^{-1}$ such that:
\[
\vartheta' = \sum_{y \geq Y} W_y \U([0, y]) + W_{-y}\U([-y, 0))
\]
and $\left|1 - \sum_y W_y\right| \leq 2\omega^{-1}$. Hence we have:
\begin{multline*}
\P_{f}\left(f_{n,m}\left(a,b\right)\textrm{ is }B'\textrm{-smooth}\right) = \O\left(\omega^{-1}\right) +
\sum_{y = Y}^{F_{\max}|r|^{-1}} W_y\P\left(\hat{f}_{n,m}(a,b) + r \U\left([0,y]\right)\textrm{ is }B'\textrm{-smooth}\right) \\
+ W_{-y}\P\left(\hat{f}_{n,m}(a,b) + r \U([-y, 0))\textrm{ is }B'\textrm{-smooth}\right)
\end{multline*}

We note that $\O(\omega^{-1}) = L_n(\frac{2}{3}, \o(1))^{-1}$ terms can be absorbed into our $\o(1)$ terms, and so it suffices to show that for any fixed, $B'$-good $r$ and any $y \in [Y, F_{\max}|r|^{-1}]$:
\begin{align*}
\P\left(\hat{f}_{n,m}(a,b) + r \U\left([0,y]\right)\textrm{ is }B'\textrm{-smooth}\right)= L_n\left(\frac{1}{3}, \frac{\kappa +\sigma\delta}{3\beta'} \right)^{-1+\o(1)}, \\
\P\left(\hat{f}_{n,m}(a,b) + r \U\left([-y, 0)\right)\textrm{ is }B'\textrm{-smooth}\right)= L_n\left(\frac{1}{3}, \frac{\kappa +\sigma\delta}{3\beta'} \right)^{-1+\o(1)}.
\end{align*}
Since $|\hat{f}_{n,m}(a,b)| \leq \hat{F}_{\max} \defeq Y L_n\left(\frac{2}{3}\right)^{-1}$, we can absorb the probability that the value on the left is negative or positive (respectively) in the above two equations. From the definition of $B'$-good and Corollary~\refp{psi-q-psi}, for any $x \in [|r|Y-\hat{F}_{\max}, F_{\max}+\hat{F}_{\max}]$:
\[
\Psi\left(x,B',r,s\right) = \frac{\Psi_r\left(x,B'\right)}{\phi(r)} L\left(\frac{1}{3}, \o(1)\right) = \frac{\Psi\left(x,B'\right)}{r} L\left(\frac{1}{3}, \o(1)\right)
\]
and so to finish the estimate we observe that for any $x \in [|r|Y-\hat{F}_{\max}, F_{\max}+\hat{F}_{\max}] $:
\[
\rho(x, B') = \rho\left(L_n\left(\frac{2}{3}, \kappa + \sigma\delta\right), B'\right) = L_n\left(\frac{1}{3}, \frac{\kappa +\sigma\delta}{3\beta'} \right)^{-1+\o(1)}. \qedhere
\]

\end{proof}

We state the following Lemma that we will prove in Section~\refp{smooth-bad-dist-proof}.
\begin{lemma}\label{smooth-bad-dist}
Fix any $b$. Then
\[
\P_{a,m}\left(a-mb\textrm{ is }B'\textrm{-good} \mid a-mb\textrm{ is }B\textrm{-smooth}\right) = 1 - \o\left(1\right)
\]
\end{lemma}
We are now able to prove Theorem~\refp{NFS-X-large}
\begin{proof}[Proof of Theorem~\refp{NFS-X-large}]\label{NFS-X-large-pf}
Lemma~\refp{a-mb-is-smooth} and Lemma~\refp{smooth-bad-dist} randomise over $a,m$ for any fixed $b$, and uniformly over $n, f$. Hence for any $b, n, f$:
\[
\P_{a, m}\left(a-bm \text{ is }B\text{-smooth and }B'\text{-good}\right) = L_n\left(\frac{1}{3}, \frac{\delta^{-1}}{3\beta}(1+\o(1))\right)^{-1}
\]
Since Lemma~\refp{f-is-smooth} randomises over $f$ for any fixed $a,b,m$, we have for each fixed $b$:
\[
\P_{a, m, f}\left(a-bm \text{ is }B\text{-smooth and }B'\text{-good}, f(a, b) \textrm{ is }B'\text{-smooth}\right) = L_n\left(\frac{1}{3}, \frac{\delta^{-1}}{3\beta} + \frac{\kappa + \sigma\delta}{3\beta'}\right)^{-1 + \o(1)}
\]
as multiplicative factors of $1+\o(1)$ may be absorbed into the $\o(1)$ in the exponent of the $L_n\left(\frac{1}{3}\right)$ terms. Summing over the $L_n\left(\frac{1}{3}, \sigma\right)^2$ choices for a fixed pair $\left(a,b\right)$:s
\begin{align*}
\E_{m,f}\left(|\X_{n,m,f}|\right) &= \sum_{a,b} \P_{n,m,f}\left(\left(f,n,m,a,b\right) \in \X\right) 
= L_n\left(\frac{1}{3}, \sigma\right)\sum_{b}\P_{n,m,f,a}\left(
\begin{gathered}\left(a-bm\right)\textrm{ is } B\textrm{-smooth} \\
\wedge f\left(a,b\right)\textrm{ is }B'\textrm{-smooth}\end{gathered}
\right) \\
&\geq L_n\left(\frac{1}{3}, \sigma\right)\sum_{b}\P_{n,m,f,a}\left(
\begin{gathered}\left(a-bm\right)\textrm{ is } B\textrm{-smooth} \wedge (a-bm)\text{ is }B'\textrm{-good} \\
\wedge f\left(a,b\right)\textrm{ is }B'\textrm{-smooth}\end{gathered}
\right) \\
&\geq L_n\left(\frac{1}{3}, 2\sigma - \left(\frac{\delta^{-1}}{3 \beta}\right)\left(1+\o\left(1\right)\right) + \left(\frac{\sigma\delta+ \kappa}{3 \beta'}\right)\left(1+\o\left(1\right)\right)\right)\qedhere
\end{align*}
\end{proof}

\section{Controlling Algebraic Obstructions to Squares and the Proof of Theorem~\ref{NFS-characters-uncond}}\label{algebraic-squares}

We begin with some high-level discussion. At the end of Step~\refp{smooth-generate} of the algorithm, we have a large collection of linear polynomials $a-Xb$ which, when sent to $\Z[\alpha]$ or $\Z$ by morphisms sending $X$ to $\alpha$ or $m$ respectively, are \emph{smooth normed} in both rings. Recall that in Step~\refp{square-formation}, we seek to find a subset of these elements whose product is sent to the square of an element of $\Z[\alpha]$ and a square in $\Z$ by these two morphisms.

Now, if we are given an element $z \in \Z$ and asked whether it is square, we need only check that for any prime $r$ dividing $z$, the multiplicity of $r$ as a factor of $z$ is \emph{even}. In this situation, we can halve the order of every prime and take a product to yield another integer whose square will be $z$. Hence given the factorisations of the images $a-mb$ in $\Z$ for $1+B$ polynomials found in Step~\refp{smooth-generate}, we can find a subset whose product is square by looking for a subset such that the total multiplicity of every prime less than $B$ across the subset is even. We can send each $a-mb$ to a vector over $\F_2$ of the orders of primes dividing $a-mb$; then the process of square formation is exactly finding an element in the kernel over $\F_2$ of a large matrix of exponents.

We might na\"{i}vely hope that we can follow this algorithm in $\Z[\alpha]$, by factoring the norms of $a - b\alpha$ to ensure that we find a subset whose product is square in both $\Z$ and $\Z[\alpha]$. However, over $K \defeq \Q\left(\alpha\right)$ and its ring of integers $\mathcal{O}_K$, this is more subtle, but the essential idea still works.

Note that $\mathcal{O}_K$ is a Dedekind domain, so non-zero prime ideal is maximal and so $\mathcal{O}_K / \mathfrak{p}$ is a field, say $\F_{r^k}$. Hence $N(\mathfrak{p}) = r^k$, and $\mathfrak{p} | (r)$ the ideal generated by $r$ in $\mathcal{O}_K$. Such a prime $\mathfrak{p}$ is said to be of \emph{$k$-th degree}.
The quotient map $\mathcal{O}_{K} \rightarrow \mathcal{O}_{K} / \mathfrak{p} \simeq \F_{r^k}$ is determined entirely by its action on $\alpha$. Hence we can identify the prime ideal $\mathfrak{p}$ with an element of $\F_{r^k}$, which is in turn identified with a minimal (and thus irreducible) polynomial $p_{\mathfrak{p}}$ over $\F_r$ of degree $k$. Note that we can apply the \emph{same} map by recalling that $\mathcal{O}_K$ is a subring of $\Q(\alpha)$, which may be quotiented by $(p_{\mathfrak{p}}(\alpha))$, or more explicitly $\mathcal{O}_K \ni g(\alpha) \rightarrow(g \mod p_{\mathfrak{p}})(\alpha)$ which preserves the representation of any element as a ratio of polynomials in $\alpha$.

On the other hand, it remains to see which polynomials correspond to primes. Suppose we are given a polynomial $p$ of degree $k$. It is plain that if the polynomial $\gcd(f, p) = 1$ over $\F_r$, then the quotient of $\mathcal{O}_K \subseteq K$ by $(p(\alpha))$ sends every element to $0$, and hence the ideal is not prime. Since $p$ is irreducible over $\F_r$, a non-trivial gcd implies that $p$ is one of the irreducible factors of $f \imod{r}$. Furthermore, the image of $\Z[\alpha]$ under the quotient map is plainly surjective. So we can identify this polynomial with the quotient map, and hence with the associated prime ideal $\mathfrak{p}$.

We can equate prime ideals $\mathfrak{p} \subset \mathcal{O}_K$ with pairs of a prime $r \in \Z$ and an irreducible factor $p_{\mathfrak{p}}$ of $f \imod{r}$. The latter representation will be substantially more straightforward to handle computationally. Furthermore, we note the particular ease of use of the \emph{degree one} primes, which correspond to simple roots of $f \imod{r}$. For these primes, the quotient map applied to a polynomial in $\Z[X]$ is mere evaluation at the root. In what follows, we will routinely abuse notation to equate the prime ideal $\mathfrak{p}$ in $\mathcal{O}_{K}$ and the irreducible polynomial divisor $p_{\mathfrak{p}}$ of $f\left(x,1\right)\imod{r}$. We will also equate the ideal $\mathfrak{p}$ with the pair $\left(r,s\right)$, with $r$ a modulus and $s$ a root of $\mathfrak{p}$ in $\F_{r^k}$ when $r$ is prime.

We note that, unlike the situation in $\Z$, there may be multiple prime ideals of the same \emph{norm}, since for a prime $r \in \Z$ the ideal $\left(r\right)$ may lift to an ideal $\left(r\right) \subseteq \mathcal{O}_{K}$ which is not a power of a single prime ideal. However, this is not a substantial problem, as the norm of the ideal $\left(r\right)$ in $\mathcal{O}_{K}$ is the greatest common divisor of the norm of each element of $(r)$, and so divides $\mathbf{N}(r) = r^d$. Since the norm of a prime ideal is an integer exceeding $1$, and norms are multiplicative, the number of prime ideals dividing $(r)$ in the ring of integers is bounded above by $d \ll \log_2 n$

Of course more is known; it is a result of Landau~\cite{LandauIdeal} that the number of prime ideals in $\mathcal{O}_{\Q\left(\alpha\right)}$ of norm less than $x$ is:
\[
\frac{x}{\log x} + x\exp\left(-\O_\alpha\left(\sqrt{\log x}\right)\right).
\]
As should be expected, the dependence on $\alpha$ in this bound in fact driven by the position of a (hypothetical) troublesome zero of the zeta function associated to the field extension $K / \Q$; Montgomery and Vaughan~\cite{MontgomeryVaughan} have a substantial discussion. We could use the fact that we have taken $f$ to be random to gain better control of the number of ideals, but as $\log n = L_n\left(\frac{1}{3}, \o\left(1\right)\right)$ already we do not need the sharper bounds.

More subtly, since $\mathcal{O}_{K}$ need not be a unique factorisation domain as we have no guarantee that irreducible elements are in fact prime, and it might be the case that the number of irreducibles of small norm is much larger than the number of primes. It is also difficult to work directly with primes in the full number field, since they generally will not be in $\Z[\alpha]$.

This obstacle is standard in the family of NFS algorithms, and the methods first suggested by Adleman~\cite{Adleman} and studied in detail by Buhler, Lenstra and Pomerance~\cite{BuhlerLenstraPomerance} allow us to avoid it. They remark that a complete analysis of these characters was out of reach, and suggest that much stronger versions of the Chebotarev Density Theorem might be required. We instead proceed to show that for our randomised field and with a stochastic collection of characters with large conductor, the number of ways in which an element might appear square and yet not be is small enough that we can apply the pigeonhole principle to find a square.

Our first task is to keep track of the ways in which a given prime $r$ might come to divide $f_d\mathbf{N}\left(a-b\alpha\right)$. In particular we observe that:
\begin{align*}
&r \mid f_d\mathbf{N}\left(a-b\alpha\right) = f(a,b)
&\Rightarrow &f(a,b) = 0 \imod{r}\\
\Rightarrow &r | b \text{ or } f(ab^{-1}, 1) = 0 \imod{r}
&\Rightarrow &r | b \text{ or } \exists s : \left(r,s\right) = 1, f\left(s, 1\right) \equiv 0 \imod{r}
\end{align*}
and that furthermore if $r | b$ then $r | f_d\mathbf{N}(a) = f_da^d$, and so $r | f_da$. In this situation we can note that $f_d(a-b\alpha)$ is divisible by every prime ideal lying over $(r)$, and so we can assume that $r$ does not divide $b$. Hence we split each prime $r < B'$ into a collection of ``primes'' $\left(r,s\right)$, one for each $0 < s < r$ coprime to $r$ with $f\left(s,1\right) \equiv 0 \imod r$.

Number theoretically, these correspond to the \emph{first degree} primes in $\mathcal{O}_{K}$:
\[
(r, s) : r\text{ prime}, r \mid f(s,1) \text{ are in correspondence with } \mathfrak{p} \mid (r), \mathbf{N}(\mathfrak{p}) = r
\]
These are particularly convenient, as the norm of the ideal generated by these prime ideals is a prime in $\Z$; as a corollary, working modulo $\mathfrak{p}$ entails mapping $\alpha$ into an element of $\Z/r\Z$ rather than $\mathbf{F}_{r^k}$. In particular, we define the following functions (after~\cite{BuhlerLenstraPomerance})
\[
e_{r,s}\left(a-b\alpha\right) \defeq \ord_r\left(f\left(a,b\right)\right) \1_{a \equiv bs \imod{r}}
\]
and note that this apportions the responsibility for the divisibility of $f\left(a,b\right)$ by $r$ to a specific solution $s$ of $f\left(s,1\right) \equiv 0 \imod r$.

Note that there are at most $d$ solutions to $f\left(s,1\right) \imod{r}$ and again $d = \log^{\frac{1}{3} + \o(1)}n$ which is much smaller than $L_n\left(\frac{1}{3}\right)$. As mentioned, that these $e_{r, s}$ correspond to the splitting of first degree primes in $\Z[\alpha]$ dividing $\left(r\right)$, and so $e_{r,s}$ extends to a linear map from the multiplicative semigroup of $K^\times$ to $\Z$~\cite[Lemma 5.5]{BuhlerLenstraPomerance}.

Hence given $1 + B + dB'$ polynomials from Step~\refp{smooth-generate} we can use linear algebra over $\F_2$ to find a subset product $P$ such that $P\left(m\right) \in \Z$ is square and $P\left(\alpha\right) \in \Z[\alpha]$ is such that $ 2 \mid e_{r,s}\left(P\left(\alpha\right)\right)$. It remains to show that extending this linear algebra can force $P\left(\alpha\right)$ to be the square of an element of $\Z[\alpha]$

We will first show that the number of ways that we can fail to produce a square in $K$ is controlled by an $\F_2$ vector space (denoted $H$) of small dimension. We will then randomly construct a multiplicative map (denoted $\Psi_\mathcal{F}$) which almost surely distinguishes all of the elements of $H$. In particular this allows us to identify when a product is a square of an element of $\Q\left(\alpha\right)$, once we know it to be an element of $\mathcal{O}_{K}$ with square and smooth norm.

This map $\Psi_\mathcal{F}$ will be multiplicative, it will be a \emph{linear} function of the order of each prime dividing $a - b\alpha$. As a corollary, we can use additional sieving to find a subset whose product maps to a square in $\Z$ and $\mathcal{O}_{K}$ and such that $\Psi_\mathcal{F}$ shows the product in the number field to be a square of an element of $\Q\left(\alpha\right)$. In particular, $\Psi_\mathcal{F}$ will be a collection of a logarithmic number of random quadratic characters on the number field. To force the square to in fact be a square of an element of $\Z[\alpha]$ requires that we multiply by an additional constant.

We note that whilst this general approach is standard, the details of our method will be somewhat different. In particular, the standard NFS produces the map $\Psi_F$ by taking a collection of maps corresponding to first degree primes $\mathfrak{p}$ lying over primes $\left(p\right)$ in $\Z$ which are just above the smoothness bound $B'$. By contrast, we will take arbitrary primes $\mathfrak{p}$ of norm below a much larger bound, in general, we will have $\log\left(\mathbf{N}\left(\mathfrak{p}\right)\right)$ being $L_n\left(\frac{1}{3}\right)$.

To show this in detail, we will have to study various extensions of $\Q\left(\alpha\right)$, corresponding precisely to adjoining roots of elements which fail to be square in the ring of integers. In particular, the standard bounds on the discriminant of $\Q\left(\alpha\right)$ extends to similar bounds on the discriminant of the quadratic extensions of interest, and we use effective results of Stark~\cite{Stark} to show that the majority of such extensions have no Siegel zero. This allows us to show that for characters of suitably large conductor, the kernel of $\Psi_{\mathcal{F}}$ is small enough that it can be handled by brute force.

We now begin the formal argument. We implicitly equate $C_2$ and the additive group of $\F_2$ (via the map $\left(-1\right)^b \rightarrow b$). Recall that $\alpha$ has minimal polynomial $f\left(x,1\right)$, $\mathbf{N}$ is the field norm on $\Z[\alpha]$ and $K \defeq \Q(\alpha)$. We define a group:
\[
H \defeq \{z \in K^\times : \forall s<r,\; e_{r,s}\left(z\right)\equiv 0 \imod{2}\} / \{z^2 : z \in \Q\left(\alpha\right)^\times\}.
\]
\begin{lemma}\label{H-dim}
$H$ is an $\F_2$ vector space of dimension at most
\[
\left(\delta\kappa + \o\left(1\right)\right)\log_2 n + \frac{\delta^2\kappa}{2\log2}\frac{(\log n)^{4/3}}{(\log\log n)^{1/3}}
\]
\end{lemma}
\begin{rem}
The $\log^{4/3 + \o(1)} n$ term does not appear in the case that $f$ is monic, and thus is not in the standard presentation of the NFS. More generally, the term is $\O(d^2 \log f_d)$, and is being driven by the increased coefficients in the minimal polynomial for an algebraic integer in $\Q(\alpha)$.
\end{rem}
\begin{proof}
The coefficients of $f$ are bounded by $L_n\left(\frac{2}{3}, \kappa\right)$ (whereas in the standard NFS the bound is $m = L_n\left(\frac{2}{3}, \delta^{-1}\right)$). 
Recall that the degree $d$ of $f$ is $\delta\sqrt[3]{\frac{\log n}{\log \log n}}$.

To bound $|H|$, we follow Buhler, Lenstra and Pomerance's presentation of the NFS, using Lemma 3.3 and the argument of Theorem 6.7 from \cite{BuhlerLenstraPomerance}. We differ firstly in that their claims are restricted to the case $\kappa = \delta^{-1}$, but the arguments are plainly seen to be more general. To implement the more general case, we keep the dependence on $\Delta$ explicit. We observe that in \cite{BuhlerLenstraPomerance} the argument is given for a univariate non-homogeneous polynomial, which in the notation of this paper is $f(x, 1)$. Note also that in this paper, we cannot guarantee that $\alpha$ is an algebraic integer, although $f_d\alpha$ is.

\begin{rem} We note that the result in \cite[Lemma 3.3]{BuhlerLenstraPomerance}, claims a bound of form $d^{2d}n^2M^{-3}$ in the setting $\delta = \kappa^{-1}$. The argument presented there does not clearly hold as $f'_{d-1} = (d-1)f_{d-1}$, but the ratio of the matching terms in the first column is $d$ and so simply subtracting the first column from the second cannot cause all entries in the second column to be of order $1$.
\end{rem}	
\begin{claim}If the coefficients of $f$ are bounded by $M = L_n\left(\frac{2}{3}, \kappa\right)$ then the discriminant $\Delta_f$ of $f$ is bounded by
$
|\Delta_f| \leq d^{2d} n^{2\delta \kappa} M^{-2}
$
\end{claim}
\begin{proof}
For $f(x, 1) = \sum f_i x^i$, we have that $|f_d\Delta_f|$ is the resultant of $f(x, 1)$ and $\frac{d}{dx}f(x,1)$. Let $f'_i = if_i$. We define the associated $(2d-1) \times (2d-1)$ Sylvester matrix:
\[
S = \begin{bmatrix}
f_d & f_{d-1}  & \cdots & \cdots & f_1 & f_0 & 0 & \cdots & 0 \\
0 & f_d & f_{d-1} & \cdots & \cdots & f_1 & f_0 & 0 & \vdots \\
\vdots & 0 & \ddots & \ddots & & & \ddots & \ddots & 0\\
0 & \cdots & 0 & f_d & f_{d-1} & \cdots &\cdots & f_1 & f_0\\
f'_d & f'_{d-1} & \cdots & f'_2 & f'_1 & 0 & \cdots & \cdots & 0 \\
0 & f'_d & f'_{d-1} & \cdots & f'_2 & f'_1 & 0 & \cdots&\vdots \\
0 & 0 &f'_d & f'_{d-1} & & \ddots & \ddots & \ddots&\vdots \\
\vdots & & \ddots & \ddots & \ddots & & \ddots & \ddots & 0\\
0 & \cdots & \cdots & 0 & f'_d & f'_{d-1} & \cdots & f'_2 & f'_1\\
\end{bmatrix},
\]
with $|\Delta| = |\operatorname{det}(S)|f_d^{-1}$. 
We modify $S$ by subtracting $f'_{d-i}/f'_d$ times the first column from each of the later columns to obtain $S'$. By construction $\operatorname{det}(S) = \operatorname{det}(S')$. The first row of $S'$ has non-zero entries $(f_d, \frac{1}{d}f_{d-1}, \frac{2}{d}f_{d-2}, \ldots, f_0)$, and so the euclidean norm of the first row of $S'$ is bounded by:
\[
\left(M^2 + M^2 \frac{d(d+1)(2d+1)}{6d^2}\right)^{\frac{1}{2}} = \left(\frac{M^2d}{3}\right)^{\frac{1}{2}} \sqrt{1 + \frac{9}{2d} + \frac{1}{2d^2}} 
 \leq \left(\frac{M^2d}{3}\right)^{\frac{1}{2}} \exp\left(\frac{9}{4d} + \frac{1}{4d^2}\right)
\]
Similarly, the norm of rows $2$ through $d-1$ of $S'$ are bounded by
\[
\left(M^2 + M^2d\right)^{\frac{1}{2}} = Md^{\frac{1}{2}} \sqrt{1 + \frac{1}{d}} \leq Md^{\frac{1}{2}} \exp\left(\frac{1}{2d}\right)
\]
and the norm of rows $d+1$ through $2d-1$ of $S'$ are bounded by
\[
\left(M^2d^2 + M^2\frac{d(d-1)(2d-1)}{6}\right)^{\frac{1}{2}} = \left(\frac{M^2d^3}{3}\right)^{\frac{1}{2}} \sqrt{1 + \frac{3}{d} + \frac{1}{2d^2}}
\leq \left(\frac{M^2d^3}{3}\right)^{\frac{1}{2}} \exp\left(\frac{3}{2d} + \frac{1}{4d^2}\right).
\]
The $d^{\text{th}}$ row has only one non-zero entry and norm $df_d$. Now Hadamard's bound provides that $|\operatorname{det}(S')|$ is at most the product of the norms of the rows of $S'$, and so:
\begin{align*}
|\operatorname{det}(S')| &\leq \left(Md^{1/2}\right)^{d-1} df_d \left(Md^{3/2}\right)^{d-1} 3^{-\frac{d}{2}} \exp\left(\frac{8d-1}{4d}+\frac{d}{4d^2}\right) \\
&= f_dM^{2d-2} d^{2d} \left(3^{-\frac{d}{2}} e^2 d^{-1}\right)
\end{align*}
Note that since $d \geq 3$, the product of the last three terms is bounded above by $e^{2} 3^{-5/2} < 1$. We have $M^d = n^{\delta\kappa}$, and hence:
\[
|\Delta_f| = |\det(S')|f_d^{-1} \leq d^{2d}n^{2\delta\kappa}M^{-2}.\qedhere
\]
\end{proof}

Let $g$ be the minimal polynomial of $f_d\alpha$. Then clearly $g(x) = \sum_i (f_i f_d^{d-i-1}) x^i$, and so
\[
|\Delta_g| = |\Delta_f| f_d^{d(d-1)} = L_n\left(\frac{4}{3}, \delta^2\kappa + \o(1)\right).
\]

\begin{claim}
$
|H| \leq \sqrt{|\Delta_g|} (\log n)^{\O(d)}\qquad\text{\cite[Theorem 6.7]{BuhlerLenstraPomerance}}.
$
\end{claim}
\begin{proof}
We follow the presentation of \cite[Theorem 6.7]{BuhlerLenstraPomerance}, differing only in that we track the dependence on $\Delta$ precisely.
We define:
\begin{align*}
V &= \{z \in K^\times : \forall s<r,\; e_{r,s}\left(z\right)\equiv 0 \imod{2}\},\\
W &= \left\{\gamma \in K^\times : \gamma\mathcal{O}_{K} = \mathfrak{a}^2, \mathfrak{a}\text{ a fractional }\mathcal{O}_{K}\text{-ideal}\right\},\\
Y &= \mathcal{O}_{K}^\times{K^\times}^2
\end{align*}
Note that $V \supset W \supset Y \supset {K^\times}^2$ and $|H| = [V : {K^\times}^2]$.
Now, \cite[Proposition 7.4]{BuhlerLenstraPomerance} gives that:
\[
[V : W] \leq [\mathcal{O}_{K} : \Z[f_d\alpha]].
\]
Additionally, if the order of the ideal class group of $\mathcal{O}_{K}$ is $h$, then:
\[
[W : Y] \leq h,
\]
as (using the notation of the definition of $W$) for any $\gamma \in W$ the map sending $\gamma$ to the ideal class of $\mathfrak{a}$ has $Y$ as its kernel. If $K$ has $2s$ complex embeddings, then Dirichlet's unit theorem implies that:
\[
[Y : {K^{\times}}^2] = 2^{d-s}
\]
since $Y / {K^{\times}}^2 \simeq \mathcal{O}_{K}^{\times} / {\mathcal{O}_{K}^{\times}}^2$. As in~\cite{BuhlerLenstraPomerance} we define the Minkowski constant $M_K$:
\begin{align*}
M_K \defeq \frac{d!}{d^d}\left(\frac{4}{\pi}\right)^s \sqrt{|\Delta_K|}
\leq \sqrt{|\Delta_K|}
\end{align*}
with the inequality following from $s \leq \lfloor\frac{d}{2}\rfloor$ and Stirling's approximation. From \cite[Chapter III, Proposition 8 and 14]{Lang} and the definition of polynomial discriminants, it is immediate that $\sqrt{\Delta_K} [\mathcal{O}_K : \Z[f_d\alpha]] = \sqrt{\Delta_g}$.
Now, from \cite[Theorem 6.5 and Remark]{Lenstra-Algo}:
\[
h \leq M_K . \frac{(d-1 + \log M_K)^{d-1}}{(d-1)!}
\]
Recall that $\log\left(|\Delta|\right) = \O(\log n)$ and $d = \o(\log n)$. Hence:
\begin{align*}
|H| &= [V :{K^{\times}}^2] \leq [\mathcal{O}_{K} : \Z[f_d\alpha]] h 2^{d-s} \\
&\leq [\mathcal{O}_{K} : \Z[f_d\alpha]]  \sqrt{|\Delta_K|} \frac{(d-1 + \log \sqrt{|\Delta_K|})^{d-1}}{(d-1)!} 2^{d-s} \\
&\leq \sqrt{|\Delta_g|}(d-1 + \log \sqrt{|\Delta_g|})^{d-1}d^{\O(d)}
\leq \sqrt{|\Delta_g|}\left(\log n \right)^{\O(d)}.\qedhere
\end{align*}
\end{proof}
We now finish proving Lemma~\refp{H-dim}. Since $\left(\log n\right)^{\O(d)} = n^{\o(1)}$, we use the above two results:
\[
|H| \leq n^{\delta\kappa + o\left(1\right)} f_d^{d(d-1)/2}
\]
Note that $f_d \leq L_n\left(\frac{2}{3}, \kappa\right)$ and that $d = \delta \log^{1/3}n (\log \log n)^{-1/3}$. Hence
\[
\log_2 |H| \leq \left(\delta\kappa+\o\left(1\right)\right)\log_2 n + \frac{\delta^2\kappa}{2\log2}\frac{\log^{4/3}n}{\log\log n^{1/3}}.
\]
Since $K^\times$ is commutative, any element of $H$ can be represented as a coset
$
h . \{z^2 : z \in K^\times\}.
$
Hence the square of any element of $H$ is in fact the identity element, since it is equivalent to $h^2 \{z^2 : z \in K\}$ and $h \in K^\times$. Thus $H$ is naturally an $\F_2$ vector space, and $v \in (K^\times)^2$ equivalent to $v\rightarrow 0$ under projection to $H$.
\end{proof}

\subsection{Characters over the number field}
We now discuss the construction of our characters $\chi_{\mathfrak{p}}$. Observe that quadratic characters on $\Z[\alpha]$ are well defined as maps from $H$, as they are multiplicative and so are trivial on any square in $\Z[\alpha]$. We restrict our attention to characters induced by the quadratic character on some finite field. We recall our previous discussion of the prime ideals, which allow us to characterise all of the maps from $\mathcal{O}_{\Q\left(\alpha\right)}$ to finite fields. In particular, on terms of the form $\left(a - \alpha b\right)$, such characters have the form:
\begin{equation}\label{eq:chi-def}
\left(a - \alpha b\right) \overset{\chi_\mathfrak{p}}{\longmapsto} \left(a - bX\right)^{\frac{1}{2}\left(r^k - 1\right)} \in \F_r[X]/\left(p_{\mathfrak{p}}\right) \simeq \F_{r^k}
\end{equation}
where $p_{\mathfrak{p}}$ is an irreducible polynomial of degree $k$ dividing  $f\left(x,1\right)\imod{r}$ exactly once. We note that as $\F_{r^k}^\times$ is cyclic, this map in fact sends every pair $\left(a,b\right)$ to $\pm 1$ or $0$, and is thus a quadratic character. Furthermore, we recall that $\mathfrak{p} \simeq \left(r, p_{\mathfrak{p}}\left(\alpha\right)\right)$ is a prime ideal of degree $k$ in $\mathcal{O}_{\Q\left(\alpha\right)}$ dividing $\left(r\right)$.

We note that in fact this representation of the character is computationally challenging, as it requires exponentiation. Instead, it is more convenient to observe that the above is:
\begin{equation}\label{eq:chi-def2}
\left(a - \alpha b\right) \overset{\chi_\mathfrak{p}}{\longmapsto} \left(\frac{a - bX}{p_{\mathfrak{p}}(X)}\right)
\end{equation}
where the right-hand side is the Legendre symbol over $\F_r[X]$.

It is natural to think of searching for $\mathfrak{p}$ by seeking to factorise $f\left(x,1\right) \imod{r}$ and examining the irreducible divisors. Given a set $\mathcal{F}$ of these $\chi_\mathfrak{p} = \chi_{r,s}$, we define
\[
\Psi_\mathcal{F} : H \rightarrow \F_2^{|\mathcal{F}|},\quad x \overset{\Psi_\mathcal{F}}{\longmapsto} \left(\chi_{r,s}\left(x\right) : \chi_{r,s} \in \mathcal{F}\right).
\]
We will produce a random set $\mathcal{F}$ such that almost surely $\ker\left(\Psi_\mathcal{F}\right)$ is small. 

\begin{lemma}\label{prime-sampling}
There is a sampleable distribution $\Upsilon$ for pairs $r,s$, such that $\chi_{r,s}$ is a character following~\refp{eq:chi-def}, such that for all but $\log \log n$ of the $h \in H$, considering $\chi_{r,s}$ as a map from $H$ to $\F_2$:
\[
\P_{\Upsilon}\left(\chi_{r,s}\left(h\right) = -1\right) \geq \frac{1+\o\left(1\right)}{2}.
\]
Sampling according to $\Upsilon$ takes at most $L_n\left(\frac{1}{3}, c\right)$ time for $c$ to be defined later. Furthermore, each character $\chi_{r,s}$ can be evaluated in time at most $L_n\left(\frac{1}{3}, \frac{c}{2}\right)$.
\end{lemma}
\begin{rem}
We will in fact achieve this unconditionally with $c = \frac{4}{3}\delta + \o(1)$. Conditional on GRH these $L_n(\frac{1}{3})$ bounds become polynomial in $\log n$. We observe that formal guarantees of this form are not present in the literature.

The heuristic notion, dating from Adleman~\cite{Adleman} is that we should be able to consider the various characters $\chi_{\mathfrak{p}}$ as uniformly distributed, \emph{independent} samples of the dual space of $H$, and so a small collection should suffice to distinguish any two elements of $H$. We will not show this here, but instead show the weaker notion above. This will still suffice to ensure that a small collection of samples will distinguish \emph{almost} all elements of $H$ from being trivial.
\end{rem}
\begin{proof}
Following an idea of Adleman~\cite{Adleman}, we will carefully study the behaviour of quadratic characters induced by primes of large norm.

Suppose we have $K$ a finite extension of $\Q$, and $L / K$ Galois, with $G = \operatorname{Gal}\left(L / K\right)$. Let $\Delta_L, \Delta_K$ be the absolute values of the discriminants of $L$ and $K$ respectively, and let $d_L$, $d_K$ be the degrees of $[L : \Q]$ and $[K : \Q]$ respectively. Given any prime $\mathfrak{p}$ in $K$ which is unramified in $L$, we define the \emph{Artin Symbol} $\left[\frac{L/K}{\mathfrak{p}}\right]$ to be the conjugacy class of the Frobenius automorphisms of $L / K$ corresponding to primes in $L$ dividing $\mathfrak{p}$. We define:
\[
\pi_{C}\left(x\right) = \left|\left\{\mathfrak{p} : \mathfrak{p}\textrm{ prime}, \mathbf{N}_K\left(\mathfrak{p}\right) < x, \left[\frac{L/K}{\mathfrak{p}}\right]\in C\right\}\right|,
\]

In the simplest case where $K = \Q$ and $L = \Q\left(\exp\left(\frac{2\pi i}{n}\right)\right)$ is a cyclotomic field, the Artin Symbol of any prime $p \in \N$, with $p \nmid n$ would correspond to the residue of $p$ modulo $n$.

We note the following theorem, which strengthens the celebrated Density Theorem of Chebotarev, which is itself a generalisation of the prime number theorem for arithmetic progressions.

\begin{fact}[The Unconditional Effective Chebotarev density theorem~\cite{LagariasOdlyzko, Serre}]
We have $L / K / \Q$ a sequence of extensions, with $L/K$ Galois, and retain the notation above. Let $C \subseteq G$ such that $gCg^{-1} = C  \;\forall g \in G$, i.e. $C$ is a union on conjugacy classes of $G$. Let $|\tilde{C}|$ be the number of conjugacy classes contained in $G$. Let $1-\nu$ be the Siegel zero of $\zeta_{L}$ if it exists, and 0 otherwise. Then there exists $c_1 > 0$ such that if $\log x \geq 10 d_L \log^2 \Delta_L$ then:
\begin{align}\label{eq:chebotarev}
\left|\pi_{G'}\left(x\right) - \frac{|C|}{|G|}\Li\left(x\right)\right| \leq \frac{|C|}{|G|}\Li\left(x^{1-\nu}\right) + \O\left(x|\tilde{C}|\exp\left(-c_1\sqrt{\frac{\log x}{d_L}}\right)\right).
\end{align}
\end{fact}
For a hands on introduction to this topic, we recommend~\cite{murtyproblems}.
To continue the proof, we set $K = \Q\left(\alpha\right)$, and choose some $h \in \mathcal{O}_K$ of minimal norm representing a non-trivial element of $H$. We let $L = K\left(\sqrt{h}\right)$. Now $G = C_2$, $d_K = d$, $d_L = 2d$. We also note that in this case the value $\left[\frac{L/K}{\mathfrak{p}}\right]$ corresponds exactly to the action of the quadratic character $\chi_\mathfrak{p}$ induced by $\mathfrak{p}$ on $h$.

Now, we use Minkowski's bound on the minimum norm of an integral ideal:
\[
\mathbf{N}_{K/\Q}\left(h\right) \leq M_{K/\Q} = \sqrt{\Delta_{K/\Q}}\left(\frac{4}{\pi}\right)^{\frac{d}{2}}\frac{d!}{d^d} = n^{\delta\kappa\left(1+\o\left(1\right)\right)}.
\]
The relative discriminant $\Delta_{L/K}$ is the norm of the different $\delta_{L/K}$ of the extension. By construction, this ideal is generated by $2h$, and so is an integral ideal~\cite[Chapter III, Proposition 2 and Corollary]{Lang}. Hence we obtain:
\begin{equation}\label{Delta-L-bound}
\Delta_{L/\Q} \leq \mathbf{N}_{K/\Q}\left(2h\right) \Delta_{K/\Q}^2 \leq n^{\left(5+\o\left(1\right)\right)\delta \kappa}.
\end{equation}
We apply~\refp{eq:chebotarev} to the extension $L/K$, noting that it is of degree 2. We obtain that for $\mathfrak{p}$ chosen uniformly randomly with $\mathbf{N}\mathfrak{p} \leq x$:
\begin{align}\label{chebotarev-application}
\left|\P\left(\chi_\mathfrak{p}\left(h\right) = 1\right) - \frac{1}{2}\right| < x^{-\nu(1 + \o(1))} + \O\left(2\log x\exp\left(-c_1\sqrt{\frac{\log x}{d_L}}\right)\right).
\end{align}
We wish to ensure that $\P\left(\chi_\mathfrak{p}\left(h\right) = 1\right) = \frac{1}{2} + \o\left(1\right)$, and so it suffices for us to insist that:
\begin{equation}\label{lower-bound-logx}
\log x = \boldsymbol{\omega}\left(d_L (\log \log x)^2\right),
\text{ and additionally } \log x = \boldsymbol{\omega}\left(\nu^{-1}\right)\text{ if }\zeta_L\text{ has a Siegel zero}
\end{equation}
Note that we \emph{do not} sample from a uniform distribution over characters of bounded norm; we will sample from a distribution which is close enough to being uniform that we can extract useful bounds.

\begin{definition}
For a field $K$ and $h$ a minimal norm representative of an element of $H$, we define $L_h = K\left(\sqrt{h}\right)$. For $\varepsilon > 0$ we define the \emph{exceptional set}:
\[
E_{K,\varepsilon} = \left\{h.\{z^2 : z \in K^\times\}\in H \text{ s.t.} \exists \nu \text{ s.t. }\zeta_{L_h}(1-\nu) = 0, \nu^{-1} > L_n\left(\frac{1}{3}, \varepsilon\right)\right\}
\]
Note that the field $L_h$ is independent of the choice of representative $h$ for the element of $H$.
\end{definition}
The exceptional set is the subset of $H$ which cannot be reliably distinguished from $0$ by characters induced by primes of size $\exp\left(L_n\left(\frac{1}{3}, \varepsilon\right)\right)$; if there is a Siegel zero of this form then it is possible that almost every prime of this size induces a character which vanishes on some element of $H$. We state the following Lemma which we will prove later.
\begin{lemma}\label{prob-zero-as-used} Suppose that $K = \Q\left(\alpha\right)$ is a number field where $\alpha$ is a root of a irreducible $f = \hat{f} + \left(x-m\right)R$ where $R$ is uniformly random. Then for $\varepsilon = \left(\frac{1}{3} + \o\left(1\right)\right)\delta$,
\[
\P_f\left(\left|E_{K, \varepsilon}\right| > \frac{4}{3}\log \log n\right) \leq L_n\left(\frac{2}{3}, \frac{\kappa-\delta^{-1}}{3}\left(1+\o\left(1\right)\right)\right)^{-1}.
\]
\end{lemma}

\begin{rem}
The proof of this lemma will be based on the sparseness of Siegel zeros of zeta functions associated to the extensions $L_h / K$. Then for \emph{most} $f$, at most $\frac{4}{3}\log \log n$ elements of $H$ cannot be distinguished from $0$, and so we can use brute force to find a pair of polynomials mapping to the same element in $H$ without altering the $L_n(\frac{1}{3})$ run time. Then their product must be trivial in $H$ and thus gives a congruence of squares. To obtain an $L_n(\frac{1}{3})$ run time it would suffice to prove the above statement with the $\log \log n$ replaced by any $L_n(\frac{1}{3}, \o(1))$ and the $L_n(\frac{2}{3})$ with any $L_n(\frac{1}{3}, \omega(1))$.
\end{rem}

Given this claim, we have an $x$ satisfying~\refp{lower-bound-logx} for all but $\frac{4}{3}\log \log n$ of the $h \in H$ and with $\log x < L_n\left(\frac{1}{3}, \varepsilon\right)$ for all but a $L_n\left(\frac{2}{3}\right)^{-1}$ fraction of our polynomials $f$. As we will only examine $L_n\left(\frac{1}{3}\right)$ polynomials $f$, we may simply choose to fail on this exceptional set of $f$ and will still guarantee that we fail with probability $\o(1)$.

Any prime $\mathfrak{p}$ with $\mathbf{N}\left(\mathfrak{p}\right) < x$ must divide a prime $p$ with $p < x$, and if $\mathfrak{p}$ is of degree $k$, then $p < \sqrt[k]{x}$. Furthermore, each $k$\textsuperscript{th} degree prime dividing $p$ corresponds to a simple degree $k$ divisor of $f$ modulo $p$.
We present an algorithm to sample $\Upsilon$. This will output ideals, most of which are prime.

\begin{enumerate}
\item[] \textsc{IdealSampler}($f$)
\item Uniformly randomly choose a degree bound $k \in [d]$.
\item Choose a uniformly random integer $r \in \left(x^{\left(k+1\right)^{-1}}, x^{k^{-1}}\right]$.
\item Use the Miller-Rabin primality test to discard composite $r$ with probability $1 - \O\left(\log^{-2}x\right)$. This takes time $\O\left(\log^3 x \log \log x\right)$. With probability at least $\boldsymbol{\Omega}\left(\left(\log x\right)^{-1}\right)$ it will occur that $r$ is prime, and so any $r$ produced at this stage is prime with probability $1-\O\left(d\left(\log x\right)^{-1}\right)$. For the purposes of exposition of the algorithm, we will assume that all the $r$ are prime.
\item Factor $f \mod\left(r\right)$ in time $\O\left(\left(d\log x\right)^3\right)$~\cite{GathenPanario}, and find the collection of irreducible and unrepeated factors $s_i$ of degree at most $k$. Observe that the factors $s_i$ correspond to primes in the number field of norm at most $x$ dividing $\left(r\right)$. For such an $r$, we have at most $d$ primes $s_i$.
\item If we find $j$ factors of degree at most $k$, we take $s$ to be \emph{one} of them uniformly at random with probability $jd^{-1}$. Otherwise return to step 1
\item Output the pair $r,s$.
\end{enumerate}
\begin{rem} To ultimately obtain the run time bounds which we need, we need the run time of \textsc{IdealSampler} to be at most $\O(\log^4 x)$. In particular, we will find that $\log x = L\left(\frac{1}{3}\right)$. This prevents using AKS-style deterministic primality testers~\cite{AKS}, and so we have to permit a small probability that $r$ is not prime.
\end{rem}
\begin{rem}
Note that if $r$ is not prime, then the factorisation of step 4 may fail; if this occurs we return to step 1. If we do obtain a character from a non-prime $r$, we observe that it is still quadratic and therefore vanishes on the squares as required. Since we obtain at most one character from each sampled $r$, and are guaranteed to find a character if $k = d$ and $r$ is chosen to be prime, the fraction of the characters which are not induced by primes is $\o\left(1\right)$. We will absorb this error term into our estimates of the probability that some $h$ is distinguished from $0$.
\end{rem}

To finish the proof of Lemma~\refp{prime-sampling}, we need to show that this algorithm is fast, the characters $\chi_{r,s}$ can be evaluated quickly and that they are sufficiently uniform that the bounds of Equation~\refp{chebotarev-application} give the bounds we need.

\begin{claim}
The expected time taken to sample $(r,s) \sim \Upsilon$ as above is at most $L_n\left(\frac{1}{3}, \left(4 + \o(1)\right)\varepsilon\right)$
\end{claim}
\begin{proof}
We note that each attempt from the start of the algorithm takes time $\O\left(\left(d\log x\right)^3\right)$, with the fourth step being slowest.

We are guaranteed to find a factor if our degree bound $k$ is $d$ (a probability $1/d$ event), the integer $r$ is prime (a probability $\boldsymbol{\Omega}(1/\log x)$ event), and we successfully take an ideal in step five (a probability $\boldsymbol{\Omega}(1/d)$ event if $k = d$). Hence the number of attempts needed to output a prime is bounded in expectation by $\O\left(d^2 \log x\right)$.

Hence the time taken to find an ideal is bounded in expectation by $\O(d^5 \log^4 x) = L_n\left(\frac{1}{3}, \left(4 + \o(1)\right)\varepsilon\right)$.
\end{proof}

\begin{claim}
For any fixed $h$, $\P_{\Upsilon}\left(\chi_{r,s}\left(h\right) = -1\right) \geq \frac{1}{2d}\left(1+\o\left(1\right)\right)$
\end{claim}
\begin{proof}
The distribution of primes $\mathfrak{p}$ generated is uniform over $\mathfrak{p} \mid \left(r\right)$ for $r \in \left(x^{\left(k+1\right)^{-1}}, x^{k^{-1}}\right]$ of degree at most $k$. This property also trivially holds for a uniform distribution over primes of norm $\leq x$. Thus the difference between $\Upsilon$ and a uniform distribution over primes of norm $\leq x$ is the distribution of the degree of these primes. 

The probability that $\Upsilon$ samples $\mathfrak{p}$ with $\N(p) \leq x$ and $\mathfrak{p} \mid (r)$ for $r$ in each of these intervals is $\frac{1}{d}$. Hence $\Upsilon$ pointwise dominates $d^{-1}$ times the uniform distribution over all primes of norm below $x$.

Then $\P_{\Upsilon}\left(\chi_{r,s}\left(h\right) = -1\right) \geq \frac{1}{d}\P_{N(\mathfrak{p}) \leq x}\left(\chi_{\mathfrak{p}}\left(h\right) = -1\right) = \frac{1}{2d}\left(1+\o\left(1\right)\right)$.
\end{proof}

\begin{claim}
Evaluating the character $\chi_{r, s}$ associated with the ideal $\mathfrak{p} \simeq (r, s)$ sampled as above on a term $a-b\alpha$ takes time at most $L_n\left(\frac{1}{3}, \left(2 + \o\left(1\right)\right)\varepsilon\right)$.
\end{claim}
\begin{remark}
The following proof is somewhat technical in that the logarithms of the numbers of interest are large. Hence we have to quite precisely track which arithmetic operations are used. The reduction to Legendre symbols is of great use, as it allows us to avoid doing arithmetic in $\F_r^k$.
\end{remark}
\begin{proof}
We note that if $r = 2$, then the character is identically $1$ as all elements of the field are squares. Hence we assume $r > 2$.
For any polynomial $P \in \F_r[X]$, let $|P| = r^{\operatorname{deg}(P)}$. We recall from Equation~\refp{eq:chi-def2} that:
\[
\chi_{r, s}(a-b\alpha) = \left(\frac{a-bX}{s(X)}\right)
\]
where the RHS is the Legendre symbol over $\F_r[T]$. We first note that for any constant $c$, we can reduce the calculation to finding a Legendre symbol mod $r$:
\begin{equation}\label{eq:fnfieldrecip}
\left(\frac{c}{P}\right) = c^{\frac{|P|-1}{2}} = \left(\frac{c}{r}\right)^{\frac{r^k - 1}{r-1}} = \left(\frac{c}{r}\right)^{k}
\end{equation}
We draw attention to the law of quadratic reciprocity in function fields, introduced initially in~\cite{ArtinQR} and discussed at length in \cite[Chapter 3]{rosen}. For any two relatively prime monic irreducible polynomials over $\F_r$:
\[
\left(\frac{P}{Q}\right)\left(\frac{Q}{P}\right) = (-1)^{\frac{|P| - 1}{2}\frac{|Q| - 1}{2}},
\]
Hence:
\begin{align*}
\chi_{r, s}(a-b\alpha) &= \left(\frac{-b}{r}\right)^{k}\left(\frac{X-ab^{-1}}{s(X)}\right) 
=\left(\frac{-b}{r}\right)^{k}(-1)^{\frac{r-1}{2}\frac{r^{\operatorname{deg}(s)}-1}{2}} \left(\frac{s(X)}{X-ab^{-1}}\right) \\
&=\left(\frac{-b}{r}\right)^{k}(-1)^{\frac{r-1}{2}\frac{r^{\operatorname{deg}(s)}-1}{2}} \left(\frac{s(ab^{-1})}{X-ab^{-1}}\right)
=\left(\frac{-b}{r}\right)^{k}(-1)^{\frac{r-1}{2}\frac{r^{\operatorname{deg}(s)}-1}{2}} \left(\frac{s(ab^{-1})}{r}\right),
\end{align*}
with the last equality following from Equation~\ref{eq:fnfieldrecip}. Parities of $\frac{r-1}{2}$ and $\frac{r^{\operatorname{deg}(s)}-1}{2}$ can be easily computed. Hence to compute $\chi_{r, s}(a-b\alpha)$ it suffices to compute $s(ab^{-1})$ and two Legendre symbols modulo $r$. By use of reciprocity over $\Q$, we can compute a Legendre symbol modulo $r$ in $\O(\log r)$ additions or subtractions of numbers of size at most $r$.

To compute $b^{-1} \imod{r}$ requires the Extended Euclidean algorithm to be run, which requires $\O(\log r)$ additions of numbers of size at most $r$. To compute $ab^{-1} \imod{r}$ requires one multiplication. To compute $s(ab^{-1}) \imod{r}$ requires at most $\O(d)$ additions and multiplications modulo $r$.

Addition or subtraction of numbers of size $r$ (or modulo $r$) takes $\O(\log r)$ steps. Multiplication modulo $r$ takes $\O(\log^2 r)$ steps by iterative addition and doubling. Hence the computation in total requires time
\[
\O(d\log^2 r) = \O(d^{-1} \log^2 x) = L_n\left(\frac{1}{3}, \left(2 + \o\left(1\right)\right)\varepsilon\right).\qedhere
\]
\end{proof}

Hence we can take $c = \left(4+\o\left(1\right)\right)\varepsilon$ to complete the proof of Lemma~\refp{prime-sampling}. We note that this is tight for both \emph{finding} and \emph{evaluating} the set of characters.
\end{proof}
\begin{rem} Note that we do not show that the characters we sample are \emph{independent} in the sense of \cite[Lemma 8.2]{BuhlerLenstraPomerance}, in that we do not prove that the characters induce independent, uniformly distributed maps in the dual space of $H$. We have instead shown merely that there is a probability, uniformly bounded away from zero, that any element of $H$ is not in the kernel of one of our sampled characters.
\end{rem}
\begin{rem}Note that normally, the NFS takes characters from the smallest primes above $B$ (i.e. $\log x = \O\left(d \log d\right)$). Even conditional on GRH, our methods require taking somewhat larger primes ($\log x = \O\left(d \log^2 d\right)$), and unconditionally we require \emph{much} larger primes to control their statistics. Furthermore, the standard NFS takes only primes of first degree, which are \emph{asymptotically} guaranteed to be almost all of the primes of bounded norm as the bound tends to infinity for a \emph{fixed} number field.  Heuristically, it might seem reasonable that these primes, over a small range, would induce sufficiently random characters to yield the required reduction to squares, but as discussed in~\cite{BuhlerLenstraPomerance}, proving this would require demonstrating exceptionally good equi-distribution properties for the Chebotarev Density theorem applied to the splitting field of $f$ at bounded norm, and gaining sufficient control would require a better effective bound on the error term.
\end{rem}


\begin{proof}[Proof of Theorem~\refp{NFS-characters-uncond}]
With the claims of the previous section, we are in a position to produce our linear $\Psi_\mathcal{F}$ with small kernel, and thus to produce a congruence of squares. We will need to track precisely the computational complexity of these operations, as some of the numbers involved have $L\left(\frac{1}{3}\right)$ bits.

First, we sample $4d(\delta\kappa\log n + \frac{\delta^2\kappa}{2\log2}\frac{\log^{4/3}n}{\log\log n^{1/3}})$ pairs $\left(r_i, s_i\right)$ independently from $\Upsilon$ as in Lemma~\refp{prime-sampling}. Note that our sample is of size $\o\left(\log^2 n\right) = L_n\left(\frac{1}{3}, \o\left(1\right)\right)$. Recall that taking each sample takes at most $L_n\left(\frac{1}{3}, c\right)$ time in expectation, so we can produce the required sample in expected time $L_n\left(\frac{1}{3}, c+\o\left(1\right)\right)$. We have $M = 1 + B + dB'+ 4d\left(\kappa\delta\log n+ \frac{\delta^2\kappa}{2\log2}\frac{\log^{4/3}n}{\log\log n^{1/3}}\right) = L_n\left(\frac{1}{3}, \max\left(\beta, \beta'\right)\right)^{1+\o\left(1\right)}$ linear polynomials. For each of these, we need to evaluate each of our characters,
which takes time $L_n\left(\frac{1}{3}, \frac{c}{2} + \max\left(\beta, \beta\right)\right)^{1+\o\left(1\right)}$.

Fix some $h \in H \backslash \{0\}$ which is not in the exceptional set, which we recall is of size at most $\log \log n$. Each map $\chi_{r_i, s_i}$ is independent and induces a map in $\operatorname{Hom}\left(H, \F_2\right)$ such that:
\[
\P\left(h \notin \ker\left(\chi_{r_i, s_i}\right)\right) \geq \frac{1+\o\left(1\right)}{2d}.
\]
As a corollary:
\begin{align*}
\P\left(h \in \ker\left(\chi_{\mathcal{F}}\right)\right) &\leq \left(1-\frac{1+\o\left(1\right)}{2d}\right)^{4d\left(\kappa\delta\log n+ \frac{\delta^2\kappa}{2\log2}\frac{\log^{4/3}n}{\log\log n^{1/3}}\right)} 
\leq |H|^{-2+\o\left(1\right)}.
\end{align*}
Hence by a union bound over the non-trivial elements of $H$ the probability that any of these non-exceptional and non-zero elements is in the kernel is $\o\left(1\right)$. Hence with high probability the kernel of $\Psi_\mathcal{F}$ has size at most $\frac{4}{3}\log \log n$.

With these additional random characters in Step~\refp{square-formation}, our existing matrix algebra allows us to reduce (concretely) $M$ linear polynomials from Step~\refp{smooth-generate} to a single polynomial $P$ such that $P\left(m\right)$ is square in $\Z$ and $P\left(\alpha\right)$ is a square in $\mathcal{O}_{\Q\left(\alpha\right)}$ multiplied by one of at most $\frac{4}{3}\log \log n$ elements of $h$. Hence after repeating the whole algorithm $\ell = \frac{4}{3}\log \log n$ times to generate some $P_1,\ldots, P_\ell$, we are able to guarantee that for some $i < j$, $P_i$ and $P_j$ lie over the same element $h$, and hence $P_iP_j$ is in fact a square in $\mathcal{O}_{\Q\left(\alpha\right)}$. In the sequel we will test all of these $\binom{\ell}{2} \sim \frac{8}{9}(\log \log n)^2$ polynomials separately.

We now provide some details to establish the required run time bounds. The matrix of exponents modulo 2 and characters is sparse. As a result, we can use fast kernel finding algorithms such as the block-Wiedemann algorithm~\cite{ThomeWiedemann} to find a suitable subset $S_i$ to construct a $P_i$ in time
\[
\O\left(M^2\right) = L_n\left(\frac{1}{3}, 2\max\left(\beta, \beta'\right)\left(1 + \o\left(1\right)\right)\right).
\]
Now, if $\gamma \in \mathcal{O}_{\Q\left(\alpha\right)}$ and $\gamma^2 \in \Z[\alpha]$, then $\gamma . f'\left(\alpha\right) \in \Z[\alpha]$ \cite[Chapter III, Proposition 2]{Lang}. We take $S = S_i \Delta S_j$. We then fix the polynomial $P$ to be
\[
P = \left[\frac{\partial f}{\partial x}\left(x, 1\right)\right]^2 \prod_{\left(a,b\right) \in \S}\left(a-bx\right), \text{ and so }
u^2 = \left[\frac{\partial f}{\partial x}\left(x, y\right)\right]\left(m, 1\right)^2\prod_{\left(a,b\right) \in \S}\left(a-mb\right)
\]
is a square in $\Z$. Hence $u$ can be found by taking the product modulo $n$ over all $r < B$ of $r$ raised to half the total order of $r$ in the terms $\left(a-mb\right)$ for $\left(a,b\right) \in \S$ and multiplying by $f'\left(m,1\right)$. That we compute the square root in this fashion is important to ensure that our computation can be done in polynomial time; we have ensured that we only need to do $M\log n$ additions and divisions to find the exponents, and at most $M\log n$ modular multiplications to compute the $u \imod n$ from the exponents.

Similarly, for at least one of the $\binom{\ell}{2}$ polynomials considered, there exists $v\in \Z[\alpha]$ such that:
\[
v^2 = \left[\frac{\partial f}{\partial x}\left(x, y\right)\right]\left(\alpha, 1\right)^2\prod_{\left(a,b\right) \in \S}\left(a - \alpha b\right).
\]
By Montgomery's method~\cite{Montgomery, Thome}, we can compute square roots in the number field, and thus find $v\left(m,1\right) \imod{n}$ in time $\O(M^2)$.
We abuse notation slightly to write $v\left(m\right)$ as the element of $\Z/n\Z$ obtained by substituting $m$ for $\alpha$. Then:
\begin{align*}
v\left(m\right)^2& \imod{n} = v\left(m\right)^2 \imod{f\left(m,1\right)} \\
&= \left(\left[\frac{\partial f}{\partial x}\left(x, y\right)\right]\left(\alpha, 1\right)^2\prod_{\left(a,b\right) \in \S}\left(a - \alpha b\right) \imod{f\left(\alpha,1\right)}\right)\left(m\right) \\
&= \left[\frac{\partial f}{\partial x}\left(x, y\right)\right]\left(m, 1\right)^2\prod_{\left(a,b\right) \in \S}\left(a - mb\right) \imod{f\left(m, 1\right)}
 = u^2 \imod{n}
\end{align*}
and so we have constructed a congruence of squares in time:
\[
L_n\left(\frac{1}{3}, \max\left(2\max\left(\beta, \beta'\right), \max\left(\beta, \beta'\right) + \frac{c}{2}, c\right)\right)^{1+\o(1)}.
\]
Hence, the run time bound is as claimed as $c \leq \left(\frac{4}{3}+\o\left(1\right)\right)\delta$, we can insist that we have at most $\frac{4}{3}\log \log n$ exceptional values of $h$, and our $f$ lies off a set of probability at most $L_n\left(\frac{2}{3}, \frac{\kappa-\delta^{-1}}{3}\left(1+\o\left(1\right)\right)\right)^{-1}$.
\end{proof}
%
\begin{rem}
To prove the analogous statement for the multiple polynomial NFS, we sample and add these characters for each $f^{(i)}$ used. Then our algebra in each field finds a square in the number field whose matching image in $\Z$ is the product of a $B$-smooth number and a square, and such that the product of these $B$-smooth parts is itself square. Hence taking the product of these relationships yields a congruence of squares.
\end{rem}

We observe an immediate strengthening of Lemma~\refp{prob-zero-as-used} conditional on GRH:
\begin{claim}
Conditional on GRH, for $\varepsilon = \log^{-1/4} n = \o(\delta)$,
\[
\P_f\left(\left|E_{K, \varepsilon}\right| > 0 \right) = 0.
\]
\end{claim}
\begin{proof}
Under GRH, there are no zeros of any of our zeta functions with real part greater than $\frac{1}{2}$. As a corollary, $\nu^{-1} \leq 2$ uniformly. Hence $|E_{K, \varepsilon}| = 0$ if $L_n\left(\frac{1}{3}, \varepsilon\right) > 2$, which is entailed by our choice of $\varepsilon$.
\end{proof}
It remains to prove Lemma~\refp{prob-zero-as-used}. We need the following result of Stark:
\begin{fact}\label{starkfact} (Stark \cite[Lemma 8]{Stark}) Let $K$ be a field of finite degree, let $c(K) = 4$ if $K/\Q$ is normal and $c(K) = 4([K : \Q])!$ otherwise. Suppose there is a real $1 - \nu$ in the range:
\[
1 - (c(K) \log|\Delta_K|)^{-1} \leq 1 - \nu < 1
\]
such that $\zeta_K(1-\nu) = 0$. Then there is a quadratic field $F \subset K$ such that $\zeta_F(1-\nu) = 0$.
\end{fact}
We note that the following slightly stronger statement follows exactly from the proof provided in \cite{Stark}:
\begin{cor}\label{starkcor}
Let $K$ be a field of finite degree, and $K'$ the normal closure of $K$. Then Fact~\ref{starkfact} holds with $c(K) = 4([K' : \Q])$.
\end{cor}
We record the following fact of Landau on the distribution of zeros of Dirichlet $\L$-functions:
\begin{fact}[Landau~\cite{Landau}, see {\cite[pp. 367]{MontgomeryVaughan}}]\label{landau-L}
There is a constant $c$ such that given two characters $\chi_r$, $\chi'_{r'}$ of moduli $r, r'$ respectively, with $\chi_r\chi'_{r'}$ non-principal, then $\L\left(s, \chi_r\right)\L\left(s, \chi'_{r'}\right)$ has at most one real zero in $\left(1-\frac{c}{\log rr'},1\right)$
\end{fact}
\begin{proof}[Proof of Lemma~\refp{prob-zero-as-used}]
We assume $f$ is irreducible; by Lemma~\refp{f-reducible-probability} the probability that $f$ is reducible can be absorbed as our error term.
Recall that $K = \Q(\alpha)$ and $h \in O_K$ is a representative of an element of $H$.

\begin{claim} If $L_h' $ is the normal closure of $L_h = K(\sqrt{h})$, $[L_h' : \Q] \leq 2^d d!$.
\end{claim}
\begin{proof}
Let $K'$ be the splitting field of $K$. By construction, $[K' : \Q] \leq d!$ and $K' / \Q$ is normal (indeed Galois). Given $h \in K$, let $\mathcal{O}_h$ be the orbit of $h$ under the action of $\text{Gal}(K'/\Q)$. Then $|\mathcal{O}_h| \leq d$. We adjoin square roots of each element of $\mathcal{O}_h$ to $K'$ to obtain a field $L'$.

Then $[L' : \Q] \leq 2^dd!$. Since degree 2 extensions are normal, the compositum of normal extensions is normal, and $L' / K'$ is a compositum of at most $d$ degree 2 extensions, the extensions $L' / K'$ and $K' / \Q$ are normal. We note that any $\sigma \in \operatorname{Aut}_{\Q}(K')$ can be extended to an element of $\operatorname{Aut}_{\Q}(L')$, as $\sigma$ acts on $\mathcal{O}_h$ as a permutation. Hence in particular $L' / \Q$ is normal, and so $L_h' \subseteq L'$. Hence $[L_h' : \Q] \leq 2^dd!$.
\end{proof}

Hence from Corollary~\refp{starkcor}, if $\nu^{-1} > 2^{d+2}d! \log \Delta_{L_h/\Q}$, then $1 - \nu$ must be a zero of some quadratic subfield $F_h = \Q\left(\sqrt{s_h}\right) \subseteq L_h$.
Note that as $[K : \Q]$ is odd there are no quadratic subfields of $K$. Hence $F_h \cap K = \Q$ and $F_h$ is the only quadratic subfield of $L_h$.

Furthermore, $L_h$ is the minimal field containing $F_h$ and $K$. Since the classes in $H$ are not related by squares of elements of $K$, the field $L_h$ does not contain a root of any $h'$ in a different class in $H$. Hence as $h$ varies, the produced $L_h$ are distinct fields and so the $s_h$ must all be distinct. We observe that by transitivity of the discriminant (eg. \cite[Thm 1.46]{koch1997algebraic} or \cite[Cor 2.10]{neukirch2013algebraic}) in the towers of fields $L_h / F_h / \Q$ and $L_h / K / \Q$:
\begin{equation}\label{norm-tower}
\Delta_{F_h/\Q}^d \mathbf{N}_{F_h/\Q}\left(\Delta_{L_h/F_h}\right) = \Delta_{K/\Q}^2 \mathbf{N}_{K/\Q}\left(\Delta_{L_h/K}\right)\quad (= \Delta_{L_h/\Q}).
\end{equation}
Furthermore, as in Equation~\refp{Delta-L-bound}, $\Delta_{L_h/K}$ is the norm of the different ideal $(2h)$ and so from Minkowski's bound:
\[
\mathbf{N}_{K/\Q}\left(\Delta_{L_h/K}\right) \leq \sqrt{\Delta_{K/\Q}} \left(\frac{4}{\pi}\right)^d\left(\frac{d!}{d^d}\right) \leq \sqrt{\Delta_{K/\Q}}.
\]
Since $\Delta_{K/\Q} \leq L_n\left(\frac{4}{3}, \frac{1}{2}\delta^2\kappa\right)$ and $\Delta_{L_h/\Q} \leq L_n\left(\frac{4}{3}, \frac{5}{4}\delta^2\kappa\right)$, $\Delta_{F_h/\Q} = \O\left(L_n\left(\frac{4}{3}, \frac{5}{4}\delta\kappa\right)\right)$.
Now, since a prime $p$ contributes a factor $(1-p^{-z})^{-1}$ to $\zeta_{F_h/\Q}\left(z\right)$ if $p \nmid \Delta_{F_h}$ and $(1-p^{-z})^{-2}$ otherwise:
\[
\zeta_{F_h/\Q}\left(z\right) = \zeta\left(z\right)\mathcal{L}\left(z, j \rightarrow \left(\frac{\Delta_{F_h/\Q}}{j}\right)\right)
\]
and by reciprocity $j \rightarrow \left(\frac{\Delta_{F_h/\Q}}{j}\right)$ is a character of modulus $\Delta_{F_h/\Q}$. From Fact~\refp{landau-L}, if there are two characters with moduli $q,q'$ respectively, at most one has an $\mathcal{L}$-function with a zero $1-\nu$ and
$
\nu^{-1} > c\log qq'
$
for some effective constant $c$. As a corollary, there is at most one character with modulus in $[q, q^e]$ with a zero at $1-\nu$ and
$
\nu^{-1} > (e+1)c\log q.
$

Note that since $\Delta_{F_h/\Q} < L_n\left(\frac{4}{3}, \left(\frac{5}{4} + \o(1)\right)\delta\kappa\right)$ and is an integer, the whole range of discriminants can be covered with only $\frac{4}{3}\log \log n$ ranges of the form $[x, x^e]$. Hence there are \emph{at most} $\frac{4}{3}\log \log n$ characters (and hence, potential extensions $L_h$) with exceptional zeros such that
\[
\nu^{-1} > (e+1)c \log\left(L_n\left(\frac{4}{3}, \left(\frac{5}{4} + \o\left(1\right)\right)\delta\kappa\right)\right) = \O\left(\delta\kappa\log^{\frac{4}{3}} n (\log \log n)^{-\frac{1}{3}}\right)
\]
as required. Note that this bound on $\nu^{-1}$ is much weaker than the required $\nu^{-1} > 2^{d+2}d! \log \Delta_{L_h/\Q}$, and so there are at most $\frac{4}{3}\log \log n$ extensions $L_h/\Q$ with exceptional zeros and $\nu^{-1} > 2^{d+2}d! \log \Delta_{L_h/\Q}$. We observe that:
\[
2^{d+2}d! \log \Delta_{L_h/\Q} \leq d^{d\left(1+\o\left(1\right)\right)}\log^{\boldsymbol{O}\left(1\right)}n = L_n\left(\frac{1}{3}, \frac{\delta}{3}\left(1+\o\left(1\right)\right)\right).\qedhere
\]
\end{proof}

\begin{rem} The use of the relative discriminant both here and in the proof of Lemma~\refp{prime-sampling} is, heuristically, to control the extent to which primes can ramify. In turn this allows tight control of the deviations of the behaviour of primes in the number fields from the behaviour over $\Z$. In both cases the detailed numerics of the bounds are not especially important, beyond the fact that they provide upper bounds whose logarithms are much smaller than $L_n\left(\frac{1}{3}\right)$.
\end{rem}

\begin{rem}
We note that we $\Delta_{L_h/\Q}$ is merely known to be the absolute discriminant of a ``random'' field (under mild assumptions about the nature of the field). If we were to heuristically take it to be a random integer of the correct size, modulo being (for example) only $0, 1 \imod{4}$, we would obtain immediately that the probability that $\Delta_{L_h/\Q}$ is divisible by the $d$\textsuperscript{th} power of some integer exceeding $L\left(\frac{1}{3}\right)$ is of order $L\left(\frac{2}{3}\right)^{-1}$. If this held we would be able to remove the condition that $d$ is odd.
\end{rem}
In fact, we can show that under this kind of heuristic, the reduction to squares can be done in $L_n\left(\frac{1}{3}, \o\left(1\right)\right)$ time. In particular, we show:
\begin{claim} Set $K = \Q\left(\alpha\right)$ and $L_h = K\left(\sqrt{h}\right)$, for $\alpha$ the root of a random $f$ and $h$ any non-zero element of the ideal class group of $K$. Then Claim~\refp{prob-zero-as-used} holds (with $\varepsilon \rightarrow 0$) if there is an $\epsilon > 0$ and an $\epsilon' \rightarrow 0$ such that:
\[
P_f\left(\exists h\in H, k \geq L\left(1/3\right)^{\epsilon'} \text{ s.t. } k^d \mid \Delta_{L_h/\Q}\right) < L\left(1/3+\epsilon\right)^{-1}
\]
\end{claim}
\begin{proof} Again, we need a result of Stark:
\begin{fact}[{\cite[Lemma 11]{Stark}}	] We assume $f$ is irreducible; by Lemma~\refp{f-reducible-probability} the probability that $f$ is reducible is can be adsorbed into $\varepsilon$. Let $F$ be a quadratic field, and $1-\nu$ a zero of $\zeta_{F/\Q}$. Then
$
\nu^{-1} < \O\left(\sqrt{\Delta_{F/\Q}}\right).
$
\end{fact}
We require
$
\log x = \boldsymbol{\omega}\left(\max\left(\sqrt{\Delta_{F/\Q}}, 4\left(d-1\right)! \log\Delta_{L_h/\Q}, 2d (\log \log x)^2\right)\right),
$
and define $F_h$ as before, which is bounded by Equation~\refp{norm-tower}.
Given the conditions of the claim,
$
\P_f\left(\Delta_{F_h/\Q} > L\left(1/3\right)^{\o\left(1\right)}\right) < L\left(1/3 + \epsilon\right)^{-1},
$
and so for all but an $L\left(1/3 + \epsilon\right)^{-1}$ fraction of $f$ we can take $\log x = L\left(1/3, \o\left(1\right)\right)$, which achieves the claimed bounds.
\end{proof}

\section{Non-trivial Factors from Found Congruences}
We now turn to some brief comments on the fruitfulness of the found congruences.
We restrict to the situation where $p \equiv q \equiv 3 \imod{4}$. Then observe that characters $\chi_p, \chi_q$ modulo $p$ and $q$ given by the respective Legendre symbols:
\[
\chi_p(x) \defeq \left(\frac{x}{p}\right), \chi_q(x) = \left(\frac{x}{q}\right)
\]
are by definition multiplicative, of order two and degree one and with:
\[
\chi_p(-1) = \chi_q(-1) = -1.
\]
Consider the character $\chi_n \defeq \chi_p\chi_q$, which by construction is a character modulo $n$. We note that $\chi_n(\pm 1) = 1$, whilst for $x^2 \equiv 1 \imod{n}$, $x \not\equiv \pm 1 \imod n$, $\chi_n(x) = -1$.

Let the multiplicative map from $\Z[\alpha] \rightarrow \Z/n\Z$ given by $1 \rightarrow 1$, $\alpha \rightarrow m$ be denoted $\phi_{m, \alpha}$.
We define a multiplicative semigroup of polynomials $\mathcal{P}$ by:
\begin{align*}
\mathcal{P} &\defeq \mathcal{P}_{m, \alpha} = \left\{h \in \Z[X] : (h(m), n) = 1, h(m) \in \Z \text{ is square}, h(\alpha) = g^2, g \in Z[\alpha]\right\}.
\end{align*}
We say that the smooth part of $\mathcal{P}$, $\mathcal{P}_S$ is the set of $h$ such that $h \in \mathcal{P}$, $h(m)$ is smooth, $h(\alpha)$ is smooth, and $h$ splits as the product of linear factors of height $L_n\left(\frac{1}{3}, \sigma\right)$. We define a character $\chi_{\mathcal{P}}$ on $\mathcal{P}$ by:
\[
\chi_{\mathcal{P}}(h) = \chi_n\left(\sqrt{h(m)}\right)\chi_n\left(\phi_{m, \alpha}\left(\sqrt{h(\alpha)}\right)\right).
\]
Note that since $\chi_n(-1) = 1$, the definition of $\chi_{\mathcal{P}}$ is not dependent on which square roots are taken. Since $(h(m), n) = 1$ for all $h \in \mathcal{P}$, $\chi_{\mathcal{P}}$ naturally extends to the group of fractions of $\mathcal{P}$. Furthermore, $\chi_{\mathcal{P}}(h^2) = 1$ for any $h \in \Z[X]$ with $(h(m), n) = 1$. Hence $\chi_{\mathcal{P}}$ may be extended to a degree one and order two character on $\Z[X]$. Let $\mathcal{G}_{\mathcal{P}}$ be the set of these extensions. Then $\mathcal{G}_{\mathcal{P}}$ is closed under multiplication by any order two character which is trivial on $\mathcal{P}$. Recall that all of the characters we define in section~\refp{algebraic-squares} are of this form. We choose a specific extension in $\mathcal{G}_{\mathcal{P}}$ and denote this by $\chi_{\mathcal{P}}$.

\begin{conjecture}\label{char-decor}
Let $n, m, \alpha$ and notation be as above. Then $\chi_{\mathcal{P}}$, restricted to $\mathcal{P}_S$, cannot be written as a product of the characters $\chi_{\mathfrak{p}}$ and the characters $(-1)^{ord_p(P(m))}$, $(-1)^{ord_{\mathfrak{p}}(P(\alpha))}$.
\end{conjecture}
\begin{rem}
Note that for any $\beta, \beta', \delta, \kappa, \sigma$ satisfying the conditions of Equations~\refp{const-bounds}, we have that the dimension of $\mathcal{P}_S$ exceeds the number of characters given by a multiplicative $L_n\left(\frac{1}{3}\right)$ factor. As a corollary, almost every character on $\mathcal{P}_S$ satisfies the conditions of the conjecture.
\end{rem}

\begin{proof}[Proof of Theorem~\refp{NFS-fruitfull-0-1/2}]
Finding a fruitful congruence is precisely finding a polynomial $P \in \mathcal{P}$ such that
\[
\sqrt{P(m)} \neq \phi_{m, \alpha}\left(\sqrt{P(\alpha)}\right) \imod{n} \Leftrightarrow \chi_{\mathcal{P}}(P) = -1.
\]
If we consider $\chi_{\mathcal{P}}$ to be an additional character on our linear terms, we now seek to solve a non-homogeneous system of equations; we require that each prime appears with even total degree, that each of the additional prime-based characters $\chi_{\mathfrak{p}}$ is $1$, and that $\chi_{\mathcal{P}} \neq -1$.

Now, since $\chi_{\mathcal{P}}$ is not a product of these characters, on the kernel of these characters in $\mathcal{P}_S$ we must have that $\chi_{\mathcal{P}} = -1$ for a subspace of codimension 1. Running the Randomised NFS for at most $L_n\left(\frac{1}{3}, 2\sigma + \o(1)\right)$ time guarantees to find \emph{every} possible factor $a - bX$, and so finds every generator of $\mathcal{P}_S$. Then since we may uniformly sample the kernel of our linear operator by Weidemann's algorithm, we can guarantee that the relationship we find is fruitful with probability $\frac{1}{2}$.
\end{proof}

\begin{remark}
Note that if the conjecture is false for some $n, f$, which implies choices of $m, \alpha$, then the same argument entails that the NFS run with these parameters can never find a non-trivial congruence of squares. We note that since the NFS has been successfully run on a number of $n$ with generic $m$ and $f = \hat{f}_{n,m}$, it would be surprising if the conjecture was false for most $f$. We emphasise that there does not seem to be a natural reason for $\chi_{\mathcal{P}}$ to be related to characters either of form $\chi_{\mathfrak{p}}$ as defined in section~\refp{algebraic-squares} or of form $(-1)^{\ord_p(f(m))}$ or $(-1)^{\ord_{\mathfrak{p}}(f(m))}$ for primes $p$ or $\mathfrak{p}$ of small norm.
\end{remark}

The situation where $p, q$ are not both $3 \imod{4}$ is more complex, as there is no single character which can be used to consistently define which branch of the square root has been taken modulo $p$ and $q$. This is turn means that there is no multiplicative character which reveals whether a congruence is fruitful or not, and so it is unclear how (even notionally) one might show that the linear algebraic step may produce non-trivial congruences.

\begin{rem}
In the case of the multiple polynomial NFS, the space $\mathcal{P}$ is the product of the semigroups defined for each $f^{(i)}$, and the character $\chi_{\mathcal{P}}$ is defined by taking a decomposition of any element of the product into squares in the number fields and taking roots in all places individually. The space $\mathcal{P}_S$ is then the product of the smooth parts of the semigroups defined for each $f^{(i)}$, and the statement of Conjecture~\ref{char-decor} and the analogous proof for Theorem~\refp{NFS-fruitfull-0-1/2} are unchanged.
\end{rem}

\section{Smooth Numbers in Progressions and the Proof of Lemma~\ref{smooth-bad-dist}.}\label{NFS-uniform-chap}\label{smooth-bad-dist-proof}

The core aim of this section will be to establish suitable bounds on the smoothness of numbers in arithmetic progressions, so that we can prove Lemma~\refp{smooth-bad-dist}. In particular we seek equi-distribution results for the smooth numbers in arithmetic progressions. We will now discuss some of the context for this work, and related results which we build upon.
To control $\pi$ or $\pi_{q, a}$, it is natural to work with the Von Mangoldt function
$
\Lambda(n) = \mathbbm{1}_{n\text{ is a prime power}} \log n,
$
as the sum of $\Lambda(n)$ is more straightforwardly controlled. To pass from results of $\Lambda$ back to results on $\pi$ is essentially standard by partial summation. The deviation of $\Lambda$, given by
\[
\Delta\left(x,q,a\right)=\left|\sum_{y < x, y \equiv a \imod{q}} \Lambda\left(y\right)-\frac{y}{\phi\left(q\right)}\right|,
\]
can be effectively bounded with the GRH. The best unconditional bounds which are uniform in the moduli $q$ are given by the Siegel-Walfisz theorem, which is famously ineffective and is too weak for our purposes. However,
we can look at the \emph{average case} or seek to only obtain bounds for \emph{most} $q$.
For example, the \emph{Bombieri-Vinogradov}
theorem states that for all
$A>0,$ $\sqrt{x}/\log^{A}x\leq Q\leq\sqrt{x},$
\[
\sum_{q\leq Q}\max_{y\leq x}\max_{a}\Delta\left(y,q,a\right)\ll{A}\sqrt{x}Q\log^{5}x,
\]
and the related \textit{Barban-Davenport-Halberstam} theorem states that
\[
\sum_{q\leq Q}\sum_{\left(a,q\right)=1}\Delta\left(x,q,a\right)^{2} \ll_A xQ\log x.
\]
In both cases, the moral is that for \emph{most} $a$ and $q$, the deviation
$\Delta\left(x,q,a\right)$ cannot be much larger than $\sqrt{x}\log^{\O\left(1\right)}x$, and so in particular the error in the prime number theorem for arithmetic progressions similarly cannot be much larger than $\sqrt{x}\log^{\O\left(1\right)}x$ for most $a$ and $q$.

Analogous equi-distribution questions for $y$-smooth numbers over arithmetic progressions (counted by
$\Psi\left(x,y,q,a\right)$) have been studied (See \cite{Hildebrand-Tenenbaum-survey}
for a survey of results).
Granville~\cite{G1, G2} and
Soundararajan~\cite{Soundararajan} studied this question further. Soundararajan proved that
$
\Psi\left(x,y,q,a\right)\sim\frac{\Psi_{q}\left(x,y\right)}{\phi\left(q\right)}
$
and an analogous statement of equidistribution on cosets of a subgroup of $\left(\Z/q\Z\right)^*$.
Recently, Harper~\cite{Harper-sound} expanded the range of $y$ for which the result
is applicable. Building on Soundararajan's work further Harper~\cite{Harper}
also provided Bombieri-Vinagradov and Barban-Davenport-Halberstam type bounds for the smooth counting function:
\begin{fact}[{Harper~\cite[Theorem 1]{Harper}}]\label{HarperFact} Let $c$ and $K$ be fixed and effective constants. Then for any $\log^K F < B < F$, with $u \defeq \log F / \log B$, and $Q \leq \sqrt{\Psi\left(F,B\right)}$:
\[
\sum_{r \leq Q} \max_{\left(s,r\right) = 1} \left|\Psi\left(F,B; r,s\right) - \frac{\Psi_r\left(F,B\right)}{\phi\left(r\right)}\right| \ll \Psi\left(F,B\right) \left(e^{-\frac{cu}{\log^2 u}} + B^{-c}\right) + Q\sqrt{\Psi\left(F,B\right)} \log^{7/2} F
\]
with an implied effective constant $C = C\left(c,K\right)$.
\end{fact}
Harper also provides a Barban-Davenport-Halberstam type theorem, which we do not need but which our methods also naturally provide.
\begin{fact}[{Harper~\cite[Theorem 2]{Harper}}]
 There exist $c$ and $K$ fixed and effective constants such that for any $\log^K F < B < F$, with $u \defeq \log F / \log B$, and $Q \leq \Psi\left(F,B\right)$:
\[
\sum_{r \leq Q} \sum_{\left(s,r\right) = 1} \left|\Psi\left(F,B; r,s\right) - \frac{\Psi_r\left(F,B\right)}{\phi\left(r\right)}\right|^2 \ll \Psi\left(F,B\right)^2 \left(e^{-\frac{cu}{\log^2 u}} + B^{-c}\right) + Q\Psi\left(F,B\right)
\]
with an implied effective constant $C = C\left(c,K\right)$.
\end{fact}
In our application we will bound these quantities when the common difference $q = a-mb$ is known to be $y$-smooth: i.e. sums of
form
\[
\sum_{\substack{q\leq Q\\q\text{ is }y-smooth}}\max_{\left(a,q\right)=1}\left\vert
\Psi\left(x,y,q,a\right)-\frac{\Psi_{q}\left(x,y\right)}{\phi\left(q\right)}\right\vert .
\]
The essential difficulty is akin to that of computing the conditional expectation
$\E\left(X\mid S\right)$ for a random variable $X$ and a rare event $S$ (i.e.
$q$ is smooth). We build on these works, and use lemmas and techniques of Harper. Drappeau~\cite{Drappeau1} provides extensions of a similar flavour,
bounding weighted sums
\[
\sum_{q\leq Q}\lambda\left(q\right)\max_{\left(a,q\right)=1}\left\vert
\Psi\left(x,y,q,a\right)-\frac{\Psi_{q}\left(x,y\right)}{\phi\left(q\right)}\right\vert .
\]
with the weighting function $\lambda$ being sub-multiplicative and with $\lambda\left(q\right) \ll q^{1-\epsilon}$.
Our results require a larger range of application; to appeal to Drappeau's results directly seems to require the use of a weight $\lambda$ with $\lambda\left(q\right) \geq \1_{\{z : z\textrm{ is }y\textrm{-smooth}\}}\left(q\right) Q\psi^{-1}(Q, y)$, which is not submultiplicative.

We will state and use
Harper's ideas to derive a sharper result for the restricted sum, as
needed in our case.

\begin{remark}
Using Harper's result (Fact~\refp{HarperFact}) directly with our arguments allows one to prove that ``almost all'' moduli are in fact $B$-good and derive a weaker expected run time bound of $L_n\left(\frac{1}{3}, \O(\log \log n)\right)$.
\end{remark}


\begin{definition}\label{def-q-max} We define:
\[
Q_{\max} \defeq \max_{a,m} |am+b|  = L_n\left(\frac{2}{3}, \delta^{-1}\left(1+\o\left(1\right)\right)\right).
\]
Hereafter will reuse the variables $m$, $b$ to maintain commonality of notation with Harper.
\end{definition}
We seek to bound the probability that a $B$-smooth modulus less than $Q_{\max}$ is $B'$-bad. Naturally, we can show that this is small if we can show that the number of $B'$-bad moduli below $Q$ is much smaller than $\Psi\left(Q_{\max}, B\right)$. We can certainly achieve this if we allow $B$ to be sufficiently large, although it will increase the bounds on the expected run time which can be achieved. We state the following lemma, which we prove later.
\begin{lemma}\label{HZ-1-large}
Let $\epsilon > 0$ be fixed. Then there exist effective constants $K, c$ such that for any $\log^K x < y < x^{1/\log \log x}$, with $u \defeq \log x / \log y$, $x^\epsilon \leq Q \leq \sqrt{\Psi\left(x,y\right)}$ and $\omega = \boldsymbol{\omega}\left(1\right)$ with $\omega = y^{\O\left(1\right)}$:
\begin{equation*}
\begin{split}
\sum_{\substack{r \in [Q\omega^{-1}, Q]\\ r\textrm{ is }y\textrm{-smooth}}} \max_{\left(a,r\right) = 1} \left|\Psi\left(x,y; r,a\right) - \frac{\Psi_r\left(x,y\right)}{\phi\left(r\right)}\right| \ll \Psi\left(x,y\right) \varrho\left(Q,y\right)\left(e^{-\frac{cu}{\log^2 u}} + y^{-c}\right) 
+ Q\sqrt{\Psi\left(x,y\right)} \log^{7/2} x
\end{split}
\end{equation*}
for some effective implied constant $C$ in the $\ll$.
\end{lemma}
\begin{rem}
We first note that this does not appear to derive from Drappeau's work~\cite{Drappeau1, Drappeau2} on weighted sums of this style. In particular, Drappeau requires that the weights are multiplicative on the integers and bounded by some function tending to zero. In our case we have to cut out the small moduli and the weights do not decline; without the exclusion of the small $r < Q\omega^{-1}$, an analogue of Lemma~\refp{HZ-1-large} need not hold, as the smooth numbers become substantially more dense.
\end{rem}
\begin{rem}
Given Harper's results and a general philosophy of cancellation up to square roots, we might expect that the range of $y$ can be extended up to $x$ and the range of $Q$ decreased to $1$. We do not need the additional strength here.

The condition that $\omega = y^{\O\left(1\right)}$ ensures that $\varrho\left(Q\omega^{-1}, y\right)$ is not much larger than $\varrho\left(Q, y\right)$, in a way which will be made precise in the proof of the Lemma.
\end{rem}

\begin{proof}[Proof of Lemma~\refp{smooth-bad-dist}]
We begin by bounding the number of moduli which are $F$-bad for some $F \in [F_{\max}L_n\left(\frac{1}{3}\right)^{-1}, F_{\max}]$. We fix $\omega \defeq B'$ for concreteness. 
Observe that $\Psi\left(F, B'\right) = FL_n\left(\frac{1}{3}\right)^{-1}$. Since $L_n\left(\frac{2}{3}\right) = \boldsymbol{\omega}\left(L_n\left(\frac{1}{3}\right)\right)$:
\[
Q \leq \sqrt{\Psi\left(F,B'\right)}L_n\left(\frac{2}{3},\frac{\epsilon}{4}\right)^{-1}.
\]
Furthermore, for any $K$ fixed, $\boldsymbol{\omega}(\log^K F) = B' = \o(F^{1/\log \log F})$.
Hence we can apply Lemma~\refp{HZ-1-large}. Suppose that a modulus $r$ is $B$-smooth and also $B'$-bad for $F$. Then for some residue $a$ with $\left(a,r\right) = 1$, the contribution to the LHS of Lemma~\refp{HZ-1-large} for this $r$ is at least:
\[
\left(1+\o\left(1\right)\right)\frac{\Psi_r\left(F, B'\right)}{\phi\left(r\right)} = \frac{\Psi\left(F, B'\right)}{r} \geq \frac{\Psi\left(F,B'\right)}{Q}
\]
where for the first equality we use Corollary~\refp{psi-q-psi}, noting that $B \leq B'$ so $r$ is $B'$-smooth, $u < \log \log n$ and the number of divisors of $r$ is bounded by $\log r$ so that the multiplicative error is $1+\o(1)$. Now:
\begin{multline*}
\sum_{\substack{r \in [Q_{\max}\omega^{-1}, Q_{\max}]\\ r\textrm{ is }y\textrm{-smooth}}} \max_{\left(a,r\right) = 1}
\left|\Psi\left(F,B'; r,a\right) - \frac{\Psi_r\left(F,B'\right)}{\phi\left(r\right)}\right| \\
\begin{aligned}
&\leq C \Psi\left(F,B'\right)\varrho\left(Q_{\max}, B'\right) \left(e^{-\frac{cu'}{\log^2 u'}} + B'^{-c}\right)
 + Q_{\max}\sqrt{\Psi\left(F,B'\right)} \log^{7/2} F \\
&= C F\varrho\left(F,B'\right)\varrho\left(Q_{\max}, B'\right) \left(e^{-\frac{cu'}{\log^2 u'}} + B'^{-c}\right)
 + Q_{\max}F^{1/2}\varrho\left(Fc,B'\right)^{1/2} \log^{7/2} F\end{aligned}
\end{multline*}
First, we observe that $F$ and $Q_{\max}$ are $L_n\left(\frac{2}{3}\right)$, whilst $B'$ is $L_n\left(\frac{1}{3}\right)$. Hence both densities $\varrho\left(Q_{\max}, B'\right)$ and $\varrho\left(F, B'\right)$ are $L_n\left(\frac{1}{3}\right)^{-1}$. From Definitions~\refp{def-q-max} and~\refp{def-F-max} we have
\[
Q_{\max} = L_n\left(\frac{2}{3}, \delta^{-1}(1+\o(1))\right), \quad
F = L_n\left(\frac{2}{3}, (\kappa + \sigma\delta)(1+\o(1))\right),
\]
and from Equation~\refp{d-bound} $2\delta^{-1} < \kappa + \sigma\delta$. Hence $FQ_{\max}^{-2} = L_n\left(\frac{2}{3}\right)$. Since, up to order $L_n\left(\frac{2}{3}\right)^{\o(1)}$, the first term is $F$ and the second is $Q_{\max}F^{1/2}$, we deduce that the first term dominates the second. If $r$ is $B'$-bad for $F$ it contributes at least
$
\left(1+\o\left(1\right)\right) \Psi\left(F, B'\right)Q_{\max}^{-1}
$
to the sum on the left-hand side.

Hence the number of moduli which are in $ [Q_{\max}\omega^{-1}, Q_{\max}]$, are $B$-smooth and $B'$-bad for $F$ is at most:
\begin{align*}
&\left(C+\o\left(1\right)\right)\frac{Q_{\max}}{\Psi\left(F, B'\right)} \Psi\left(F,B'\right)\varrho\left(Q_{\max}, B'\right) \left(e^{-\frac{cu'}{\log^2 u'}} + B'^{-c}\right) 
=\left(C+\o\left(1\right)\right)\Psi\left(Q_{\max}, B'\right) \left(e^{-\frac{cu'}{\log^2 u'}} + B'^{-c}\right) \\
\end{align*}
If a modulus is $B'$-bad near $F_{\max}$, it must be $B'$-bad for some
\[
F \in \left\{F_{\max}L_n\left(\frac{1}{3}\right)^{-1}, F_{\max}\right\}\cup \left\{2^i : 2^i \in \left[F_{\max}L_n\left(\frac{1}{3}\right)^{-1}, F_{\max}\right]\right\},
\]
a set of logarithmic size. We can absorb a logarithmic factor into the constants $c, C$, so the number of moduli which are in $ [Q_{\max}\omega^{-1}, Q_{\max}]$, are $B$-smooth and $B'$-bad is at most:
\[
\left(C+\o\left(1\right)\right)\Psi\left(Q_{\max}, B'\right) \left(e^{-\frac{cu'}{2\log^2 u'}} + B'^{-\frac{c}{2}}\right) = \o\left(\Psi\left(Q_{\max}, B'\right)\right)
\]
Hence even assuming that every $B$-smooth number below $Q_{\max}\omega^{-1}$ is $B'$-bad:
\begin{align*}
\P_{a,m}\left(a-mb\textrm{ is }B'\textrm{-good} \mid a-mb\textrm{ is }B\textrm{-smooth}\right)
\geq 1 - \frac{\Psi\left(Q_{\max} \omega^{-1}, B'\right) + \o\left(\Psi\left(Q_{\max}, B'\right)\right)}{\Psi\left(Q_{\max}, B'\right)}
= 1 - \o\left(1\right).
\end{align*}
\end{proof}

It remains to prove Lemma~\refp{HZ-1-large}. We follow the proof strategy and notation of similar results by Harper~\cite{Harper}. At a high-level, we express the sum on the LHS as a sum over characters $\chi_r$ of modulus $r$, and then write this in terms of primitive characters $\chi^*_r$ of conductor $q$. We split the primitive characters into three sets based on the size of their conductor, and for small and intermediate sized conductors we have to separately deal with characters whose $\L$ functions have zeros near 1. For more details of this general strategy see~\cite{IwaniecKowalski}.

Here, our primary extension over previous work is that that the modulus of each character is $y$-smooth. We are also able to restrict the range of summation to comparatively large moduli. If instead we were to consider every $y$-smooth number less than $Q$, we would not be able to substantially reduce the contribution of characters with moduli very small by comparison to $y$. That these conditions are useful in practice and are tractable on the analytic side suggests a large collection of potentially fruitful new results, restricting sums of this type to a very sparse set instead of an interval.

We study only moduli which are at least $Q\omega^{-1}$, where $\omega$ is not too large with respect to $y$. In particular, this ensures that the set of moduli we sum over is always sparse restricted to any reasonably large subinterval of $[Q\omega^{-1}, Q]$. This allows us to give a substantially stronger bound as the sum is reduced by at least a factor of $\varrho\left(Q,y\right)$, modulo a slight reduction in the constant $c$.

\begin{proof}[Proof of Lemma~\refp{HZ-1-large}]

We will fix $c, K$ depending on $\eta > 0$, with $c$ small and $K$ large. We will fix $\eta$ to be small in terms of a constant $b$ to be determined.
We set 
\begin{equation}\label{m-M-defn}
m \defeq \min\left(y^\eta, e^{\eta \sqrt{\log x}}\right),\quad M \defeq x^\eta.
\end{equation}
The following five Facts are due to Harper~\cite{Harper}, and concern the size of character sums over smooth numbers and density estimates using characters whose $\L$ functions have roots with real part near 1 and small imaginary part.

We recall that the \emph{conductor} $q$ of a Dirichlet Character $\chi_r$ of modulus $r$ is the least $q > 0$ such that $\chi_r(x) = \chi_r(x+q)$ for all $x$. As an immediate corollary, $q \mid r$ and so if the $r$ is $y$-smooth then so are $q$ and $rq^{-1}$. We also recall the \emph{saddlepoint} $\alpha$ of Fact~\refp{HT-smooths}.
\begin{fact}[{\cite[Theorem 3]{Harper}}]\label{character-bound}
There exist constants $b,B > 0$, 
such that if $\log^B x \leq y \leq x$ and $\chi_q$ is a non-principal Dirichlet character with conductor $r \defeq \operatorname{cond}\left(\chi_q\right) \leq y^b$ and modulus $q \leq x$, with the largest real zero $\beta = \beta_{\chi_q}$ of $\L\left(s, \chi_q\right)$ is $\leq 1 - B/\log y$, then:
\[
|\Psi\left(x,y;\chi_q\right)| \ll \Psi\left(x,y\right)\sqrt{\log x \log y}\left(e^{-\left(b\log x\right) \min\left(\left(\log r\right)^{-1}, 1-\beta\right)}\log \log x + e^{-b\sqrt{\log x}} + y^{-b}\right)\qed
\]
\end{fact}
\begin{fact}[{\cite[Proposition 1]{Harper}}]\label{zero-free-character}
There exist constants $d, C > 0$ such that for $\log^{1.1} x \leq y \leq x$
, and $\chi_q$ a non-principal Dirichlet character with conductor $r \defeq \operatorname{cond}\left(\chi_q\right) \leq x^d$ and modulus $q \leq x$, with $\L\left(z,\chi_q\right)$ having no zeros for
$\Re\left(z\right) > 1-\epsilon, |\Im\left(z\right)| < H$, with
\[
C \left(\log y\right)^{-1} < \epsilon \leq \alpha\left(x,y\right) / 2,\qquad y^{0.9}\log^2 x \leq H \leq x^d,
\]
and either
\[
y \geq \left(Hr\right)^C \quad\textrm{or}\quad \epsilon \geq 40\log \log\left(qyH\right) \left(\log y\right)^{-1}
\]
then:
\[
|\Psi\left(x,y;\chi\right)| \ll \Psi\left(x,y\right)\sqrt{\log x \log y}\left(x^{-0.3\epsilon}\log H + H^{-0.02}\right) \qed
\]
\end{fact}
\begin{fact}[\cite{IwaniecKowalski}, with {\cite[pp. 15]{Harper}} giving the precise form]\label{zero-density}
\[
\sum_{R < r \leq 2R}\sum_{\substack{\chi^*_r\imod{r}\\ \L\left(z, \chi^*_r\right) = 0\textrm{ for some}\\ \Re\left(z\right) > \frac{299}{300}, |\Im\left(z\right)| \leq r^{100}}} \frac{1}{\phi\left(r\right)} \ll R^{-1/10} \qed
\]
\end{fact}
\begin{fact}[{\cite[Proposition 2]{Harper}}]\label{large-characters}
For any $0 < \eta < 1/80$,
$y \leq x^{9/10}$, and $x^\eta \leq Q \leq \sqrt{x}$:
\begin{equation*}
\sum_{M \leq r \leq Q}
\sum_{\chi_r^*}
\sum_{s \in [M, Q]}\frac{1}{\phi\left(s\right)}
\sum_{\substack{\chi_s \\ \chi^*_r\textrm{ induces }\chi_s}} |\Psi\left(x,y;\chi_s\right)|  \ll \log^{7/2}x \sqrt{\Psi\left(x,y\right)}\left(Q + x^{1/2 - \eta}\log^2 x\right)\qed
\end{equation*}
\end{fact}
\begin{fact}[{\cite[pp. 16]{Harper}}]\label{Harper-Page}
For any real and non-principal character $\chi_q$ of modulus at most $Q$ and conductor at most $y^\eta$:
\begin{equation*}
|\Psi\left(x,y;\chi_q\right)| \ll \Psi\left(x,y\right)\sqrt{\log x \log y} \left(\log y\exp\left(-\O\left(\frac{u}{\log^2\left(u+1\right)}\right)\right) + \O\left(y^{-0.02}\right)\right)\qed
\end{equation*}
\end{fact}

The first step of the argument is to note that by the orthogonality of characters we can write the $r$-periodic function $\1_{x \equiv a \imod r}$ for $\left(a,r\right) = 1$ as:
\[
\frac{1}{\phi\left(r\right)}\sum_{\chi_r} \chi_r\left(x\right) \chi_r\left(a\right)^{-1}
\]
Note that the contribution of the principal character $\chi_0$ to the above formula is exactly $\phi^{-1}\left(r\right)$, and so:
\[
\Psi\left(x,y; r,a\right) - \frac{\Psi_r\left(x,y\right)}{\phi\left(r\right)} =
\frac{1}{\phi\left(r\right)}\sum_{\substack{\chi_r \imod{r} \\ \chi_r \neq \chi_0}} \Psi\left(x, y; \chi_r\right) \chi_r\left(a\right)^{-1}
\]

Taking the modulus of the left-hand side, and noting that for any $\left(a,r\right) = 1$ and any $\chi_r$ with $|\chi_r\left(a\right)| = 1$:
\begin{equation}\label{eq:conj1-large-sum}
\sum_{\substack{r \in [Q\omega^{-1}, Q]\\ r\textrm{ is }y\textrm{-smooth}}}
\max_{\left(a,r\right) = 1} \left|\Psi\left(x,y; r,a\right) - \frac{\Psi_r\left(x,y\right)}{\phi\left(r\right)}\right|
 \leq \sum_{\substack{r \in [Q\omega^{-1}, Q]\\ r\textrm{ is }y\textrm{-smooth}}} \frac{1}{\phi\left(r\right)} \sum_{\substack{\chi_r \imod{r} \\ \chi_r \neq \chi_0}}|\Psi\left(x,y;\chi_r\right)| \eqdef \mathcal{W}
\end{equation}
We split $\mathcal{W}$ into contributions from characters of \emph{small} conductor $r < m$, of \emph{medium} conductor $m \leq r < M$, or with \emph{large} conductor $M \leq r$. In the first two cases, we additionally split the characters between a small number of exceptional characters whose $\mathcal{L}$ functions have zeros near $1$, and the generic case where the $\mathcal{L}$ function has no such zero. Let:
\begin{align*}
\mathcal{G}_1 &\defeq \bigcup_{1 < r \leq m}\left\{ \chi^*_r\imod{r} : \L\left(z, \chi^*_r\right) \neq 0\textrm{ for }z\in \R, z > 1 - \frac{\eta B}{\log m}\right\}, \\
\mathcal{G}_2 &\defeq \bigcup_{m < r \leq M}\left\{ \chi^*_r\imod{r} : \L\left(z, \chi^*_r\right) \neq 0\textrm{ for }\Re\left(z\right) > \frac{299}{300}, |\Im\left(z\right)| \leq r^{100}\right\}.
\end{align*}
We will control the contribution of characters in $\mathcal{G}_1$ and $\mathcal{G}_2$ with Facts~\refp{character-bound} and~\refp{zero-free-character} respectively. The number of characters of small conductor which are not in $\mathcal{G}_1$ is controlled by Page's theorem, and their contribution is bounded trivially. The contribution of characters of medium conductor which are not in $\mathcal{G}_2$ is controlled via Fact~\refp{zero-density}, and those with large modulus by Fact~\refp{large-characters}.

\begin{rem}\label{G1-loses-term} Suppose $\chi_r \in \mathcal{G}_1$ with largest real root of $\L\left(z, \chi_r\right)$ at $\beta$. Then $1 - \beta > \frac{\eta B}{\log m}$ and $\log r \leq \log m \leq \eta \sqrt{\log x}$. In particular:
\[
\left(b\log x\right) \min\left(\left(\log r\right)^{-1}, 1-\beta\right) > \min\left(\eta^{-1}, B\right)b\sqrt{\log x},
\]
with $\eta$ taken to be small and $B$ large. Hence if Fact~\refp{character-bound} is applied the exponent in the first error term can be taken to be much lower than $-b\sqrt{\log x}$, since $b$ is small, and hence the second term dominates the first.
\end{rem}

For our application we require a few ancillary claims:
\begin{claim}\label{smooth-inverse-sum} For any $S$ exceeding an absolute constant $S_0$, and any $\omega$ such that $\log \omega \leq \frac{1}{2}\log^2 S - \log S - \frac{3}{2}$:
\[
\sum_{\substack{S < s < S\omega \\ s\textrm{ is }y\textrm{-smooth}}} \frac{1}{\phi\left(s\right)} \leq 4\varrho\left(S, y\right)\log \omega \log \log S.
\]
\end{claim}
\begin{rem} In applications, we have $\log \omega = \o(\log^2 S)$, which plainly suffices.
\end{rem}
\begin{proof}
We obtain a uniform lower bound on $\phi(x) x^{-1}$ for $x \leq S\omega$. Note that
$
\phi(x)x^{-1} = \prod_{p \mid x, p \textrm{ prime}}\left(1 - p^{-1}\right).
$
Then any value of $\phi(x)x^{-1}$ attained for $x \leq S\omega$ is attained for a square-free $x$ in the same range. Suppose there are primes $p < p'$ such that $p \nmid x$ and $p' \mid x$. Let $x' = xp{p'}^{-1} < x$. Then
\[
\frac{\phi(x')}{x'} = \frac{\phi(x)}{x}\left(1-\frac{1}{p}\right)\left(1-\frac{1}{p'}\right)^{-1}
= \frac{\phi(x)}{x}\left(1-\frac{p^{-1}-{p'}^{-1}}{1-{p'}^{-1}}\right) < \frac{\phi(x)}{x}.
\]
Immediately, we deduce that for $x \leq S\omega$ the minimal value of $\phi(x)x^{-1}$ is obtained for $x = \prod_{p \leq k, p \textrm{ prime}} p$ for some prime $k$. For such an $x$,
\[
\frac{\phi(x)}{x} = \prod_{p \leq k, p \textrm{ prime}}(1 - p^{-1}),
\]
which is a decreasing function of $k$. Hence the minimal value is obtained for $k$ maximal such that $x \leq S\omega$.

For all $k \geq 2$ and $x$ the product of the primes below $k$, we have~\cite[Theorems 4 and 7]{Rosser}:
\[
\log x = \sum_{p \leq k, p\textrm{ prime}} \log p > \frac{k}{2} - 1,\quad \frac{\phi(x)}{x} = \prod_{p \leq k, p\textrm{ prime}} \left(1-\frac{1}{p}\right) \geq \frac{1}{2\log (k+2)}
\]
Note that $k \leq 2\log x + 1$, and $\log x \leq \log S + \log \omega \leq \frac{1}{2}\log^2 S - \frac{3}{2}$. Then:
\[
\frac{\phi\left(x\right)}{x} \geq \frac{1}{2\log (2 \log x + 3)} \geq \frac{1}{4\log \log S}.
\]
Hence for all $s \leq S\omega$, $\phi\left(s\right) \geq s\left(4\log \log S\right)^{-1}$, and so:
\begin{align*}
\sum_{\substack{S < s < S\omega \\ s\textrm{ is }y\textrm{-smooth}}} \frac{1}{\phi\left(s\right)}
&\leq 4\log \log S \sum_{\substack{S \leq s \leq S\omega \\ s\textrm{ is }y\textrm{-smooth}}} s^{-1}
\leq 4\log \log S \sum_{i=0}^{\lceil\log_2 \omega\rceil}\sum_{\substack{S2^i \leq s \leq S2^{i+1} \\ s\textrm{ is }y\textrm{-smooth}}} s^{-1}
\\
&\leq 4\log \log S \sum_{i=0}^{\lceil\log_2 \omega\rceil}\frac{\Psi\left(S2^{i+1},y\right) - \Psi\left(S2^i, y\right)}{S2^i}\\
&\leq 4\log \log S \sum_{i=0}^{\lceil\log_2 \omega\rceil} \varrho\left(S2^i, y\right)
\leq 4\varrho\left(S, y\right)\log \omega \log \log S
\end{align*}
To show the last two inequalities, we note that $\Psi\left(2x, y\right) \leq 2\Psi\left(x, y\right)$, a result of Hildebrand~\cite[Theorem 4]{Hildebrand}. Hence $\Psi\left(S2^{i+1},y\right) - \Psi\left(S2^i, y\right) \leq \Psi\left(S2^i, y\right)$ which yields the first inequality; from $\varrho\left(2x, y\right) \leq \varrho\left(x, y\right)$, we obtain $\varrho\left(S2^i, y\right) \leq \varrho\left(S, y\right)$ as required for the second inequality.
\end{proof}

\begin{claim}\label{smooth-density-smoothness} Suppose that $u = \frac{\log x}{\log y} \rightarrow \infty$. Then for any $c \geq 0$.
$
\varrho\left(x, y\right) = \varrho\left(xy^c, y\right) (\log x)^{\O(\lceil c\rceil)}.
$
\end{claim}
\begin{proof}
From Fact~\refp{HT-smooths}, for any $1 \leq v \leq y$:
\begin{align*}
\Psi\left(vx, y\right) &= \Psi\left(x,y\right)v^{\alpha\left(x,y\right)}\left(1+\O\left(u^{-1}+y^{-1}\log y\right)\right), \text{ where} \\
\alpha\left(x,y\right) &= \frac{\log\left(\frac{y}{\log x} + 1\right)}{\log y} \left(1+ \O\left(\frac{\log \log \left(y+1\right)}{\log y}\right)\right) \end{align*}
As a corollary, for any $0 \leq z \leq 1$:
\begin{align*}
\frac{\varrho\left(xy^z,y\right)}{\varrho\left(x,y\right)} &= \frac{\Psi\left(xy^z,y\right)}{y^z\Psi\left(x,y\right)} 
= \frac{1}{y^z}\left(\frac{y}{\log x} + 1\right)^{z\left(1+\O\left(\frac{\log \log\left(y+1\right)}{\log y}\right)\right)}
\left(1 + \O\left(\frac{1}{u} + \frac{\log y}{y}\right)\right)
\end{align*}
Note that $\log \left(\frac{y}{\log x} + 1\right) = \O(\log y)$, $\left(\log x\right)^{-1} + y^{-1} = (1+\o(1))\left(\log x\right)^{-1}$ and $\log y = u^{-1} \log x$. Hence:
\begin{align*}
\frac{\varrho\left(xy^z,y\right)}{\varrho\left(x,y\right)} &= \left(\frac{1}{\log x} + \frac{1}{y}\right)^z \exp\left(\O\left(\log \log\left(y+1\right)\right)\right)
\left(1 + \O\left(\frac{1}{u} + \frac{\log y}{y}\right)\right) \\
&= \frac{\log^{\O\left(1\right)}y}{\log^z x}
\left(1 + \frac{1}{u}\O\left(1 + \frac{\log x}{y}\right)\right) 
= \log^{-z+\O\left(1\right)}x
\end{align*}
Since $\log\left(xy^i\right) = \left(1+i/u\right)\log x = \log^{1+\o_u\left(1\right)}x$ for all $0 \leq i \leq c$, we can apply the above bound $\lceil c \rceil$ times with $z = c \lceil c\rceil^{-1}$ to obtain the claimed bound.
\end{proof}
We first bound the contribution to $\mathcal{W}$ 
from characters in $\mathcal{G}_1$ via Fact~\refp{character-bound} and Remark~\refp{G1-loses-term}:
\begin{align*}
&\sum_{\chi^* \in \mathcal{G}_1}\sum_{\substack{q \in [Q\omega^{-1}, Q]\\ q\textrm{ is }y\textrm{-smooth}}}\frac{1}{\phi\left(q\right)}
\sum_{\substack{\chi_q \\ \chi^*\textrm{ induces }\chi_q}} |\Psi\left(x,y;\chi_q\right)|
\\
&\ll \sum_{\chi^* \in \mathcal{G}_1}\sum_{\substack{q \in [Q\omega^{-1}, Q]\\ q\textrm{ is }y\textrm{-smooth}}} \frac{1}{\phi\left(q\right)}
\sum_{\substack{\chi_q \\ \chi^*\textrm{ induces }\chi_q}}
\Psi\left(x,y\right) \sqrt{\log x \log y}\left(e^{-b\sqrt{\log x}} + y^{-b}\right),
\end{align*}
We write the smooth modulus as $q = s r$ for $r = \operatorname{cond}\left(\chi^*\right)$. Using the fact that $
\phi\left(rs\right) \geq \phi\left(r\right)\phi\left(s\right) \forall r, s$, and that the number of primitive characters of modulus $r$ is at most $\phi\left(r\right)$, the above is:
\begin{align*}
&= \Psi\left(x,y\right) \sqrt{\log x \log y}\left(e^{-b\sqrt{\log x}} + y^{-b}\right)
\sum_{\substack{r < m\\r\textrm{ is }y\textrm{-smooth}}}
\sum_{\chi_r^* \in \mathcal{G}_1}
\sum_{\substack{s \in \left[\frac{Q}{\omega r}, \frac{Q}{r}\right]\\ s\textrm{ is }y\textrm{-smooth}}} \frac{1}{\phi\left(rs\right)} \\
&\leq \Psi\left(x,y\right) \sqrt{\log x \log y}\left(e^{-b\sqrt{\log x}} + y^{-b}\right)
\sum_{\substack{r < m\\r\textrm{ is }y\textrm{-smooth}}}
\sum_{\chi_r^* \in \mathcal{G}_1} \frac{1}{\phi\left(r\right)}
\sum_{\substack{s \in \left[\frac{Q}{\omega r}, \frac{Q}{r}\right]\\ s\textrm{ is }y\textrm{-smooth}}} \frac{1}{\phi\left(s\right)}
\\
&\leq \Psi\left(x,y\right) \sqrt{\log x \log y}\left(e^{-b\sqrt{\log x}} + y^{-b}\right)
\sum_{\substack{r < m\\r\textrm{ is }y\textrm{-smooth}}} 1
\sum_{\substack{s \in \left[\frac{Q}{\omega r}, \frac{Q}{r}\right]\\ s\textrm{ is }y\textrm{-smooth}}} \frac{1}{\phi\left(s\right)}
\end{align*}
Note that $\omega r = y^{\O(1)}$. We now use Claims~\refp{smooth-inverse-sum} and~\refp{smooth-density-smoothness}to bound the above as
\begin{align*}
&\leq \Psi\left(x,y\right) \sqrt{\log x \log y}\left(e^{-b\sqrt{\log x}} + y^{-b}\right)
\sum_{\substack{r < m\\r\textrm{ is }y\textrm{-smooth}}}4\varrho\left(\frac{Q}{\omega r}, y\right)\log\omega \log \log Q
\\
&\leq \Psi\left(x,y\right) \left(e^{-b\sqrt{\log x}} + y^{-b}\right)
\varrho\left(Q, y\right) \left[4m\sqrt{\log x \log y}\log\omega \log \log Q \log^{\O\left(1\right)}x\right]
\end{align*}
From Equation~\refp{m-M-defn}, $m \leq y^{\eta}$, and we can ensure $\eta < \frac{b}{4}$. So we obtain:
\[
\ll \Psi\left(x,y\right) \left(e^{-\frac{b}{2}\sqrt{\log x}} + y^{-\frac{b}{2}}\right)\varrho\left(Q, y\right).
\]
This suffices for Lemma~\refp{HZ-1-large}, as $u \log y (\log u)^{-2} = \log x (\log u)^{-2}$ so:
\begin{equation}\label{yu-is-small}
\min\left(\log y, u \log^{-2} u\right) = \o\left(\sqrt{\log x}\right).
\end{equation}
Hence the first term can be absorbed into $c$; the second can be absorbed if we choose $c < \frac{b}{2}$.

We now bound the contribution to $\mathcal{W}$ from characters in $\mathcal{G}_2$. Recall that these characters have modulus in $(m, M]$. We take $\eta$ small enough that $M^2 \log M < x^{1/1000}$. Set:
\[
\epsilon \defeq \min\left\{\frac{1}{300}, \frac{10 \log r}{\log y}\right\}\quad\textrm{and}\quad H\defeq r^{100}
\]
Proceeding similarly and using Fact~\refp{zero-free-character}:
\begin{align*}
&\sum_{\chi^* \in \mathcal{G}_2}\sum_{\substack{q \in [Q\omega^{-1}, Q]\\ q\textrm{ is }y\textrm{-smooth}}}\frac{1}{\phi\left(q\right)}
\sum_{\substack{\chi_q \\ \chi^*\textrm{ induces }\chi_q}} |\Psi\left(x,y;\chi_q\right)|
\\
&\leq \sum_{\substack{m \leq r < M\\ r\textrm{ is }y\textrm{-smooth}}}
\sum_{\chi^*_r \in \mathcal{G}_2}\frac{1}{\phi\left(r\right)}
\sum_{\substack{s \in \left[\frac{Q}{\omega r}, \frac{Q}{r}\right]\\ s\textrm{ is }y\textrm{-smooth}}}\frac{1}{\phi\left(s\right)}
\sum_{\substack{\chi_{rs} \\ \chi^*_r\textrm{ induces }\chi_{rs}}} |\Psi\left(x,y;\chi_{rs}\right)|
\\
&\ll \sum_{\substack{m \leq r < M\\ r\textrm{ is }y\textrm{-smooth}}}
1
\sum_{\substack{s \in \left[\frac{Q}{\omega r}, \frac{Q}{r}\right]\\ s\textrm{ is }y\textrm{-smooth}}}\frac{1}{\phi\left(s\right)}
\Psi\left(x,y\right) \sqrt{\log x \log y}\left(\frac{\log r}{x^{0.001}} + r^{-2}\right)
\end{align*}
Recalling that $M^2 \log M \leq x^{0.001}$ and using Claim~\refp{smooth-inverse-sum} we obtain:
\begin{align*}
&\ll \Psi\left(x,y\right) \sqrt{\log x \log y}
\sum_{\substack{m \leq r < M\\ r\textrm{ is }y\textrm{-smooth}}}
\sum_{\substack{s \in \left[\frac{Q}{\omega r}, \frac{Q}{r}\right]\\ s\textrm{ is }y\textrm{-smooth}}}\frac{1}{r^2\phi\left(s\right)}
\\
&\leq \Psi\left(x,y\right) \sqrt{\log x \log y}
\sum_{\substack{m \leq r < M\\ r\textrm{ is }y\textrm{-smooth}}}
\frac{4\log \omega  \log \log Q}{r^2} \varrho\left(\frac{Q}{\omega r}, y\right)
\end{align*}
Note that $\varrho\left(\frac{Q}{\omega r}, y\right)r^{-1} = \frac{\omega}{Q} \Psi\left(\frac{Q}{\omega r}, y\right)$, and so is decreasing in $r$. Hence by Claim~\refp{smooth-density-smoothness}, the above is:
\begin{align*}
&\leq\Psi\left(x,y\right)\sqrt{\log x \log y}\;4\log \omega \log \log Q \varrho\left(\frac{Q}{\omega m}, y\right)m^{-1}
\sum_{\substack{m \leq r < M\\ r\textrm{ is }y\textrm{-smooth}}}r^{-1}
\\
&\leq\Psi\left(x,y\right)m^{-1} \varrho\left(Q, y\right) \left[4\log^{\O\left(1\right)}x \sqrt{\log x \log y} \log \omega \log \log Q \log M\right]\\
&\leq\Psi\left(x,y\right)\left(e^{-\eta\sqrt{\log x}/2} + y^{-\eta / 2}\right)\varrho\left(Q, y\right) \\
\end{align*}
as from Equation~\refp{m-M-defn}, $m^{1/2}$ dominates the logarithmic terms. Again by Equation~\refp{yu-is-small} this suffices for Lemma~\refp{HZ-1-large}.

We now bound the contribution to $\mathcal{W}$ 
from characters of small conductor which are not in $\mathcal{G}_1$.
By Page's Theorem~\cite[Lemma 8]{Page}, there is at most one character with conductor $< m$ and not in $\mathcal{G}_1$, if $\eta$ is chosen small enough in terms of $B$. Furthermore, such a character must be real. If such a character $\chi^*_{e}$ exists with conductor $r_e$, we proceed similarly and bound its contribution as:
\begin{align*}
&\sum_{\substack{r_e \mid q \in \left[Q\omega^{-1}, Q\right]\\ q\textrm{ is }y\textrm{-smooth}}}\frac{1}{\phi\left(q\right)}
\sum_{\substack{\chi_q \\ \chi^*_e\textrm{ induces }\chi_q}} |\Psi\left(x,y;\chi_q\right)| 
\ll \frac{1}{\phi\left(r_e\right)}
\sum_{\substack{s \in \left[\frac{Q}{\omega r_e}, \frac{Q}{r_e}\right]\\ s\textrm{ is }y\textrm{-smooth}}}\frac{1}{\phi\left(s\right)}
\max_{\substack{\chi_q : q < m\\ \chi^*_e\textrm{ induces }\chi_q}} |\Psi\left(x,y; \chi_q\right)| \\
&\ll \frac{4}{\phi\left(r_e\right)}\log \omega \log \log Q (\log x)^{\O\left(1\right)}\varrho
\left(Q, y\right)
\max_{\substack{\chi_q : q < m\\ \chi^*_e\textrm{ induces }\chi_q}} |\Psi\left(x,y; \chi_q\right)|
\end{align*}
Fact~\refp{Harper-Page} bounds all of these $|\Psi\left(x,y;\chi_q\right)|$,
and we can absorb logarithmic terms into the constant in $\exp(\O(u \log^{-2} u))$. Hence the contribution of characters lying over $\chi_e^{*}$ is:
\[
\ll \Psi\left(x,y\right)\varrho\left(Q, y\right)\left(\exp\left(-\O\left(\frac{u}{\log^2\left(u+1\right)}\right)\right) + y^{-\O\left(1\right)}\right)
\]
which suffices for Lemma~\refp{HZ-1-large}.

We now bound the contribution to $\mathcal{W}$ from characters of medium conductor which are not in $\mathcal{G}_2$. Similarly we get:
\begin{align*}
&\sum_{\substack{\chi^* \notin \mathcal{G}_2 \\ \operatorname{cond}\left(\chi^*\right)\in \left[m,M\right)}}\sum_{\substack{q \in [Q\omega^{-1}, Q]\\ q\textrm{ is }y\textrm{-smooth}}}\frac{1}{\phi\left(q\right)}
\sum_{\substack{\chi_q \\ \chi^*\textrm{ induces }\chi_q}} |\Psi\left(x,y;\chi_q\right)|
\\
&\ll\sum_{\substack{m < r \leq M\\ r\textrm{ is }y\textrm{-smooth}}}
\sum_{\chi^*_r \notin \mathcal{G}_2}\frac{1}{\phi\left(r\right)}
\sum_{\substack{s \in \left[\frac{Q}{\omega r}, \frac{Q}{r}\right]\\ s\textrm{ is }y\textrm{-smooth}}}\frac{1}{\phi\left(s\right)}
|\Psi\left(x,y, \chi_r\right)|
\end{align*}
Using the trivial bound $|\Psi\left(x,y, \chi_r\right)| \leq \Psi\left(x,y\right)$ and Claim~\refp{smooth-inverse-sum} the above is:
\begin{align*}
&\ll \Psi\left(x,y\right)\log \omega \log \log Q
\sum_{\substack{m < r \leq M\\ r\textrm{ is }y\textrm{-smooth}}}
\sum_{\chi^*_r \notin \mathcal{G}_2}\frac{1}{\phi\left(r\right)} \varrho\left(\frac{Q}{\omega r}, y\right) \\
&\ll \Psi\left(x,y\right) \log \omega \log \log Q
\sum_{i=0}^{\lfloor\log_2\left(M / m\right)\rfloor}
\sum_{\substack{m2^i < r \leq m2^{i+1}\\ r\textrm{ is }y\textrm{-smooth}}}
\sum_{\chi^*_r \notin \mathcal{G}_2}\frac{1}{\phi\left(r\right)} \varrho\left(\frac{Q}{\omega r}, y\right)
\end{align*}
Note that $Q/\omega M = y^{\boldsymbol{\omega}\left(1\right)}$, so in Fact~\refp{HT-smooths} the saddlepoint $\alpha \rightarrow 0$ and hence $\varrho\left(\frac{Q}{\omega r}, y\right)r^{-1/10}$ decreases when $r$ is doubled. Hence using Fact~\refp{zero-density} and Claim~\refp{smooth-density-smoothness} the above is:
\begin{align*}
&\ll \Psi\left(x,y\right) \log \omega \log \log Q
\sum_{i=0}^{\lfloor\log_2\left(M / m\right)\rfloor}
\varrho\left(\frac{Q}{\omega m2^i}, y\right)\left(m2^i\right)^{-1/10} \\
&\ll \Psi\left(x,y\right) \log \omega \log \log Q \log M
m^{-1/10} \varrho\left(\frac{Q}{\omega m}, y\right) \\
&\ll \Psi\left(x,y\right)\varrho\left(Q, y\right) \log^{\O\left(1\right)}x \log \omega \log \log Q \log M
m^{-1/10}\\
&\ll \Psi\left(x,y\right) \varrho\left(Q, y\right) \left(e^{-\eta\sqrt{\log x}/20} + y^{-\eta / 20}\right)
\end{align*}
as we absorb the logarithmic terms into $m^{1/20}$; this suffices for Lemma~\refp{HZ-1-large}.

We bound the contribution to $\mathcal{W}$ from large modulus characters using Fact~\refp{large-characters}.
Now $\varrho(x, \log^a x) = x^{-1/a + \o(1)}$ for any constant $a$ \cite[Corollary 7.9]{MontgomeryVaughan} and $y > \log^K x$, and so
\[
\sqrt{\varrho(x, y)}\varrho(Q, y) \geq x^{-\frac{3}{2K} + \o(1)}
\]
Hence if we set $K > 3\eta$ we can absorb all the logarithmic terms to bound the contribution of large modulus characters as
\[
\ll \log^{7/2}x \sqrt{\Psi\left(x,y\right)}Q  + \Psi(x,y)\varrho(Q, y)x^{-\eta}
\]
which suffices for Lemma~\refp{HZ-1-large} as $x^{-\eta} = y^{-\o(1)}$.
\end{proof}

\section*{Acknowledgements}
We thank Enrico Bombieri, Sary Drappeau, Andrew Granville, Adam Harper, Kumar Murty and Kannan Soundararjan for their technical suggestions and discussions. We thank Paul Balister and Rob Morris for their extensive comments and suggestions.

\bibliographystyle{plain}
\bibliography{NFS}

\end{document}